\documentclass[11pt,reqno]{amsart}

\usepackage{tgpagella}

\usepackage[framemethod=TikZ]{mdframed}
\usepackage{xcolor}
\usepackage{epstopdf}

 \usepackage{framed}
 \usepackage{framed}
\renewmdenv[
    backgroundcolor=blue!1,
    linecolor=blue!40,
    outerlinewidth=1pt,
    roundcorner=1mm,
    skipabove=0mm,
    skipbelow=0mm,
]{boxed}

\usepackage{geometry}                
\usepackage{amsmath,amssymb,amsthm}

\usepackage{amsmath,amssymb,amsthm}
 \usepackage[authoryear]{natbib}
\usepackage{graphicx}
\usepackage{comment}

\theoremstyle{plain}
\newtheorem{theorem}{Theorem}[section]

\newtheorem{lemma}{Lemma}[section]
\newtheorem{corollary}{Corollary}[section]
\newtheorem{proposition}{Proposition}[section]
\theoremstyle{definition}

\newtheorem{remark}{Remark}[section]
\newtheorem{example}{Example}

\DeclareMathOperator{\diag}{diag}

\newcommand{\Gn}{{\mathbb{G}_n}}

\newcommand{\R}{\mathbb{R}}
\newcommand{\RR}{\mathbb{R}}

\renewcommand{\qed}{\hfill{\tiny \ensuremath{\blacksquare} }}%

\newcommand{\Ep}{{\mathrm{E}}}

\newcommand{\En}{{\mathbb{E}_n}}
\renewcommand{\Pr}{{\mathrm{P}}}

\newcommand{\uu}{u}

\newcommand{\eps}{\varepsilon}

\numberwithin{equation}{section}

\usepackage{tikz}
\usetikzlibrary {positioning}

\usepackage{bm}

\begin{document}

\title[HD econometrics and Regularized GMM]{High-Dimensional Econometrics and Regularized GMM}

\author[]{By Alexandre Belloni, Victor Chernozhukov, Denis Chetverikov, Christian Hansen, and Kengo Kato}

\begin{abstract}
This chapter presents key concepts and theoretical results for analyzing estimation and inference in high-dimensional models.  High-dimensional models are characterized by having a number of unknown parameters that is not vanishingly small relative to the sample size.  We first present results in a framework where estimators of parameters of interest may be represented directly as approximate means.  Within this context, we review fundamental results including high-dimensional central limit theorems, bootstrap approximation of high-dimensional limit distributions, and moderate deviation theory.  We also review key concepts underlying inference when many parameters are of interest such as multiple testing with family-wise error rate or false discovery rate control.  We then turn to a general high-dimensional minimum distance framework with a special focus on generalized method of moments problems where we present results for estimation and inference about model parameters.  The presented results cover a wide array of econometric applications, and we discuss several leading special cases including high-dimensional linear regression and linear instrumental variables models to illustrate the general results. 
\end{abstract}
\maketitle

\tableofcontents

\section{Introduction}
In this chapter, we review some of the main ideas and concepts from the literature on estimation and inference in high dimensions. High-dimensional models naturally arise in many contexts.  First, empirical researchers may want to build more flexible models in an effort to approximate real phenomena better.  Second, they may want to use more ``flexible" controls to make conditional exogeneity more plausible in an effort to (more plausibly) identify causal/structural effects.  Third, researchers may want to analyze  policy effects on very high-dimensional outcomes and/or across many groups.  Fourth, researchers may wish to  leverage high-dimensional  exclusion restrictions (``many instruments") in an effort to pin down structural parameters better.  These and other contexts motivate the set of methods and results we overview in this chapter.  In addition to providing an overview of useful tools, we develop some new results in order to make existing results more useful for applications in econometrics. We note that, since the literature on high-dimensional estimation and inference is large, we have opted to review only {\em some} of the main results from this literature. In this regard, our exposition complements other reviews, e.g. \cite{BC11b}, \cite{FLQ11}, \cite{CS17}, \cite{CHS}. For a textbook-level treatment, we refer an interested reader, for example, to \cite{BvdG11}, \cite{G15}, \cite{HTW15}, and \cite{G16}.

High-dimensionality typically refers to a setting where the number of parameters in a model is non-negligible compared to the sample size available.  The presence of a large number of parameters often necessitates us to design estimation and inference methods that are different from those used in classical, low-dimensional, settings. High-dimensional models have always been of interest in econometrics and have recently been gaining in popularity.  The recent interest in these models is due to both the availability of rich, modern data sets and to advances in the analysis of high-dimensional settings, such as the emergence of high-dimensional central limit theorems and regularization and post-regularization methods for estimation and inference.


\subsection{Inference with Many Approximate Means}
We split the chapter into two parts. In the first part, we consider inference using the {\em Many Approximate Means} (MAM) framework. In particular, we assume that we have a potentially high-dimensional vector of parameters
$$
\theta_0:=(\theta_{0 1},\dots,\theta_{0 p})^\prime\in\RR^p
$$
and its estimator
$$
\hat{\theta}:=(\hat{\theta}_1,\dots,\hat{\theta}_p)^\prime\in\RR^p,
$$
having an approximately linear form,
\begin{equation}\label{linearizeIntro}
\sqrt{n}(\hat{\theta}-\theta_0) = \frac{1}{\sqrt{n}}\sum_{i=1}^n Z_{i}+r_n,
\end{equation}
where $Z_1,\dots,Z_n$ are independent zero-mean random vectors in $\RR^p$, sometimes referred to as ``influence functions'', and  $r_n\in\RR^p$ is a vector of linearization errors that are asymptotically negligible; see the next section for the formal requirement. The vectors $Z_1,\dots,Z_n$ are either directly observable or can be consistently estimated. Here, we allow for the case $p\gg n$.



This framework is rather general and covers, in particular, the case of testing multiple means with $r_{n}=0$.  More generally, this framework covers multiple linear and non-linear $M$-estimators and also accommodates many de-biased estimators; see, e.g., \cite{HeShao00} for explicit conditions giving rise to linearization (\ref{linearizeIntro}) in low-dimensional settings and  \cite{BCK:biometrika} for conditions in the high-dimensional settings with $p \gg n$.

While conceptually easy-to-understand, the MAM framework allows us to present fundamental concepts in high-dimensional settings:
\begin{itemize}
\item[1.] Simultaneous inference.
\item[2.] Inference with False Discovery Rate control.
\item[3.] Estimation based on $\ell_1$-regularization.  
\end{itemize}
The first concept here includes simultaneous confidence interval construction for all (or some) components of the vector $\theta_0 = (\theta_{0 1},\dots,\theta_{0 p})'$. As we explain in the next section, constructing {\em simultaneous} confidence intervals is especially important in the high-dimensional settings and we explain how to construct such intervals. This concept also includes multiple testing with family-wise error rate (FWER) control, where we simultaneously test hypotheses about different components of the vector $\theta_0$ and we want to make sure that the probability of at least one false null rejection does not exceed the pre-specified level $\alpha$. The second concept includes multiple testing with false discovery rate (FDR) control, where we simultaneously test hypotheses about different components of $\theta_0$ and we want to make sure that the fraction of falsely rejected null hypotheses among all rejected null hypotheses does not exceed the pre-specified level $\alpha$, at least in expectation. FDR control is more liberal than FWER control, so procedures with FDR control typically have larger power than those with FWER control. This higher power may be particularly important, for example, in genoeconomics, where procedures with FWER control often fail to find any association between the outcome variables and genes; see Example \ref{ex: genes} below for the details. The third concept includes estimation of linear functionals of the vector $\theta_0$. We will show that estimating such functionals sometimes requires forms of regularization, and we will explain the details of $\ell_1$-regularization. This discussion will prepare us for the more ambitious problems arising in the second part of the chapter.



To perform the tasks described above, we will use some fundamental tools:
\begin{itemize}
\item[I.]  High-Dimensional Central Limit Theorem (with $p \gg n$)
\item[II.]  Moderate Deviation Theorem (Central Limit Theorem over Tail Areas)
\item[III.]  Regularization (focusing on $\ell_1$-type regularization)
\end{itemize}
One of the main goals of the first part of this chapter will therefore be to provide statements and discussion of these key tools in a simple but interesting framework. Outside of being useful in the MAM framework, these tools play an important role in the theory of high-dimensional estimation and inference more generally.  


We next review some simple motivating examples that fall into the MAM framework.
\begin{example}[Randomized Control Trials with Many Outcome Variables]\label{ex: rct many outcomes}
Consider a randomized control trial with $n$ participants, where each participant $i = 1,\dots, n$ is randomly assigned to either the treatment group ($D_i = 1$) or the control group ($D_i = 0$). Let $\gamma$ denote the probability of being assigned to the treatment group and suppose that for each participant $i$, we observe a large number of outcome variables represented by a vector $Y_i = (Y_{i1},\dots,Y_{i p})'\in\mathbb R^p$, which is often the case in practice. For each outcome variable $j = 1,\dots, p$, we then can estimate the average treatment effect
$$
\theta_{0 j} = \Ep[Y_{i j} \mid D_i = 1] - \Ep[Y_{i j} \mid D_i = 0]
$$ 
by
$$
\hat\theta_j = \frac{1}{n}\sum_{i = 1}^n\left( \frac{D_i Y_{i j}}{\gamma} - \frac{(1 - D_i)Y_{i j}}{1 - \gamma} \right).
$$
Clearly, this setting falls into the MAM framework with
$$
Z_{i j} = \left( \frac{D_i Y_{i j}}{\gamma} - \Ep[Y_{i j}\mid D_i = 1] \right) - \left( \frac{(1 - D_i)Y_{i j}}{1 - \gamma} - \Ep[Y_{i j}\mid D_i = 0] \right)
$$
for all $i = 1,\dots, n$ and $j = 1,\dots, p$. We note also that the MAM framework covers many other, more complicated, treatment effect settings, beyond randomized control trials, since treatment effect estimators are often asymptotically linear; e.g. see \cite{IR15}, \cite{HIR03}, \cite{A05}, and \cite{CCDDHNR18}, among many others.
\qed
\end{example}
\begin{example}[Randomized Control Trials with Many Groups]\label{ex: rct many groups}
Consider the same randomized control trial as in the previous example and suppose that for each participant $i = 1,\dots, n$, we have only one outcome variable $Y_i\in\mathbb R$ but we observe several discrete covariates, represented by a vector $X_i = (X_{i 1},\dots, X_{i d})'\in\mathbb R^d$. For simplicity, we can assume that each covariate is binary, so that $X_{i k} \in \{0,1\}$ for all $k = 1,\dots, d$. In this case, the vector $X_i$ can take $p = 2^d$ different values, denoted by $x_1,\dots, x_p$, and we can split all participants into $p$ groups depending on their values of $X_i$. Suppose, also for simplicity, that each group consists of $\bar n = n/p$ participants, i.e. all groups are equal in size. We then can estimate group-specific average treatment effects
$$
\theta_{0 j} = \Ep[Y_i\mid X_i = x_j, \ D_i = 1] - \Ep[Y_i\mid X_i = x_j, \ D_i = 0]
$$
by
$$
\hat\theta_j = \frac{1}{\bar n}\sum_{i\colon X_i = x_j} \left( \frac{D_i Y_i}{\gamma } - \frac{(1 - D_i)Y_i}{1 - \gamma} \right),
$$
and this setting again falls into the MAM framework, with $n$ replaced by $\bar n$. Note also that since the number of groups, $p = 2^d$, is exponential in the number of covariates, $d$, it is likely that $p \gg \bar n$ or at least $p\sim n$, making our analysis in this chapter particularly relevant.

To give a specific example of an experiment with many groups, consider the Tennessee Student Teacher Achievement Ratio (STAR) project conducted from 1985-89 and studied, e.g. in \cite{K99} among many others. In this project, over 11000 students from kindergarten to third grade in 79 schools were randomly assigned into small (13 to 17 students) or regular (22 to 25 students) classes. Classroom teachers were also randomly assigned to classes. Different student achievements were subsequently measured over many years. The project also collected many demographic variables characterizing students, teachers, and schools. For example, available data include gender (male or female) and race (white, black, asian, hispanic, native american, or other) for both students and teachers, month of birth (from Jan to Dec) for students, and educational achievement (associate, bachelor, master, master+, specialist, or doctoral) and years of teaching (from 0-42) for teachers. All these characteristics can be used to form a large number of groups of student-teacher pairs.
\qed
\end{example}

\begin{example}[Genoeconomics]\label{ex: genes}
In genoeconomics, a field that combines genetics and economics, researchers are interested in studying how genes affect economic behavior. This field is of interest because genetic information, for example, can provide direct measures of preferences of economic agents and can serve as a source of exogenous variation. The vast majority of the humane genome is the same among all humans, and the differences occur ``only'' in around 52 million SNPs (single-nucleotide polymorphisms). Most SNPs take only three values, (0, 1, 2), and modern technologies allow measuring the values of many, if not all, of these SNPs with minimal costs. The datasets in genoeconomics, therefore, often take the following form: We have a random sample of $n$ humans, where $n$ is of order of hundreds or thousands, and for each human $i = 1,\dots, n$, we have an outcome variable $Y_i\in\mathbb R$ and the vector of SNP values, $X_i = (X_{i 1},\dots, X_{i p})' \in\mathbb R^p$, where $p$ can be of order of thousands or even millions. To measure the association between the outcome variable and the SNP $j = 1,\dots, p$, we can use the slope coefficient $\theta_j$ in the linear regression
$$
Y_i = \alpha_j + \theta_{0 j} X_{i j} + \epsilon_{i j},\quad \Ep[\epsilon_{i j}\mid X_{i j}] = 0,
$$
which can be estimated by
$$
\hat\theta_j = \frac{\sum_{i=1}^n(X_{i j} - \bar X_j)(Y_i - \bar Y)}{\sum_{i = 1}^n (X_{i j} - \bar X_j)^2},
$$
where $\bar X_j = \sum_{i=1}^n X_{i j}/n$ and $\bar Y = \sum_{i=1}^n Y_i$. Clearly, this setting falls into the MAM framework with
$$
Z_{i j} = \frac{(X_{i j} - \Ep[X_{i j}])(Y_{i} - \Ep[Y_i])}{\text{Var}(X_{i j})}
$$
for all $i = 1,\dots, n$ and $j = 1,\dots, p$. We refer the reader to \cite{BCCGL12} for more detailed discussion of genoeconomics.
\qed
\end{example}
\begin{example}[Structural Models with Many Parameters]
The examples above outline simple cases where parameters are estimated either by sample means (Examples \ref{ex: rct many outcomes} and \ref{ex: rct many groups}) or by quantities that can be easily approximated by sample means (Example \ref{ex: genes}). In structural econometrics, we often use more sophisticated estimators of structural parameters, such as GMM. In the second part of the chapter, we therefore develop a high-dimensional regularized GMM estimator. This could be of interest, for example, in demand elasticity estimation, where the elasticity parameter $\theta_{0 j}$ varies across product groups (or product characteristics) $j = 1,\dots, p$. We show that it is possible to construct asymptotically unbiased estimators of these parameters using the double/de-biased regularized GMM approach. These estimators are asymptotically linear and fall into the MAM framework. We therefore can use inferential tools developed for the MAM framework to construct simultaneous confidence intervals and conduct multiple hypothesis testing using various approaches we discuss in the first part of the chapter.\qed
\end{example}



\subsection{Inference with Many Parameters and Moments} In the second part of the chapter, we study estimation and inference in the high-dimensional GMM setting, where both the number of moment equations and the dimensionality of the parameter of interest may be large. Specifically, we consider a random vector $X\in\mathbb R^{d_x}$, a vector of parameters $\theta\in\mathbb R^p$, and a vector-valued score function $g(X, \theta)$ mapping $\RR^{d_x}\times\mathbb R^p$ into $\mathbb R^m$, for some $m\geq p$. For the moment function
\begin{equation}\label{eq: gmm moment conditions}
 g(\theta) : =  \Ep [g(X, \theta)],
\end{equation}
we assume that the true parameter value $\theta_0$ satisfies
\begin{equation}\label{eq: general moment condition}
g(\theta_0) = 0.
\end{equation}
We are then interested in estimating $\theta_0$ and carrying out inference on $\theta_0$ using a random sample $X_1,\dots, X_n$ from the distribution of $X$. We allow for the case $m\gg n$ and $p\gg n$.

We develop a Regularized GMM estimator (RGMM) of $\theta_0$ and study its properties under various structural assumptions, such as sparsity or approximate sparsity of $\theta_0$. This novel estimator extends the Dantzig selector of Cand\`{e}s and Tao in \cite{CT07} that was developed specifically for estimating linear mean regression models. 

To gain intuition behind the RGMM estimator, we also consider a general minimum distance estimation problem, where the parameter $\theta_0$ is known to satisfy \eqref{eq: general moment condition} but the function $g(\theta)$ does not necessarily take the form \eqref{eq: gmm moment conditions}. Assuming that an estimator of $g(\theta)$ is available, we formulate a Regularized Minimum Distance (RMD) estimator of $\theta_0$ and develop its properties under easily-interpretable high-level conditions. Specializing these conditions for the GMM setting then allows us to derive properties of the RGMM estimator under relatively low-level conditions. 

Like other estimators developed for high-dimensional models, such as Lasso, the RGMM estimator is suitable for coping high-dimensionality of the problem but has a complicated asymptotic distribution, making inference based on this estimator problematic. We therefore also develop a Double/Debiased RGMM estimator (DRGMM) that is asymptotically linear, and thus fits into the MAM framework, reemphasizing the role of the MAM framework, and reducing the problem of inference on $\theta_0$ to our analysis in the first part of the paper. Importantly, for our DRGMM estimator, we also consider a version with the optimal weighting matrix.

\begin{example}[Linear Mean Regression Model]\label{ex: linear regression introduction}
One of the simplest examples falling into the GMM framework is the linear mean regression model,
\begin{equation}\label{eq: linear regression introduction}
Y = W'\theta_0 + \epsilon,\quad\Ep[\epsilon \mid W] = 0,
\end{equation}
where $Y\in\mathbb R$ is an outcome variable, $W\in\mathbb R^p$ a vector of covariates, $\epsilon\in\mathbb R$ noise, and $\theta_0\in\mathbb R^p$ a parameter of interest. 
This model fits into the GMM framework with
$$
g(X,\theta) = (Y - W'\theta)W,\quad X = (W,Y),
$$
which corresponds to the most widely studied case in the literature on high-dimensional models, e.g. the Lasso estimator of Tibshirani \cite{T96} and the Dantzig Selector of Cand\`{e}s and Tao \cite{CT07}. More generally, we can take a vector-valued function $t\colon \mathbb R^p\to\mathbb R^m$ with $m \geq p$, consider a vector $t(W)$ of transformations of $W$, and set
$$
g(X,\theta) = (Y - W'\theta) t(W),\quad X = (W,Y).
$$
By considering a sufficiently rich vector of functions $t$ and using optimally-weighted GMM, we can construct an estimator that achieves the semiparametric efficiency bound.

In this example, as well as in Examples \ref{ex: linear IV regression model} and \ref{Ex:NL-IV} below, we are often interested in a {\em low}-dimensional sub-vector of $\theta_0$ corresponding to the covariates of interest in the vector $W$ but sometimes the whole vector $\theta_0$ or some {\em high}-dimensional sub-vector of $\theta_0$ is of interest as well. For example, suppose that $W = (P',D\cdot P')'$, where $D\in\{0,1\}$ is a binary treatment variable and $P\in\mathbb R^{p/2}$ is a high-dimensional vector of controls, so that the model \eqref{eq: linear regression introduction} becomes
$$
Y = P'\alpha_0 + D\cdot P'\beta_0 + \epsilon,\quad \Ep[\epsilon\mid P, D] = 0
$$
with $\theta_0 = (\alpha_0',\beta_0')'$, where both $\alpha_0$ and $\beta_0$ are $(p/2)$-dimensional vectors of parameters.  Assuming that $D$ is randomly assigned conditional on $P$ then implies that $P'\beta_0$ is the Conditional Average Treatment Effect (CATE) for the outcome variable $Y$ and that $\beta_0$ is the vector of derivatives of the CATE with respect to the vector of controls $P$. Thus, the whole vector $\beta_0$ may be of interest.
\qed
\end{example}
\begin{example}[Linear IV Regression Model]\label{ex: linear IV regression model}
Consider the linear IV regression model
$$
Y = W'\theta_0 + \epsilon,\quad \Ep[\epsilon\mid Z] = 0,
$$
where we use the same notation as above and, in addition, $Z\in\mathbb R^{d_Z}$ is a vector of instruments. This model fits into the GMM framework with
$$
g(X,\theta) = (Y - W'\theta)Z,\quad X = (W,Z,Y),
$$
or, more generally,
$$
g(X,\theta) = (Y - W'\theta)t(Z),\quad X = (W,Z,Y),
$$
where $t\colon \mathbb R^{d_Z}\to\mathbb R^m$ is a vector-valued function. The case where $p$ is small and $m$ is large (larger than the sample size $n$) was originally studied in \cite{BCCH12} and the case where both $p$ and $m$ are large is considered in \cite{CHS}, \cite{GT14}, \cite{BCHN17}, and \cite{GLT17}. One of the novel parts of our analysis is that we can allow for the optimal weighting of the moment conditions.
\qed
\end{example}

\begin{example}[Nonlinear IV Regression Model]\label{Ex:NL-IV}
More generally, consider a non-linear IV regression model
$$
\Ep[f(Y,W'\theta_0)\mid Z] = 0,
$$
where we use the same notation as in Example \ref{ex: linear IV regression model} and, in addition, $f\colon \mathbb R^2\to\mathbb R$ is some function. This can be of interest, for example, in the analysis of discrete choice models where $Y$ is binary (or, more generally, discrete). In the same fashion as above, this model fits into the GMM framework with
$$
g(X,\theta) = f(Y,W'\theta)Z,\quad X = (W,Z,Y)
$$
or, more generally,
$$
g(X,\theta) = f(Y,W'\theta)t(Z),
$$
where $t\colon \mathbb R^{d_Z}\to\mathbb R^m$ is a vector-valued function.
\qed
\end{example}

\noindent
{\bf Notation.} In what follows, all models and probability measures $P$ can be indexed by the sample size $n$, so that models and their dimensions can change with $n$, allowing dimensionality to increase with $n$. We use ``wp $\to 1$" to abbreviate the phrase ``with probability that converges to 1", and we use arrows $\to_{\Pr}$ and $\leadsto_{\Pr}$ to denote convergence in probability and in distribution, respectively. The symbol $\sim$
means ``distributed as". The notation $a \lesssim b$ means $a = O(b)$ and $a \lesssim_{\Pr} b$ means $a = O_{\Pr}(b)$. We also use the notation $a \vee b = \max \{  a, b \}$ and $a \wedge b = \min \{ a , b \}$. For any $a\in\mathbb R$, $\lfloor a \rfloor$ denotes the largest integer that is smaller than or equal to $a$, and $\lceil a\rceil$ denotes the smallest integer that is larger than or equal to $a$. For a positive integer $m$,  $[m] = \{1,\dots, m\}$. 

Next, for any vector $x = (x_1,\dots,x_p)'$, we denote the $\ell_{1}$ and $\ell_{2}$ norms of $x$ by $\|x\|_1 = \sum_{j=1}^p|x_j|$ and $\| x \|_{2} = (\sum_{j=1}^p x_j^2)^{1/2}$, respectively. The $\ell_{0}$-``norm" of $x$, $\|x\|_0$, denotes the number of non-zero components of the vector $x$. Moreover, for any vector $x = (x_1,\dots,x_p)'$ in $\mathbb R^p$ and any set of indices $T\subset\{1,\dots,p\}$, we use $x_T = (x_{T 1},\dots,x_{T p})'$ to denote the vector in $\mathbb R^p$ such that $x_{T j} = x_j$ for $j\in T$ and $x_{T j} = 0$ for $j\in T^c$, where $T^c = \{ 1,\dots, p \} \setminus T$. For any matrix $A$ of $p$ columns, we use $\|A\|$ to denote the operator norm of $A$: $\| A \| = \sup_{x \in \RR^{p}, \| x \|_{2}=1} \| Ax \|_{2}$. 

The transpose of a column vector $x$ is denoted by $x'$. For a differentiable map $\RR^{d} \ni x \mapsto f(x) \in \RR^{k}$, we use $\partial_{x'} f$ to denote the $k\times d$ Jacobian matrix  $\partial f / \partial x' = (\partial f_{i}/\partial x_{j})_{1 \le i \le k,1 \le j \le d}$, and we correspondingly use the expression
$\partial_{x'} f(x_0)$ to denote $\partial_{x'} f (x) \mid_{x = x_0}$, etc. When we have an event $A$ whose occurrence depends on two independent random vectors, $X$ and $Y$, we use $\Pr_X(A)$ to denote the probability of $A$ with respect to the distribution of $X$, holding $Y$ fixed. For given $Z_{1},\dots,Z_{n}$, we use the notation $Z_{1}^{n} = (Z_{1},\dots,Z_{n})$.  We use $\Phi$ and $\phi$ to denote the cdf and pdf of the standard normal distribution.

Finally, we use standard empirical process theory notation. In particular, $\En[\cdot]$ abbreviates the average $n^{-1}\sum_{i=1}^n[\cdot]$ over index $i=1,\dots,n$, e.g. $\En[f(z_i)]$ denotes $n^{-1}\sum_{i=1}^n f(z_i)$. Also, if $Z$ is a random vector with law $P$ and support $\mathcal Z$, $(Z_i)_{i\in[n]}$ is a random sample from the distribution of $Z$, and $\mathcal F$ is a class of functions $f\colon\mathcal Z\to\mathbb R$, then $\mathbb{G}_n f  := \mathbb{G}_n f  (Z) := n^{-1/2} \sum_{i=1}^n (f(Z_i) - \Ep[f(Z)])$
for all $f\in\mathcal F$.



\section{Inference with Many Approximate Means}

\subsection{Setting}

Suppose that we have a parameter $\theta_0 = (\theta_{0 1},\dots,\theta_{0 p})' \in\mathbb R^p$ and an estimator $\hat{\theta}=(\hat{\theta}_1,\dots,\hat{\theta}_p)^\prime\in\RR^p$ of this parameter that has an approximately linear form:
\begin{equation}\label{linearize}
\sqrt{n}(\hat{\theta}-\theta_0) = \frac{1}{\sqrt{n}}\sum_{i=1}^n Z_{i}+r_n,
\end{equation}
where $Z_1,\dots,Z_n$ are independent zero-mean random vectors in $\RR^p$, sometimes called the ``influence functions," and  $r_n=(r_{n1},\dots,r_{np})^\prime\in\RR^p$ is a vector of linearization errors that are small in the sense that
\begin{equation}\label{eq: vanishing approximation error}
\| r_n\|_\infty = o_P\left(1/\sqrt{ \log (p n)}\right),
\end{equation}
 with a more precise requirement provided in Condition A. The vectors $Z_1,\dots,Z_n$ may not be directly observable, and we assume some estimators $\hat{Z}_1,\dots,\hat{Z}_n$ of these vectors are available in this case. In this section, we are interested in carrying out different types of inference on $\theta_0$. We are primarily interested in the case where $p$ is larger or much larger than $n$, but the results below apply when $p$ is smaller than $n$ as well. Throughout the chapter, we refer to this setting as the {\em Many Approximate Means} (MAM) framework.
 
In this section, we review results from the literature on the high-dimensional Central Limit Theorem (CLT), high-dimensional bootstrap theorems, moderate deviations for self-normalized sums, simultaneous confidence intervals, multiple testing with the Family-Wise Error Rate (FWER) control, and multiple testing with the False Discovery Rate (FDR) control. All results to be reviewed below exist in the literature for the case of many {\em exact} means, where the approximation errors are not present, $r_n = 0$, and the vectors $Z_i$ are observed. We extend these results to allow for many {\em approximate} means and also for unobservable but estimable vectors $Z_i$, i.e. we extend the results to cover the MAM framework. This extension is important because many estimators we work with are {\em asymptotically} linear but do not have to be linear in finite samples.

At the end of this section, we also consider the problem of estimating linear functionals of $\theta_0$, which motivates such concepts as sparsity and $\ell_1$-regularization and prepares us for the discussion in the second part of the chapter.

\subsection{CLT, Bootstrap, and Moderate Deviations}\label{MAM: CLT}
To perform inference on $\theta_0$, we first need to develop a distributional approximation for 
$$
\sqrt n(\hat\theta - \theta_0).
$$
When $p$ is fixed and $n$ gets large, $\sqrt n(\hat\theta - \theta_0)$ converges in distribution to a zero-mean Gaussian random vector by a classical CLT but here we are interested in the case with $p\to\infty$ or even $p/n\to\infty$ as $n\to\infty$ making classical CLTs inapplicable. We therefore rely on the high-dimensional CLT developed in \cite{CCK13,CCK15,CCK17}. To state the result, and also to extend it to allow for the MAM framework, we will use the following regularity conditions. Let $(B_n)_{n\geq 1}$, $(\delta_n)_{n\geq 1}$, and $(\beta_n)_{n\geq 1}$ be given sequences of constants satisfying $B_n\geq 1$, $\delta_n \searrow 0$, and $\beta_n \searrow 0$. Here, $B_n$ is allowed to grow to infinity as $n$ gets large. 


\medskip

\noindent \textbf{Condition M}.  \textit{ (i) $n^{-1}\sum_{i=1}^n \Ep[Z_{i j}^2] \geq 1$ for all $j \in [p]$ and (ii) $n^{-1}\sum_{i=1}^n\Ep[|Z_{i j}|^{2 + k}] \leq B^k_n$ for all $j\in[p]$ and $k = 1,2$. }

\medskip

Since $r_n$ is asymptotically negligible, in the sense that \eqref{eq: vanishing approximation error} holds, it follows from \eqref{linearize} that, for all $j\in[p]$, the asymptotic variance of $\sqrt n(\hat \theta_j - \theta_{0 j})$ is equal to $n^{-1}\sum_{i=1}^n \Ep[Z_{i j}^2] \geq 1$.  Thus, the first part of Condition M requires that this variance is bounded away from zero. Such a condition precludes existence of super-efficient estimators and is typically imposed even in classical settings, where $p$ is small relative to $n$. The second part of Condition M imposes the mild requirement that $n^{-1}\sum_{i=1}^n \Ep[|Z_{i j}|^3]$ and $n^{-1}\sum_{i=1}^n \Ep[|Z_{i j}|^4]$ do not increase too quickly with $n$.

\medskip

\noindent \textbf{Condition E}. 
\textit{Either of the following moment bounds holds:
\vspace{-3mm}
\begin{itemize}
\item[E.1] $\Ep[\exp(|Z_{i j}|/B_n)]\leq 2\text{ for all }i \in [n] \text{ and }j \in [p],$ and $\left( \frac{B_n^2\log^7(p n)}{n} \right)^{1/6}  \leq \delta_n$, or
\item[E.2] $\Ep\left[\max_{j\in [p]}(|Z_{i j}|/B_n)^4\right] \leq 1\text{ for all } i \in [n]$, and
$\left( \frac{B_n^4\log^7(p n)}{n} \right)^{1/6} \leq \delta_n.$
\end{itemize}
}

\medskip




The first part of Condition E.1 requires that $Z_{i j}$'s have light tails. In particular, under Condition E.1, the tails have to be sub-exponential:
$$
\Pr(|Z_{i j}| > x) = \Pr(\exp(|Z_{i j}|/ B_n) > \exp(x/B_n)) \leq 2\exp(-x/B_n),\quad \text{ for all }x>0,
$$
by Markov's inequality. In fact, Lemma 2.2.1 in \cite{VW96} shows that if
$$
\Pr(|Z_{i j}| > x) \leq 2\exp(-x/C_n),\quad\text{ for all }x>0,
$$
for some $C_n>0$, then $\Ep[\exp(|Z_{i j}|/B_n)] \leq 2$ holds with $B_n = 3C_n$. Thus, the first part of Condition E.1 is equivalent to all $Z_{i j}$'s having sub-exponential tails. The first part of Condition E.2, on the other hand, allows for heavy-tails but imposes some moment conditions on $\max_{j\in[p]}|Z_{i j}|$. Conditions E.1 and E.2 are therefore non-nested. The second parts of Conditions E.1 and E.2 impose restrictions on how fast $B_n$ and $p$ can grow. Note that we never impose E.1 and E.2 simultaneously. 

\medskip

\noindent \textbf{Condition A}. \textit{(i) The linearization errors obey $\Pr( \max_{j\in[p]}|r_{n j}| >  \delta_n / \log^{1/2}(p n)) \leq \beta_{n}$, and (ii) the estimates of the influence functions obey $\Pr(  \max_{j\in[p]}\En[(\hat{Z}_{ij}-Z_{ij})^2] > \delta^2_{n}/\log^2(p n)) \leq \beta_n$.}

\medskip

The first part of Condition A requires that the approximation errors in the vector $r_n$ are asymptotically negligible, and clarifies \eqref{eq: vanishing approximation error}. Note that if \eqref{eq: vanishing approximation error} holds, then it is rather standard to show that there exist {\em some} sequences of positive constants $(\delta_n)_{n\geq 1}$ and $(\beta_n)_{n\geq 1}$ satisfying $\delta_n\searrow 0$ and $\beta_n\searrow 0$ such that the first part of Condition A holds. The second part of Condition A requires the estimators $\hat Z_{i j}$ of $Z_{i j}$ to be sufficiently precise. Again, if $\hat Z_{i j}$'s satisfy
\begin{equation}\label{eq: estimates of z}
\max_{j\in[p]}\sqrt{\En[(\hat Z_{i j} - Z_{i j})^2]} = o_P(1/\log(p n)),
\end{equation}
then there exist {\em some} $(\delta_n)_{n\geq 1}$ and $(\beta_n)_{n\geq 1}$ satisfying $\delta_n\searrow 0$ and $\beta_n\searrow0$ such that the second part of Condition A holds. 



In order to state a key CLT result, let $\mathcal A$ be the class of all (closed) rectangles in $\mathbb R^p$, i.e. sets $A$ of the form
$$
A = \Big\{w = (w_1,\dots,w_p)'\in\mathbb R^p\colon w_{l j} \leq w_j \leq w_{r j}\text{ for all }j \in [p] \Big\},
$$
where $w_l = (w_{ l 1},\dots,w_{l p})'$ and $w_r = (w_{r 1},\dots,w_{r p})'$ are two vectors such that $w_{l j} \leq w_{r j}$ for all $j\in[p]$. (Here, both $w_{l j}$ and $w_{r j}$ can take values of $-\infty$ or $+\infty$.) Denote $V := n^{-1}\sum_{i=1}^n\Ep[Z_i Z_i']$, and let $N(0,V)$ be a zero-mean Gaussian random vector in $\mathbb R^p$ with covariance matrix $V$. The following theorem establishes the Gaussian approximation for the distribution of $\sqrt n(\hat\theta - \theta_0)$, which extends Proposition 2.1 in \cite{CCK17} to allow for many {\em approximate} means.


\begin{theorem}[CLT for Many Approximate Means]\label{thm: clt for mam}
Under Conditions M, E,  and A, the distribution of $\sqrt{n} (\hat \theta - \theta_0 )$ over rectangles is approximately Gaussian:
\begin{align}
&\sup_{A\in\mathcal A}\Big|\Pr(\sqrt n(\hat\theta - \theta_0) \in A) - \Pr(N(0,V)\in A)\Big| \leq C (\delta_n + \beta_n),\label{eq: clt for mam 1}
\end{align} 
where $C$ is a universal constant.
\end{theorem}
It is useful to note that Theorem \ref{thm: clt for mam} allows $p$ to be larger or much larger than $n$. For example, the theorem implies that if $Z_i$'s are i.i.d zero-mean random vectors with each component bounded in absolute value by some constant $C$ (independent of $n$) almost surely and the variance of each component bounded from below by one, then
\begin{equation}\label{eq: clt simple implication}
\sup_{A\in\mathcal A}\Big|\Pr(\sqrt n(\hat\theta - \theta_0) \in A) - \Pr(N(0,V)\in A)\Big| \to 0\text{ as }n\to\infty
\end{equation}
as long as $\log^7 p = o(n)$ and \eqref{eq: vanishing approximation error} and \eqref{eq: estimates of z} hold. Thus, Theorem \ref{thm: clt for mam} shows that Gaussian approximation over rectangles is possible even if $p$ is {\em exponentially} large in $n$.

Note, however, that the Gaussian approximation here is stated only for the probability of $\sqrt n(\hat \theta - \theta_0)$ hitting rectangles $A\in\mathcal A$. The same Gaussian approximation may not hold if we look at more general classes of sets, e.g. all (Borel measurable) convex sets. In fact, it is known that if we replace the class of all rectangles $\mathcal A$ in \eqref{eq: clt simple implication} by the class of all Borel measurable convex sets, then we must assume that $p = o(n^{1/3})$, meaning $p\ll n$, in order to satisfy \eqref{eq: clt simple implication}; see discussion on p. 2310 of \cite{CCK17}. On the other hand, as we will see below, the class of all rectangles is large enough to make Theorem \ref{thm: clt for mam} useful in many applications. See also Remark \ref{rem: sparsely convex sets} below on how we can extend the class of rectangles and still allow for $p\gg n$.

The Gaussian approximation result of Theorem \ref{thm: clt for mam} is useful as applications below indicate, but does not immediately give a practical distributional approximation since the covariance matrix $V$ is typically unknown. We therefore also consider bootstrap approximations. In particular, we consider the Gaussian (or multiplier) and empirical (or nonparametric) types of bootstrap.
For the Gaussian bootstrap, let $e = (e_1,\dots,e_n)'$ be a vector consisting of i.i.d. $N(0,1)$ random variables independent of the data yielding the estimator $\hat\theta$. 
 A Gaussian bootstrap draw of the estimator $\hat\theta$ is then defined as
\begin{equation}\label{eq: multiplier draw}
\sqrt{n}(\hat \theta^* - \hat \theta) := \frac{1}{\sqrt n}\sum_{i=1}^n e_i \hat Z_i \quad \text{or}  \quad \hat \theta^* =
\hat \theta + \frac{1}{n}\sum_{i=1}^n e_i \hat Z_i.
\end{equation}
Alternatively, letting $e = (e_{1},\dots,e_{n})'$ be a vector following the multinomial distribution with parameters $n$ and success probabilities $1/n,\dots,1/n$ independent of the data, we can define an empirical bootstrap draw of the estimator $\hat\theta$ as
\begin{equation}\label{eq: empirical draw}
\sqrt{n}(\hat \theta^* - \hat \theta) := \frac{1}{\sqrt n}\sum_{i=1}^n (e_{i} - 1) \hat Z_i \quad \text{or}  \quad \hat \theta^* =
\hat \theta + \frac{1}{n}\sum_{i=1}^n (e_i - 1) \hat Z_i.
\end{equation}
Equivalently, the empirical bootstrap draw of $\hat{\theta}$ can be constructed as 
\begin{equation}\label{eq: alternative empirical draw}
\sqrt{n} (\hat{\theta}^{*} - \hat{\theta})= \frac{1}{\sqrt{n}} \sum_{i=1}^{n} (\hat{Z}_{i}^{*} - \overline{\hat{Z}}) \quad \text{or} \quad \hat{\theta}^{*} = \hat{\theta} + \frac{1}{n} \sum_{i=1}^{n} (\hat{Z}_{i}^{*} -\overline{\hat{Z}}),
\end{equation}
where $\hat{Z}_{1}^{*},\dots,\hat{Z}_{n}^{*}$ are an i.i.d. sample from the empirical distribution of $\hat{Z}_{1},\dots,\hat{Z}_{n}$, and $\overline{\hat{Z}} = n^{-1} \sum_{i=1}^{n} \hat{Z}_{i}$. Indeed, the latter expression (\ref{eq: alternative empirical draw}) reduces to the former expression (\ref{eq: empirical draw}) by setting each $e_{i}$ as the number of times that $\hat{Z}_{i}$ is ``redrawn'' in the bootstrap sample, and the vector $e=(e_1,\dots,e_n)'$ then follows the multinomial distribution with parameters $n$ and success probabilities $1/n,\dots,1/n$ independent of the data.

The following theorems show that the distribution of the bootstrap draw with respect to $e$ approximates the Gaussian distribution given in Theorem \ref{thm: clt for mam}.

\begin{theorem}[Gaussian Bootstrap for Many Approximate Means]\label{thm: bootstrap for mam}
Under Conditions M, E, and A,  the distribution of a $N(0,V)$ vector over rectangles can be approximated by the Gaussian bootstrap: there exists a universal constant $C$ such that for the Gaussian bootstrap draw $\hat\theta^*$ given in (\ref{eq: multiplier draw}),
\begin{align*}
\sup_{A\in\mathcal A}\Big|\Pr_e(\sqrt{n}(\hat \theta^* - \hat \theta)  \in A) - \Pr(N(0, V)\in A)\Big| \leq C \delta_n
\end{align*}
holds with probability at least $1 - \beta_n - n^{-1}$ in the case of E.1 and $1-\beta_n - (\log n)^{-2}$ in the case of E.2.
\end{theorem}

\begin{theorem}[Empirical Bootstrap for Many Approximate Means]\label{thm: empirical bootstrap for mam}
Assume that Conditions M, E, and A are satisfied. In addition, assume that $\Pr(\max_{i\in[n]}\max_{j\in[p]}|\hat Z_{i j} - Z_{i j}| > 1) \leq \beta_n$. Then the distribution of a $N(0,V)$ vector over rectangles can be approximated by the empirical bootstrap: there exists a universal constant $C$ such that for the empirical bootstrap draw $\hat\theta^*$ given in (\ref{eq: empirical draw}),
\begin{align*}
\sup_{A\in\mathcal A}\Big|\Pr_e(\sqrt{n}(\hat \theta^* - \hat \theta)  \in A) - \Pr(N(0, V)\in A)\Big| \leq C \delta_n
\end{align*}
holds with probability at least $1 - 2\beta_n - n^{-1}$ in the case of E.1 and
\begin{equation}\label{eq: bootstrap for mam e2}
\sup_{A\in\mathcal A}\Big|\Pr_e(\sqrt{n}(\hat \theta^* - \hat \theta)  \in A) - \Pr(N(0, V)\in A)\Big| \leq C \delta_n \log^{1/3} n
\end{equation}
holds with probability at least $1-2\beta_n - (\log n)^{-2}$ in the case of E.2.
\end{theorem}

\begin{remark}[Comparison of Gaussian and Empirical Bootstraps]
Comparing Theorems \ref{thm: bootstrap for mam} and \ref{thm: empirical bootstrap for mam} suggests that the Gaussian bootstrap may be more accurate than the empirical bootstrap. However, it is important to remember that both theorems only give upper bounds on the distributional approximation errors, and so such a conjecture may or may not be valid. In fact, there is some evidence that the empirical bootstrap may be more accurate than the Gaussian bootstrap because the former is able to better match higher-order moments of $Z_i$'s; see \cite{DZ17}.\qed
\end{remark}

\begin{remark}[Sparsely Convex Sets]\label{rem: sparsely convex sets}
We note that Theorems \ref{thm: clt for mam}-\ref{thm: empirical bootstrap for mam} can be extended to allow for somewhat more general classes of sets, beyond the class of rectangles. In particular, these theorems can be extended to allow for classes of {\em sparsely convex sets}. For an integer $s>0$, we say that $A\subset\mathbb R^p$ is an $s$-sparsely convex set if there exists an integer $Q>0$ and convex sets $A_q\subset\mathbb R^p$, $q\in[Q]$, such that $A = \cap_{q\in[Q]}A_q$ and the indicator function of each $A_q$, $w\mapsto 1\{w\in A_q\}$, depends on at most $s$ components of its argument $w = (w_1,\dots,w_p)'$. Each rectangle, for example, is clearly a $1$-sparsely convex set. An example of $2$-sparsely convex set is
$$
A = \Big\{ w = (w_1,\dots,w_p)'\in\mathbb R^p\colon \max_{j,k\in[p];\ j\neq k}(w_j^2 + w_k^2) \leq x \Big\},\quad x\geq 0.
$$
Theorems  \ref{thm: clt for mam}-\ref{thm: empirical bootstrap for mam}  can be extended to allow $\mathcal A$ to be the class of all $s$-sparsely convex sets as long as $s$ is not too large; we refer to \cite{CCK17} for details in the case of many exact means and leave the case of many approximate means to future work.
\qed 
\end{remark}

\begin{remark}[Weakening Condition E.2]
We also note that the second part of Condition E above can be slightly weakened and generalized. In particular, it can be replaced by the following condition:
\begin{itemize}
\item[E.2'] For some $q \in [4,\infty)$, $\Ep\left[\max_{j\in [p]}(|Z_{i j}|/B_n)^q\right] \leq 1\text{ for all } i \in [n]$, and
$\left( \frac{B_n^2\log^7(p n)}{n} \right)^{1/6} +  \left( \frac{B^2_n\log^3(p n)}{n^{1 - 2/q}} \right)^{1/3} \leq \delta_n.$
\end{itemize}
If we use this alternative version of Condition E, Theorems \ref{thm: clt for mam}-\ref{thm: empirical bootstrap for mam} still hold but the constant $C$ in these theorems then depend on $q$, whenever E.2' is used. The same remark also applies to all theorems below where Condition E is used.\qed
\end{remark}



\begin{figure}[htbp]
\begin{center}
\hspace*{-0.5cm}\includegraphics[width=6.0in]{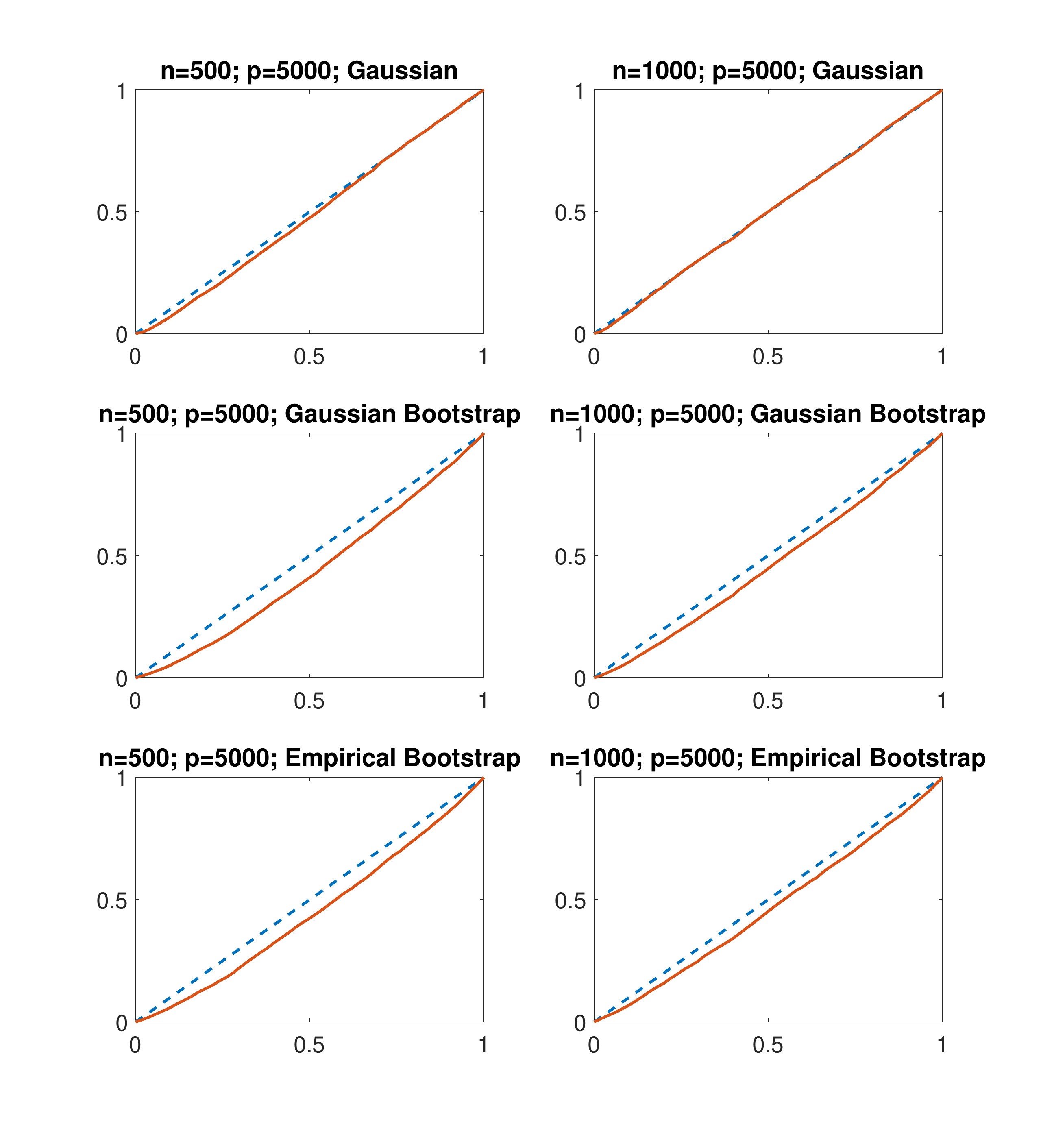}
\caption{P-P plots comparing the distribution of $\max_{j\in[p]}\sqrt n|\hat\theta_j - \theta_{0 j}|$ with its Gaussian, Gaussian bootstrap, and empirical bootstrap approximations in the example motivated by the problem of selecting the regularization parameter of the RMD estimator in Section \ref{sec: many moments}. Here, $Z_{i j}$'s are generated as
$Z_{i j}=W_{i j}\eps_{i}$ with $\eps_{i} \sim t(4),$ (the $t$-distribution with four degrees of freedom), and $W_{i j}$'s are non-stochastic (simulated once using $U[0,1]$ distribution independently across $i$ and $j$). We assume that $Z_{i j}$'s are observed and set $\hat Z_{i j} = Z_{i j}$ for all $i$ and $j$. The dashed line is 45$^\circ$. To generate bootstrap approximations, we use randomly selected sample of $Z_{i j}$'s. The figure indicate that all three approximations are good, and the quality of the approximation is particularly good for the tail probabilities, which is most relevant for practical applications.}
\label{fig: hd approximation}
\end{center}
\end{figure}

Figure \ref{fig: hd approximation} illustrates Theorems \ref{thm: clt for mam}, \ref{thm: bootstrap for mam}, and \ref{thm: empirical bootstrap for mam} for rectangles $A$ of a particular type:
$$
A = \Big\{w = (w_1,\dots,w_p)'\in\mathbb R^p\colon -x \leq w_j\leq x\text{ for all }j\in[p]\Big\},\quad x\geq 0.
$$
Specifically, Figure \ref{fig: hd approximation} plots 
$$
\Pr(\sqrt n\| \hat\theta - \theta_0 \|_{\infty} \leq x)\text{ against }\Pr(\| N(0,V) \|_{\infty}\leq x)
$$ 
and 
$$
\Pr(\sqrt n\| \hat\theta - \theta_0 \|_{\infty} \leq x)\text{ against }\Pr_e(\sqrt n\|\hat\theta^* - \hat\theta\|_{\infty} \leq x)
$$ 
as $x$ varies from $0$ to $\infty$ for different values of $n$ and $p$ and a distribution of $Z_{i}$'s motivated by the problem of selecting the regularization parameter of the RMD estimator in Section \ref{sec: many moments}, where $\hat\theta^*$ is either the Gaussian or the empirical bootstrap draw. The figure indicates that both Gaussian and bootstrap approximations in Theorems \ref{thm: clt for mam}, \ref{thm: bootstrap for mam}, and \ref{thm: empirical bootstrap for mam} are rather precise.

Another useful result for inference in high-dimensional settings is a moderate deviation theorem for self-normalized sums, which we present below. This result typically leads to conservative inference but requires very weak moment conditions. In particular, it does not require Condition E.

\begin{theorem}[Moderate Deviations for Many Approximate Means]\label{thm: moderate deviations for mam}
Assume that Conditions M and A are satisfied. Also, let $\bar C>0$ be some constant and assume that $(2\bar C)^3B_n\log^{3/2}(p n)/\sqrt n \leq \delta_n\leq1$. Then there exist constants $n_0$ and $C$ depending only on $\bar C$ such that
\begin{equation}\label{eq: SNMD for mam 1}
\left | \Pr\left(\frac{\sqrt n(\hat \theta_j - \theta_{0 j})}{ (\En[\hat Z_{i j}^2])^{1/2} } > x \right) -  (1- \Phi(x)) \right | 
\leq C\Big( (1- \Phi(x))\delta_n + \beta_n + (p n)^{-1}\Big)
\end{equation}
for all $n\geq n_0$ and $1\leq x\leq \bar C\log^{1/2}(p n)$. In addition,
\begin{equation}\label{eq: thm 4 part 2}
\Pr\left(\max_{j\in[p]} \frac{\sqrt n(\hat\theta_j - \theta_{0 j})}{(\En[\hat Z_{i j}^2])^{1/2}} > \Phi^{-1}(1-\alpha/p) \right) \leq \alpha + C \Big(\alpha \delta_n + \beta_n + (p n)^{-1} \Big)
\end{equation}
for all $n\geq n_0$ and $\alpha$ such that $1 \leq \Phi^{-1}(1 - \alpha / p) \leq \bar C\log^{1/2}(p n)$.
\end{theorem}

Since for any $x>0$, we have $1 - \Phi(x) \leq \exp(-x^2/2)$ by Proposition 2.5 in \cite{D14}, it follows that $\Phi^{-1}(1 - 1/p n) \leq \sqrt{2\log(p n)}$, and so setting $\alpha = 1/n$ in \eqref{eq: thm 4 part 2} gives the following corollary of Theorem \ref{thm: moderate deviations for mam}:
\begin{corollary}[Maximal Inequality for Many Approximate Means]\label{cor: maximal inequality for mam}
Assume that Conditions M and A are satisfied. Also, assume that $(2\sqrt 2)^3B_n\log^{3/2}(p n)/\sqrt n \leq \delta_n\leq1$. Then there exist universal constants $n_0$ and $C$ such that for all $n\geq n_0$,
$$
\Pr\left(\max_{j\in[p]} \frac{\sqrt n(\hat\theta_j - \theta_{0 j})}{(\En[\hat Z_{i j}^2])^{1/2}} > \sqrt{2\log(p n)}\right) \leq C \Big(\beta_n + n^{-1} \Big).
$$
In particular,
\begin{equation}\label{eq: maximal growth}
\frac{| \sqrt n(\hat\theta_j - \theta_{0 j}) |}{(\En[\hat Z_{i j}^2])^{1/2}} = O_P(\sqrt{\log(p n)})
\end{equation}
uniformly over $j\in[p]$.
\end{corollary}
If $Z_{i j}$'s are all bounded, or at least sub-Gaussian, it is straightforward to show by combining the union bound and exponential inequalities, such as those of Hoeffding or Bernstein, that
\begin{equation}\label{eq: maximal growth 2}
\frac{| \sqrt n(\hat\theta_j - \theta_{0 j}) |}{(V_{j j})^{1/2}} = O_P(\sqrt{\log(p n)})
\end{equation}
uniformly over $j\in[p]$. Comparing \eqref{eq: maximal growth} and \eqref{eq: maximal growth 2} now reveals an interesting feature of Theorem \ref{thm: moderate deviations for mam} and Corollary \ref{cor: maximal inequality for mam}: replacing the {\em true value} $V_{j j}$ of the asymptotic variance of $\sqrt n(\hat\theta_j - \theta_{0 j})$ by an {\em estimator} $\En[\hat Z_{i j}^2]$ allows us to obtain the same bound, $O_P(\sqrt{\log(p n)})$, for the normalized version of $\sqrt n(\hat\theta_j - \theta_{0 j})$ without imposing strong moment conditions on the data, such as boundedness, since Corollary \ref{cor: maximal inequality for mam} only assumes four finite moments of the $Z_{i j}$'s (via Condition M). Results of this form were used previously by \cite{BCCH12} in the theory of high-dimensional estimation via Lasso to allow for noise with heavy tails. Also, \cite{CCK13b} used such results to develop computationally efficient tests of many moment inequalities for heavy-tailed data.


\subsection{Simultaneous Confidence Intervals.}
When only one $\theta_{0 j}$ is of interest, it follows from standard arguments that
$$
\frac{\sqrt n(\hat\theta_j - \theta_{0 j})}{(\En[\hat Z_{i j}^2])^{1/2}}\leadsto_{\Pr} N(0,1)
$$
as long as $r_{n j} = o_P(1)$ and $\En[(\hat Z_{i j} - Z_{i j})^2] = o_P(1)$.  We can thus, e.g., construct a two-sided confidence interval for $\theta_{0 j}$ with asymptotic coverage $1 - \alpha$ for some $\alpha\in(0,1)$ as
$$
CS_j(1-\alpha) = \left[\hat\theta_j - (\En[\hat Z_{i j}^2])^{1/2}\frac{z_{\alpha/2}}{\sqrt n};\ \ \hat\theta_j + (\En[\hat Z_{i j}^2])^{1/2}\frac{z_{\alpha/2}}{\sqrt n} \right],
$$
where $z_{\alpha/2}$ denotes the $(1-\alpha/2)$-quantile of $N(0,1)$, i.e.,  $1 - \Phi(z_{\alpha/2}) = \alpha/2$. However, the confidence intervals above are too optimistic when many components $\theta_{0 j}$ of the parameter vector $\theta_0$ are of interest, and it is likely that one or several $\theta_{0 j}$'s will fall out of their respective confidence intervals. Therefore, to obtain valid inferential statements, we need to carry out a multiplicity adjustment to explicitly take into account that many $\theta_{0 j}$ are of interest. In this subsection, we demonstrate how to perform this adjustment and construct {\em simultaneous} confidence intervals for multiple components of $\theta_0$  using Theorems \ref{thm: clt for mam}--\ref{thm: moderate deviations for mam}. 

The following quantity will play an important role in our analysis:
$$
\lambda(1 - \alpha):=(1 - \alpha)\text{ quantile of }\|\sqrt n W(\hat\theta - \theta_0)\|_{\infty},
$$
where $W := \text{diag}(w_1,\dots,w_p)$ is a diagonal, potentially unknown, weighting matrix. For concreteness, for all $j \in [p]$, we often set $w_j = V_{j j}^{-1/2}$, which normalizes each $\sqrt n w_j(\hat\theta_j - \theta_{0 j})$ to have asymptotic variance one, or $w_j = 1$, which simplifies the analysis.

If we knew $\lambda(1 - \alpha)$ and $W$, we would be able to use confidence intervals
$$
CS =\prod_{j\in[p]} CS_j, \quad CS_j := \left[\hat\theta_j - \frac{\lambda(1 - \alpha)}{ w_j\sqrt n}; \ \ \hat\theta_j + \frac{\lambda(1- \alpha)}{ w_j \sqrt n} \right],
$$
since they clearly satisfy the desired coverage condition,
\begin{equation}\label{eq: conf intervals main property}
\Pr\Big(\theta_{0 j} \in CS_j\text{ for all }j\in[p]\Big) = 1- \alpha + o(1).
\end{equation}
In practice, however, $\lambda(1 - \alpha)$ is typically unknown and has to be estimated from the data. To this end, let $\hat W:=\text{diag}(\hat w_1,\dots,\hat w_p)$ be an estimator of $W$ and let
$$
\hat\lambda(1 - \alpha) := (1 - \alpha)\text{ quantile of }\| \sqrt n \hat W(\hat\theta^* - \hat\theta) \|_{\infty}\mid \hat W, (\hat Z_i)_{i=1}^n,
$$
where $\hat\theta^*$ is obtained via the Gaussian bootstrap, \eqref{eq: multiplier draw}.  The case where $\hat\theta^*$ is obtained via the empirical bootstrap, \eqref{eq: empirical draw}, can be analyzed similarly. We then can define feasible confidence intervals as
\begin{equation}\label{eq: simultaneous conf intervals}
CS =\prod_{j\in[p]} CS_j, \quad CS_j := \left[\hat\theta_j - \frac{\hat\lambda(1 - \alpha)}{\hat w_j\sqrt n}; \ \ \hat\theta_j + \frac{\hat\lambda(1- \alpha)}{\hat w_j \sqrt n} \right].
\end{equation}
Below, we will show that these feasible confidence intervals still satisfy \eqref{eq: conf intervals main property}, under certain regularity conditions. To prove this claim, we impose the following condition:

\medskip

\noindent \textbf{Condition W}.
\textit{(i) For some $C_W\geq 1$, the diagonal elements of the matrix $W = \text{diag}(w_1,\dots,w_p)$ satisfy $C_W^{-1} V_{j j}^{-1/2} \leq w_j \leq C_W$ for all $j \in [p]$, and (ii) the estimator $\hat W = \text{diag}(\hat w_1,\dots,\hat w_p)$ of the matrix $W = \text{diag}(w_1,\dots,w_p)$ satisfies $\Pr(  \max_{j\in[p]}|\hat w_j - w_j|^2 (1+ \En[\hat Z_{i j}^2]) > \delta^2_{n}/\log^2(p n)) \leq \beta_n$.}

\medskip

This condition holds trivially with $C_W = 1$ if we set $\hat w_j = w_j = 1$ for all $j\in[p]$ (recall that by Condition M, we have $V_{j j}\geq 1$ for all $j\in[p]$). Also, we show that Conditions M, E, and A imply Condition W, with possibly different $\delta_n$ and $\beta_n$, if we set $w_j = V_{j j}^{-1/2}$ and $\hat w_j = (\En[\hat Z_{i j}^2])^{-1/2}$ for all $j \in [p]$ as a part of the proof of Theorem \ref{thm: MHT} below.

The key observation that allows us to show that the confidence intervals \eqref{eq: simultaneous conf intervals} satisfy the desired coverage condition \eqref{eq: conf intervals main property} will be to show that the bootstrap quantile function $\hat\lambda(1 - \alpha)$, as well as the original quantile function $\lambda(1 - \alpha)$, can be approximated by the Gaussian quantile function,
$$
\lambda^g(1 - \alpha) := (1 - \alpha)\text{ quantile of }\| W N(0,V) \|_{\infty}.
$$
It is this place, where Theorems \ref{thm: clt for mam}-\ref{thm: empirical bootstrap for mam} play a key role. Formally, we have the following results.


\begin{theorem}[Quantile Comparison]\label{lem: quantile comparison}
Assume that Conditions M, E, A, and W are satisfied. Then there exists a constant $C$ depending only on $C_W$ such that for $\epsilon_n:= C(\delta_n + \beta_n)$,
\begin{equation}\label{eq: quantile comparison 1}
\lambda^g(1 - \alpha - \epsilon_n) \leq \lambda(1 - \alpha) \leq \lambda^g(1 - \alpha + \epsilon_n).
\end{equation}
In addition,
\begin{equation}\label{eq: quantile comparison 2}
\lambda^g(1 - \alpha - \epsilon_n) \leq \hat\lambda(1 - \alpha) \leq \lambda^g(1 - \alpha + \epsilon_n),
\end{equation}
holds with probability at least $1 - 2\beta_n - n^{-1}$ in the case of E.1 and $1-2\beta_n - (\log n)^{-2}$ in the case of E.2. Moreover, for any $a\in(0,1)$,
\begin{equation}\label{eq: gaussian quantile bound 1}
\lambda^g(1 - a) \leq \bar \sigma \Phi^{-1}(1 - a/(2 p)) \leq \bar\sigma \sqrt{2\log(2p/a)},
\end{equation}
where $\bar\sigma := \max_{j\in[p]} (w_j V_{j j}^{1/2})$.
\end{theorem}

\begin{theorem}[Simultaneous Confidence Intervals]\label{thm: simultaneous conf intervals}
Assume that Conditions M, E, A, and W are satisfied. Then the confidence intervals \eqref{eq: simultaneous conf intervals} satisfy the desired coverage condition \eqref{eq: conf intervals main property}. Moreover, with probability $1 - o(1)$, uniformly over $j \in [p]$, the maximum of weighted radii of these confidence intervals is bounded from above by
$$
\sup\{ \| W (\hat \theta- \theta) \|_\infty: \theta \in CS \} \leq  (1+o(1)) \lambda^g(1 - \alpha + \epsilon_n)/\sqrt n,
$$
where $CS$ is defined in \eqref{eq: simultaneous conf intervals}.
\end{theorem}

\begin{remark}[Simultaneous Confidence Intervals via Moderate Deviations]
The construction of the simultaneous confidence intervals above relies upon the bootstrap approximation $\hat\lambda(1 - \alpha)$ of the quantile function $\lambda(1 - \alpha)$. Alternatively, we can use the moderate deviation theorem for this purpose. Specifically, set $w_j = V_{j j}^{-1/2}$ and $\hat w_j = (\En[\hat Z_{i j}^2])^{-1/2}$ for all $j\in[p]$. Then it follows from Theorem \ref{thm: moderate deviations for mam} that $\Phi^{-1}(1 - \alpha/2 p)$ can be used as a good upper bound for $\lambda(1 - \alpha)$. Therefore, using the same arguments as those in the proof of Theorem \ref{thm: simultaneous conf intervals}, we can show that, under certain regularity conditions allowing for $p\gg n$, the confidence intervals
$$
CS =\prod_{j\in[p]} CS_j, \quad CS_j := \left[\hat\theta_j - \frac{\Phi^{-1}(1 - \alpha /2 p)}{(\En[\hat Z_{i j}^2])^{-1/2}\sqrt n}; \ \ \hat\theta_j + \frac{\Phi^{-1}(1 - \alpha/2 p)}{(\En[\hat Z_{i j}^2])^{-1/2}\sqrt n} \right].
$$
satisfies the desired coverage condition \eqref{eq: conf intervals main property}.
\qed
\end{remark}



\subsection{Multiple Testing with FWER Control}
In this subsection, we are interested in simultaneously testing hypotheses about different components of $\theta_0$. For concreteness, for each $j \in [p]$, we consider testing
$$
H_j:\theta_{0 j}\leq\bar\theta_{0j}\text{ against }H_j^\prime:\theta_{0 j}> \bar\theta_{0j}
$$
for some given value $\bar\theta_{0 j}$, where $H_j$ is the null and $H_j'$ is the alternative. The results below also apply for testing $H_j:\theta_{0 j} = \bar\theta_{0 j}$ against $H_j':\theta_{0 j}\neq\bar\theta_{0 j}$ with obvious modifications of test statistics and critical values.

Since we are interested in testing these hypotheses simultaneously for all $j \in [p]$, we seek a procedure that would reject {\em at least} one true null hypothesis with probability not larger than $\alpha + o(1)$, uniformly over a large class of data-generating processes and, in particular, uniformly over the set of true null hypotheses. In the literature, procedures with this property are said to have {\em strong control of the Family-Wise Error Rate} (FWER).

More formally, let $\mathcal P$ be a set of probability measures for the distribution of the data corresponding to different data generating processes, and let $P\in\mathcal P$ be the true probability measure.
Each null hypothesis $H_{j}$ is equivalent to $P\in\mathcal P_{j}$
for some subset $\mathcal P_{j}$ of $\mathcal P$.
Let $\mathcal{W}:=\{1,\dots,p\}$ and for $w\subset \mathcal{W}$
denote $\mathcal P^{w}:=(\cap_{j\in w}\mathcal P_{j})\cap(\cap_{j\notin w}\mathcal P_{j}^c)$ where $\mathcal P_{j}^c:=\mathcal P\backslash\mathcal P_{j}$. In words, $\mathcal P_j$ is the set of probability measures corresponding to the $j$th null hypothesis $H_j$ being true and $\mathcal P^w$ is the set of probability measures such that all null hypotheses $H_j$ with $j\in w$ are true and all null hypotheses $H_j$ with $j\notin w$ are false.

Corresponding to this notation, strong FWER control means
\begin{equation}
\sup_{w \subset \mathcal{W}}\sup_{P\in\mathcal P^{w}}
\Pr_{P}\{\text{reject at least one \ensuremath{H_{j}} for \ensuremath{j\in w}}\}
\leq \alpha+o(1) \label{eq: strong control}
\end{equation}
where $\Pr_{P}$ denotes the probability distribution generated by the probability measure of the data $P$. We seek a procedure that satisfies \eqref{eq: strong control}.

We consider three different (but related) procedures: the Bonferroni, Bonferroni-Holm, and Romano-Wolf procedures. All three procedures will be based on the $t$-statistics,
\begin{equation}\label{eq: t statistics}
t_{j}:=\frac{\sqrt{n}(\hat{\theta}_{j}-\bar\theta_{0j})}{(\En[\hat Z_{i j}^2])^{1/2}}, \quad j \in [p].
\end{equation}
For each $j \in [p]$, the {\em Bonferroni} procedure rejects $H_j$ if $t_j > \Phi^{-1}(1 - \alpha/p)$. It is clear why this procedure satisfies \eqref{eq: strong control}: For any $w\subset \mathcal W$ and any $P\in\mathcal P^w$,
$$
\max_{j\in w} t_j \leq \max_{j\in w}\frac{\sqrt{n}(\hat{\theta}_j-\theta_{0 j})}{(\En[\hat Z_{i j}^2])^{1/2}} \leq \max_{j\in[p]} \frac{\sqrt{n}(\hat{\theta}_j-\theta_{0 j})}{(\En[\hat Z_{i j}^2])^{1/2}},
$$
and under the conditions of Theorem \ref{thm: moderate deviations for mam},
$$
\Pr_P\left( \max_{j\in[p]}\frac{\sqrt{n}(\hat{\theta}_j-\theta_{0 j})}{(\En[\hat Z_{i j}^2])^{1/2}} > \Phi^{-1}(1 - \alpha/p) \right) \leq \alpha + o(1),
$$
where $o(1)$ does not depend on $(w,P)$.  This result is formally stated in Theorem \ref{thm: fwer bonferroni}.

The Bonferroni procedure is a one-step method, which determines the hypotheses $H_j$ to be rejected in just one step. This procedure can be improved by employing multi-step methods.  In particular, so-called {\em stepdown} methods also have strong FWER control but may reject some $H_j$'s that are not rejected by the Bonferroni procedure, thus yielding important power improvements. We will consider the following form of the stepdown methods:
\begin{itemize}
\item[(1)] For a subset $w\subset \mathcal{W}$, let $c_{1-\alpha,w}$ be a (potentially conservative) estimator of
the $(1-\alpha)$ quantile of $\max_{j\in w}t_{j}$. On the first step, let $w(1)=\mathcal{W}$ and reject
all hypotheses $H_{j}$ satisfying $t_{j}>c_{1-\alpha,w(1)}$.
If no null hypothesis is rejected, then stop.  If some $H_{j}$'s are
rejected, let $w(2)$ be the set of all null hypotheses that
were not rejected on the first step and move to (2).  \\
\item[(2)] On step $l\geq2$, reject all hypotheses $H_{j}$ for $j\in w(l)$ satisfying $t_{j}>c_{1-\alpha,w(l)}$.
If no null hypothesis is rejected, then stop. If some $H_{j}$'s are
rejected, let $w(l+1)$ be the subset of all null hypotheses $j$
among $w(l)$ that were not rejected and proceed to the next step.
\end{itemize}
Here, we obtain the {\em Bonferroni-Holm} procedure, suggested in \cite{H79}, by setting
$$
(BH)\qquad c_{1-\alpha,w} = \Phi^{-1}(1 - \alpha/|w|)
$$
for all  $w\subset\mathcal W=\{1,\dots,p\}$, where $|w|$ denotes the number of elements in $w$, and we obtain the {\em Romano-Wolf} procedure, suggested in \cite{RW05}, by setting
$$
(RW)\qquad c_{1 - \alpha,w} = (1 - \alpha)\text{ quantile of }\max_{j\in w}\frac{\sqrt n(\hat\theta^* - \hat\theta)}{(\En[\hat Z_{i j}^2])^{1/2}}\mid (\hat Z_i)_{i=1}^n,
$$
where $\hat\theta^*$ is a bootstrap version of $\hat\theta$. In what follows, we maintain that $\hat\theta^*$ is obtained with the Gaussian bootstrap, \eqref{eq: multiplier draw}, though the case of the empirical bootstrap can also be considered. To show that the Bonferroni-Holm and Romano-Wolf procedures have the strong FWER control \eqref{eq: strong control}, we will use the moderate deviation result in Theorem \ref{thm: moderate deviations for mam} and the high-dimensional CLT and bootstrap results in Theorems \ref{thm: clt for mam} and \ref{thm: bootstrap for mam}, respectively.

In \cite{RW05}, Romano and Wolf proved the following general result regarding the strong FWER control of the stepdown methods: If the critical values $c_{1-\alpha,w}$ satisfy
\begin{align}
&c_{1-\alpha,w^{\prime}}\leq c_{1-\alpha,w^{\prime\prime}} \quad \text{whenever $w^{\prime}\subset w^{\prime\prime}$ and}\label{eq: critical value property 1}\\
&\sup_{w\subset \mathcal{W}}\sup_{P\in\mathcal P^{w}}\Pr_P \left ( \max_{j\in w}t_{j}>c_{1-\alpha,w} \right ) \leq \alpha+o(1),\label{eq: conditions MHT}
\end{align}
then the stepdown method described above satisfies (\ref{eq: strong control}). Indeed, let $w$
be the set of true null hypotheses and suppose that the method rejects
at least one of these hypotheses. Let $l$ be the step when the method
rejects a true null hypothesis for the first time, and let $H_{j_0}$ be
this hypothesis. Clearly, we have $w(l)\supset w$. It then follows from \eqref{eq: critical value property 1} that
\begin{equation*}
\max_{j\in w}t_{j}\geq t_{j_0}>c_{1-\alpha,w(l)}\geq c_{1-\alpha,w}.
\end{equation*}
Combining these inequalities with (\ref{eq: conditions MHT}) yields (\ref{eq: strong control}).


We now formally establish strong FWER control of the three procedures described above.

\begin{theorem}[Strong FWER Control by Bonferroni and Bonferroni-Holm Procedures]\label{thm: fwer bonferroni}
Let $\mathcal P$ be a class of probability measures for the distribution of the data such that Conditions M and A are satisfied with the same $B_n$, $\delta_n$, and $\beta_n$ for all $P\in\mathcal P$ and assume that $(2\sqrt 2)^3 B_n \log^{3/2}(p n)/\sqrt n \leq \delta_n$ for all $n\geq 1$. Then both Bonferroni and Bonferroni-Holm procedures have the strong FWER control property \eqref{eq: strong control}.
\end{theorem}

\begin{theorem}[Strong FWER Control by Romano-Wolf Procedure]\label{thm: MHT}
Let $\mathcal P$ be a class of probability measures for the distribution of the data such that Conditions M, E, and A are satisfied with the same $B_n$, $\delta_n$, and $\beta_n$ for all $P\in\mathcal P$. Then the Romano-Wolf procedure with the Gaussian bootstrap critical values $c_{1-\alpha,w}$ described above has the strong FWER control property (\ref{eq: strong control}).
\end{theorem}

\begin{remark}[Comparison of Procedures with Strong FWER Control]
Clearly, the Bonferroni procedure is the first step of the Bonferroni-Holm procedure.  Thus, the set of hypotheses rejected by the Bonferroni-Holm procedure always contains all the hypotheses rejected by the Bonferroni procedure, which implies that the Bonferroni-Holm procedure is a more powerful version of the Bonferroni procedure. Since both procedures have strong FWER control, it follows that the former is preferable among two procedures. Regarding the comparison between the Bonferroni-Holm and Romano-Wolf procedures, both procedures have their own advantages. In particular, the Bonferroni-Holm procedure requires weaker conditions, as it does not require Condition E; but the Romano-Wolf procedure is typically more powerful because it uses non-conservative bootstrap critical values $c_{1-\alpha,w}$.\qed
\end{remark}

\subsection{Multiple Testing with FDR Control}
As in the previous subsection, here we are again interested in testing
\begin{equation}\label{eq: hypotheses fdr}
H_j:\theta_j\leq\bar\theta_{0j}\text{ against }H_j^\prime:\theta_j> \bar\theta_{0j}
\end{equation}
simultaneously for all $j = 1,\dots,p$. In this subsection, however, we seek a procedure with a different type of control. In particular, we look for procedures controlling the {\em the False Discovery Rate} (FDR).

When $p$ is large or very large, the requirement of controlling the FWER, which we had in the previous subsection, is often considered too stringent. This is because large values of $p$ lead to large critical values $c_{1-\alpha,\mathcal W}$ used by the Bonferroni-Holm and Romano-Wolf procedures in the first step; and, when $p$ is very large, it often turns out that these procedures simply fail to reject any null hypothesis. In \cite{BH95}, Benjamini and Hochberg offer an alternative, weaker, requirement: Instead of controlling the FWER, they suggest to control the False Discovery Rate (FDR), which is defined as the expected proportion of falsely rejected null hypotheses among all rejected null hypotheses. 

To define the FDR formally, let $\mathcal H_0\subset\{1,\dots,p\}$ be the set of indices $j$ corresponding to the true null hypotheses $H_j$, and suppose that we have a procedure for testing \eqref{eq: hypotheses fdr} simultaneously for all $j \in[p]$ that rejects null hypotheses $H_j$ with $j\in\mathcal H_R\subset\{1,\dots,p\}$. Thus, the total number of rejected null hypotheses is $|\mathcal H_R|$, and the total number of falsely rejected null hypotheses is $|\mathcal H_0\cap\mathcal H_R|$. We then can define the {\em False Discovery Proportion} (FDP) by
$$
FDP:=\frac{|\mathcal H_0\cap \mathcal H_R|}{|\mathcal H_R| \vee 1}.
$$
Here, we put $|\mathcal H_R\vee 1$ instead of $|\mathcal H_R|$ in the denominator to avoid division by zero since it may happen that all null hypotheses are accepted, so that $|\mathcal H_R| = 0$, and we want to define $FDP = 0$ when $|\mathcal H_R| = 0$. The FDR is then defined as the expected value of the FDP:
$$
FDR = \Ep[FDP].
$$
We would like to have a procedure with the FDR being as small as possible and with large power in terms of rejecting false null hypotheses; so, for $\alpha\in(0,1)$, we seek a procedure controlling the FDR in the sense that
$$
FDR \leq \alpha + o(1).
$$
Like in the previous subsection, we can also say that we have a strong control of FDR over a set $\mathcal P$ of probability measures for the distribution of the data corresponding to different data generating processes if
\begin{equation}\label{eq: fdr strong control}
\sup_{P\in\mathcal P}FDR_P \leq \alpha + o(1),
\end{equation}
where we use index $P$ to emphasize that the FDR depends on the distribution of the data represented by the probability measure $P$.

Next, we describe the {\em Benjamini-Hochberg} procedure, suggested in \cite{BH95}, for testing \eqref{eq: hypotheses fdr} simultaneously for all $j \in [p]$ with the FDR control. Recall the $t$-statistics $t_j$ defined in \eqref{eq: t statistics}, and let $t_{(1)} \geq \dots \geq t_{(p)}$ be the ordered sequence of $t_j$'s. Also, define $t_{(0)} := +\infty$. Then let
$$
\hat k := \max\Big\{ j=0,1,\dots,p\colon 1 - \Phi(t_{(j)}) \leq \alpha j / p \Big\}.
$$
The Benjamini-Hochberg procedure rejects all $H_{j}$ with $t_{j} \geq t_{(\hat k)}$. For simplicity, to prove that this procedure controls the FDR, we will assume in this subsection that the random vectors $Z_i$, $i\in[n]$, are i.i.d., all having the same distribution as that of some random vector $Z^0 = (Z_1^0,\dots,Z_p^0)$. In addition, we will impose the following conditions:

\medskip

\noindent \textbf {Condition C}. \textit{ For some $0<r<1$ and $0 < \rho < (1-r)/(1 + r)$, (i) the correlation between any two components of $Z^0$ is bounded in absolute value by $r$: $\max_{1\leq j < k \leq p}|\Ep[Z^0_j Z^0_k]| \leq r$, and (ii) for all $j \in [p]$, we have $|\Ep[Z^0_j Z^0_k]| \geq (\log p)^{-3}$ for at most $p^\rho$ different values of $k \in [p]$.}

\noindent \textbf {Condition F}. \textit{ (i) There exists $\mathcal H = \mathcal H_n \subset \{1,\dots,p\}$ such that $|\mathcal H| \geq \log\log p$ and $\sqrt n(\theta_{0 j} - \bar\theta_{0 j}) / V_{j j}^{1/2} \geq 2\sqrt{\log p} $ for all $j\in \mathcal H$; (ii) for some $\gamma\in(0,1)$, the number of false null hypotheses $H_j$, say $p_1$, satisfied $p_1\leq \gamma p$. }

\medskip

Conditions C and F are adapted from \cite{LS14}. Condition C restricts dependence between $Z_{1}^{0},\dots,Z_{p}^{0}$ and essentially corresponds to Condition (C1) in \cite{LS14}. Condition C implies that each $t$-statistic $t_j$ is highly correlated with at most $p^{\rho}$ other $t$-statistics and is only weakly correlated with all other $t$-statistics. Condition F is a technical condition that restricts the number of false null hypotheses to be not too small and not too large.

\begin{theorem}[Control of False Discovery Rate]\label{thm: fdp}
Let $\mathcal P$ be a class of probability measures for the distribution of the data such that Conditions M, A, C, and F are satisfied with the same $B_n$, $\delta_n$, $\beta_n$, $r$, $\rho$, and $\gamma$ for all $P\in\mathcal P$. Also, assume that $B_n = B$ for some constant $B$ (independent of $n$) and all $n\geq 1$ and that uniformly over $P\in\mathcal P$, $p \to\infty $ and $\log p = o(n^{\zeta})$ for some $\zeta < 3/23$ as $n\to\infty$. Then the Benjamini-Hochberg procedure has the FDR control property \eqref{eq: fdr strong control}.
\end{theorem}

Theorem \ref{thm: fdp} extends Theorem 4.1 in \cite{LS14} to allow for many {\em approximate} means. The proof of this theorem relies critically on the moderate deviation result in Theorem \ref{thm: moderate deviations for mam}. 
It is worth mentioning that \cite{LS14} finds a phase transition phenomenon that the Benjamini-Hochberg method can not control the FDR when $\log p \ge c_0 n^{1/3}$ for some constant $c_{0} > 0$; see Corollary 2.1 in \cite{LS14}. 



\subsection{Inference on Functionals of Many Approximate Means}
We next consider the problem of estimating linear functionals of the vector $\theta_0$, say $a'\theta_0$, where $a\in\mathbb R^p$ is a vector of loadings. These functionals are of interest in many settings. For instance, as we discussed in Example \ref{ex: linear regression introduction}
in the Introduction, the conditional average treatment effects may take the form of such functionals. We consider two cases separately: $\|a\|_1 = 1$ and $\|a\|_2 = 1$. As we will see, the analysis of the first case is straightforward and the results discussed so far immediately apply in this case. On the other hand, we will see that the second case is substantially more complicated, and treating it will require introducing some form of regularization. It is this second case which will help us to develop intuition for the results to be discussed in the second part of the chapter.

\subsubsection{Inference on Functionals of Many Means without Regularization}
Inference using the MAM framework extends to the case where we are interested in linear functionals of $\theta_0$ of the form
$$
a'\theta_0, \text{ with }  \|a\|_1 = 1.
$$
Indeed, by H\"{o}lder's inequality,
$$
\sup_{\|a\|_1 = 1} |a'(\hat \theta - \theta_0)| =  \| \hat \theta - \theta_0\|_\infty,
$$
which can be small even if $p$ is much larger than $n$. For instance, in the ``ideal noise
model'' where
\begin{equation}\label{eq:Normal}
\hat \theta - \theta_0 \sim N(0, I_p/n),
\end{equation}
with $I_p$ denoting the $p\times p$ identity matrix,
we have that
$$
 \| \hat \theta - \theta_0\|_\infty \lesssim_P \sqrt{\log p/n},
$$
so the error $\|\hat\theta - \theta_0\|_{\infty}$ is small if $\log p$ is much smaller than $n$.

Moreover, we have for a large collection of functionals $a_k'\theta_0$, indexed by $a_k$ with $\|a_k\|_1 =1$ and $k \in [p]$ that under \eqref{eq: vanishing approximation error},
$$
\sqrt n a_k'(\hat \theta - \theta_0) = \frac{1}{\sqrt{n}} \sum_{i=1}^n a_k'Z_{i} + o_P(1/\sqrt{\log(pn)}) \quad \text{ uniformly in $k \in [p]$. }
$$
It follows that $(a_k'\hat \theta_0)_{k=1}^p$ are approximate means themselves, so we can
use the ``many approximate means approach"  for the inference on these functionals as well.



\subsubsection{Inference on Functionals of Many Means with Regularization}\label{Sec: RegularizedMM}

Here we again consider inference on linear functionals $a'\theta_0$.  However, in contrast to the discussion above, we assume that $\|a \|_2 = 1$ instead of $\|a\|_1 = 1$. Using only the condition that $\|a\|_2 = 1$,
the previous ``reduction back to many means" is not possible. Indeed,
$$
\sup_{\|a\|_2 = 1}| a'(\hat \theta-\theta_0)| = \| \hat \theta - \theta_0\|_2,
$$
so estimating all such functionals uniformly well will require having $\| \hat \theta - \theta_0\|_2$
small, which is typically not possible. For example, even in the ideal noise model \eqref{eq:Normal}, we have
$$
\| \hat \theta - \theta_0\|_2^2 = \chi^2(p)/n = (p/n) (1 + O_P(1/\sqrt{p})),
$$
which diverges to infinity if $p$ increases more rapidly than $n$.

We therefore need to replace the estimator $\hat\theta$ by another estimator, say $\tilde\theta$, which performs favorably even when $p$ is large relative to $n$ in the sense that
$$
 \| \tilde \theta - \theta_0 \|^2_{2}  \text{ is small}.
$$
In order to achieve this goal, we will need two key ingredients.  First, we will need to assume that $\theta_0$ has some structure which can help to reduce the complexity of the estimation. Second, we will need to construct an estimator $\tilde\theta$ that is employing this structure to reduce the estimation error. In what follows, we shall rely on approximate sparsity as the structure on $\theta_0$ and make use of $\ell_1$-norm regularization to obtain the estimator $\tilde \theta$.

A simple structure that provides useful intuition is the {\em exact sparsity structure with separation from zero}. Under this structure, $\theta_0$ has $s$ large components of size bigger than the maximal estimation error $\rho = \| \hat \theta - \theta_0\|_\infty$ and all remaining components {\em exactly} equal to zero. If $\theta_0$ has this structure and the maximal estimation error $r$ is known, we can use a simple estimator $\tilde\theta = (\tilde\theta_1,\dots,\tilde\theta_p)'$ with
$$
\tilde\theta_{j} = \begin{cases}
\hat\theta_j, & \text{ if }|\hat\theta_j| > \rho,\\
0, & \text{ if }|\hat\theta_j|\leq \rho,
\end{cases}
$$
for all $j\in[p]$. This estimator might be referred to as a ``selection-based estimator.''  In practice, $\rho$ will typically be unknown, but its distribution can be estimated by the bootstrap, as discussed in Section \ref{MAM: CLT}.  We can then replace $\rho$ in the definition of $\tilde \theta$ by an estimate of the $(1 - \alpha)$-quantile of its distribution for some $\alpha = \alpha_n\to 0$. Figure \ref{FigureSelection} illustrates this selection-based estimator.

\begin{figure}[htbp]
\begin{center}
\includegraphics{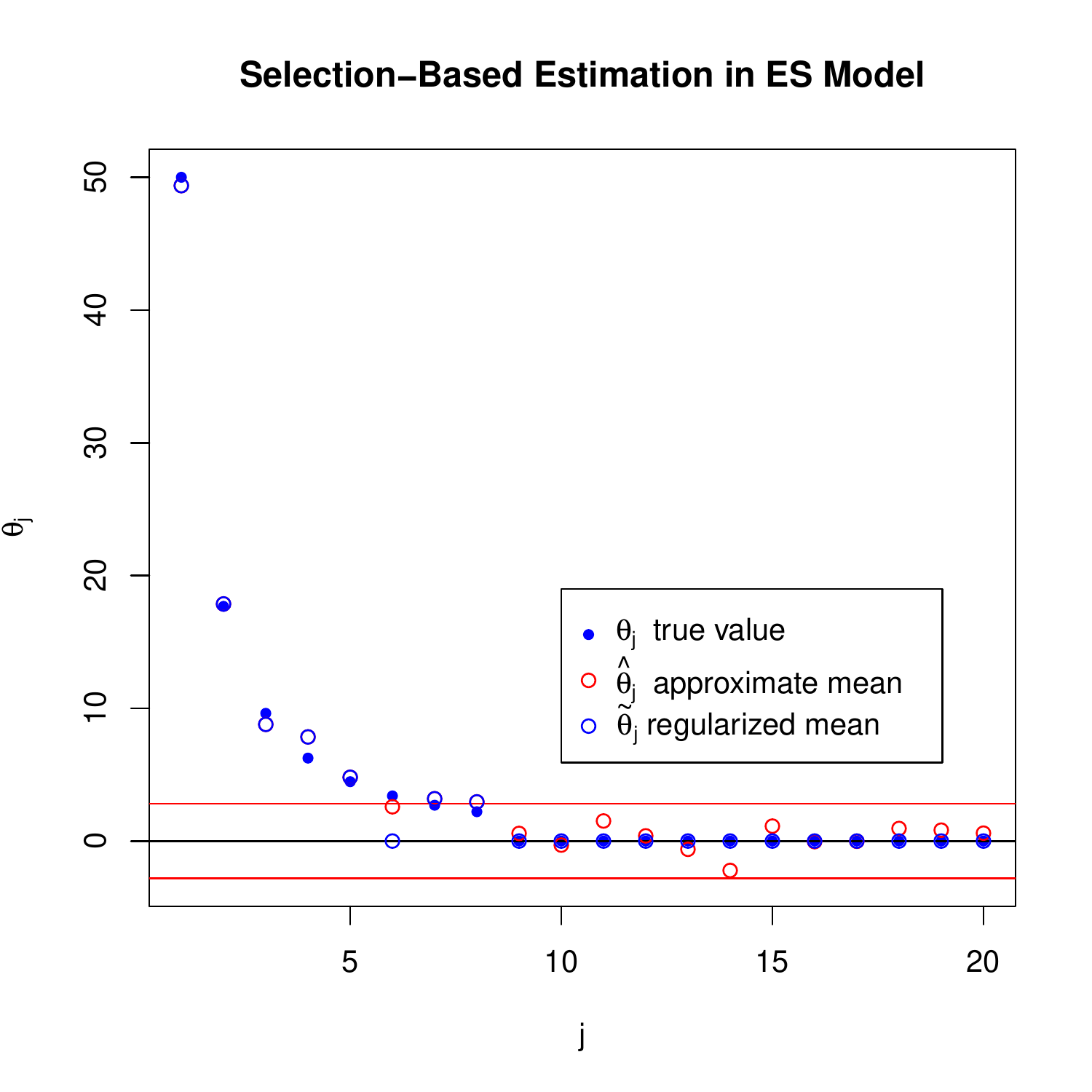}
\caption{Illustration of the Selection-Based Estimator in Exactly Sparse Models.
The true values $\theta_j$ are shown as blue points.  The approximate mean estimators
$\hat \theta_j$ are shown as red circles.  The selection-based mean estimators $\tilde \theta_j$ are shown
in blue circles.  Here $\theta_j = 50/j^{3/2}$ for $j=1,...,8$ and $\theta_j = 0$ for $j>8$, $\hat \theta_j = \theta_j + N(0,1)$, and $\tilde \theta_j = 1(|\hat \theta_j | > \lambda ) \hat \theta_j$ with $\lambda = \Phi^{-1}(1- \alpha/2p)$ with $\alpha =.1$. Note that the $\theta_j$'s do not magically stay sufficiently far from zero for model selection to work perfectly, invalidating the (rather incredible) perfect model selection story very frequently used in the classical and modern literature.  Note that the resulting selection-based estimator still performs well from the estimation point of view. Selection-based estimators of this kind have been known as post-model selection estimators.  These selection-based estimators indeed perform well in exactly and approximately sparse models as shown by Belloni and Chernozhukov \cite{BC13} who study the properties of the post-Lasso estimators in more general models.}
\label{FigureSelection}
\end{center}
\end{figure}

Assuming exact sparsity with separation structure provides a substantial dimension reduction that greatly eases the task of learning $\theta_0$ as long as $s$ is small.  For example, in the ideal noise model, letting $T\subset\{1,\dots,p\}$ denote the set of indices of non-zero components of the vector $\theta_0$, the use of $\tilde\theta$ defined in the previous paragraph under this structure leads to the following bound:
\begin{equation}\label{oracle}
\| \tilde \theta - \theta_0 \|_2 \leq  \sqrt{\sum_{j \in T} | \hat \theta_j - \theta_{0j} |^2}  \leq \sqrt{ \| N(0, I_s/n) \|^2}  = \sqrt{\chi^2(s)/n}  \lesssim_P  \sqrt{s /n},
\end{equation}
which can be small provided that $s$ is much smaller than $n$.  However, the exact sparsity with separation structure seems  unrealistic as a model for real econometric applications.  A structure in which all parameters are either exactly zero or magically align themselves to be larger in magnitude than $\rho$ seems extremely unintuitive and unlikely to correspond to most sensible economic models.  We will not work with this structure further, though we will consider an exact sparsity structure with {\em no separation} since it helps to convey some of the main ideas of the theory of high-dimensional estimation.

More generally, we will consider an {\em approximately} sparse structure where the coefficients, sorted in non-increasing order in terms of absolute size, smoothly decline in magnitude towards zero. For this structure, we expect to achieve a rate of convergence similar to that in \eqref{oracle}. Intuitively, under an approximately sparse structure, the unregularized estimator $\hat \theta$ still informs us about the components $\theta_{0 j}$ that can not be distinguished from zero.  We can thus use this information to set the regularized estimator $\tilde \theta_j$ for such components to zero. In such cases, we will be making an error by rounding those coefficients to zero, but the error will be negligible as long as the coefficients decrease to zero sufficiently quickly. We develop these results formally below.

For a given estimator $\hat\theta$, we consider the following $\ell_1$-regularization procedure:
\begin{equation}\label{eq: l1 regularization procedure}
\tilde \theta  =\arg\min_{\theta \in \Bbb{R}^p} \| \theta\|_1:  \quad  \| \hat \theta -
\theta \|_\infty \leq \lambda,
\end{equation}
where we minimize the $\ell_1$ norm of the coefficients subject to them deviating
from the initial estimates, $\hat\theta$, in the $\ell_\infty$ norm by at most $\lambda$.
In \eqref{eq: l1 regularization procedure}, $\lambda$ is the regularization parameter which controls the shrinkage in the estimator $\tilde\theta$.  At one extreme, setting $\lambda = 0$ results in no regularization and yields $\tilde \theta = \hat \theta$.  At the other extreme, setting $\lambda =\infty$ produces maximal regularization and results in $\tilde \theta = 0$.  In general, we aim to set $\lambda$ to be of the order of the estimation error $\| \hat \theta - \theta_0 \|_\infty$ as in \eqref{eq: lambda quantile} below. Figure \ref{FigureLasso} illustrates the regularized estimator $\tilde\theta$.

To establish properties of the estimator $\tilde\theta$ in \eqref{eq: l1 regularization procedure}, observe that the optimization problem in \eqref{eq: l1 regularization procedure} separates into $p$ independent problems: For each $j \in [p]$,
\begin{equation}\label{eq: one-dimensional reduction}
\tilde \theta_j = \min_{\theta_j \in \Bbb{R} } | \theta_j|:  \quad  | \hat \theta_j - \theta_j| \leq \lambda.
\end{equation}
It follows that the explicit solution $\tilde\theta$ of the optimization problem in \eqref{eq: l1 regularization procedure} is given by
$$
\tilde \theta_j = (|\hat \theta_j | - \lambda )_+ \text{sign} (\hat \theta_j),\quad j\in[p],
$$
where for any $a\in\mathbb R$, we use $(a)_{+}$ to denote $a1\{a>0\}$. This solution is known as the soft-thresholded estimator. In this simple setting, it also coincides with the well-known Lasso and Dantzig selector estimators.

\begin{figure}[htbp]
\begin{center}
\includegraphics{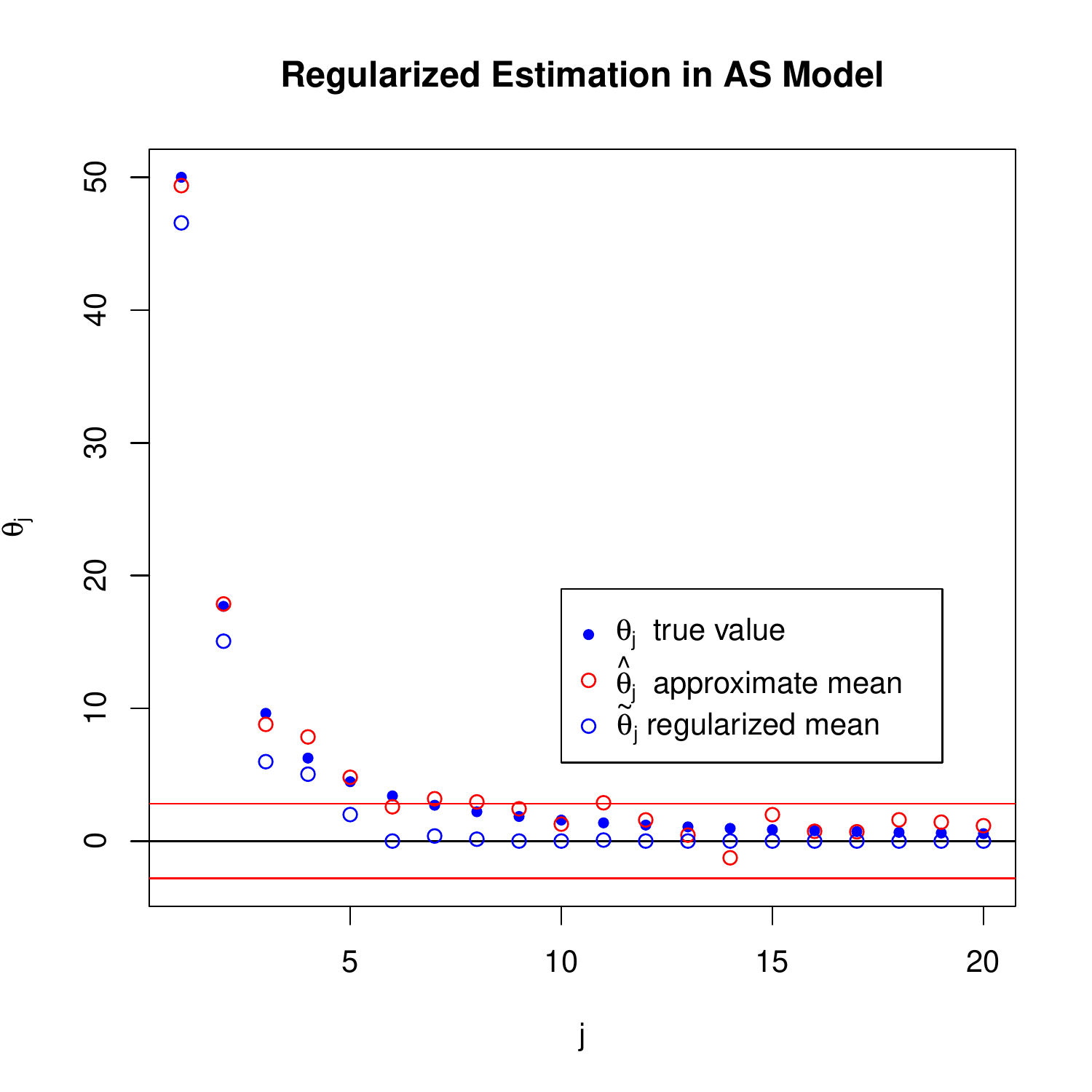}
\caption{Illustration of the Regularized Estimator in an Approximately Sparse Models.
The true values $\theta_j$ are shown as black points.  The approximate mean estimators
$\hat \theta_j$ are shown as red circles.  The $\ell_1$-regularized mean estimators $\tilde \theta_j$ are shown
in blue circles.  Here $\theta_j = 50/j^{3/2}$, $\hat \theta_j = \theta_j + N(0,1)$,
and $\tilde \theta_j = (|\hat \theta_j | - \lambda )_+ \text{sign} (\hat \theta_j)$ with $\lambda = \Phi^{-1}(1- \alpha/2p)$ with $\alpha =.1$.}
\label{FigureLasso}
\end{center}
\end{figure}

To carry out \eqref{eq: l1 regularization procedure}, we need to choose the regularization parameter $\lambda$. We will assume that $\lambda$ is chosen so that
\begin{equation}\label{eq: lambda quantile}
\lambda \geq  (1-\alpha)\text{-quantile of }
\| \hat \theta - \theta_0\|_\infty.
\end{equation}
As follows from the discussion above, we can approximate such a $\lambda$ either via self-normalized moderate deviations,
\begin{equation}\label{eq: lambda choice self normalization}
\lambda = n^{-1/2}\Phi^{-1}(1 - \alpha/2 p)\max_{j\in[p]}(\En[\hat Z_{i j}^2])^{1/2},
\end{equation} 
or via the bootstrap as outlined in Section \ref{MAM: CLT}. In the ideal noise model (\ref{eq:Normal}), we can also choose $\lambda$ as
\begin{equation}\label{eq: lambda choice ideal noise}
\lambda = n^{-1/2}\Phi^{-1}(1- \alpha/2p) \leq \sqrt{ 2 \log (2p/\alpha)/ n}.
\end{equation}

We analyze the estimator $\tilde\theta$ in \eqref{eq: l1 regularization procedure} under three different conditions.

\medskip

\noindent \textbf{Condition ES}. \textit{The parameter
$\theta_0$ is exactly sparse: There exists $T \subset \{1,...,p\}$ with cardinality $s$ such that $\theta_{0j} \neq 0$ only for $j\in T$. }

\noindent \textbf{Condition AS}.
\textit{ The parameter
$\theta_0$ is approximately sparse: For some $A>0$ and $a>1/2$, the non-increasing rearrangement $(|\theta_0|^*_{j})_{j\in[p]}$ of absolute values of coefficients  $(|\theta_{0j}|)_{j\in[p]}$ obeys}
$$
| \theta_0|^*_j \leq {A} j^{-{a}}, \quad j\in[p].
$$

\noindent \textbf{Condition DM}. 
\textit{ The parameter
$\theta_0$ has bounded $\ell_1$ norm: $\| \theta_0\|_1 \leq K$ for some $K>0$.}

\medskip

\begin{figure}[htbp]
\begin{center}
\includegraphics{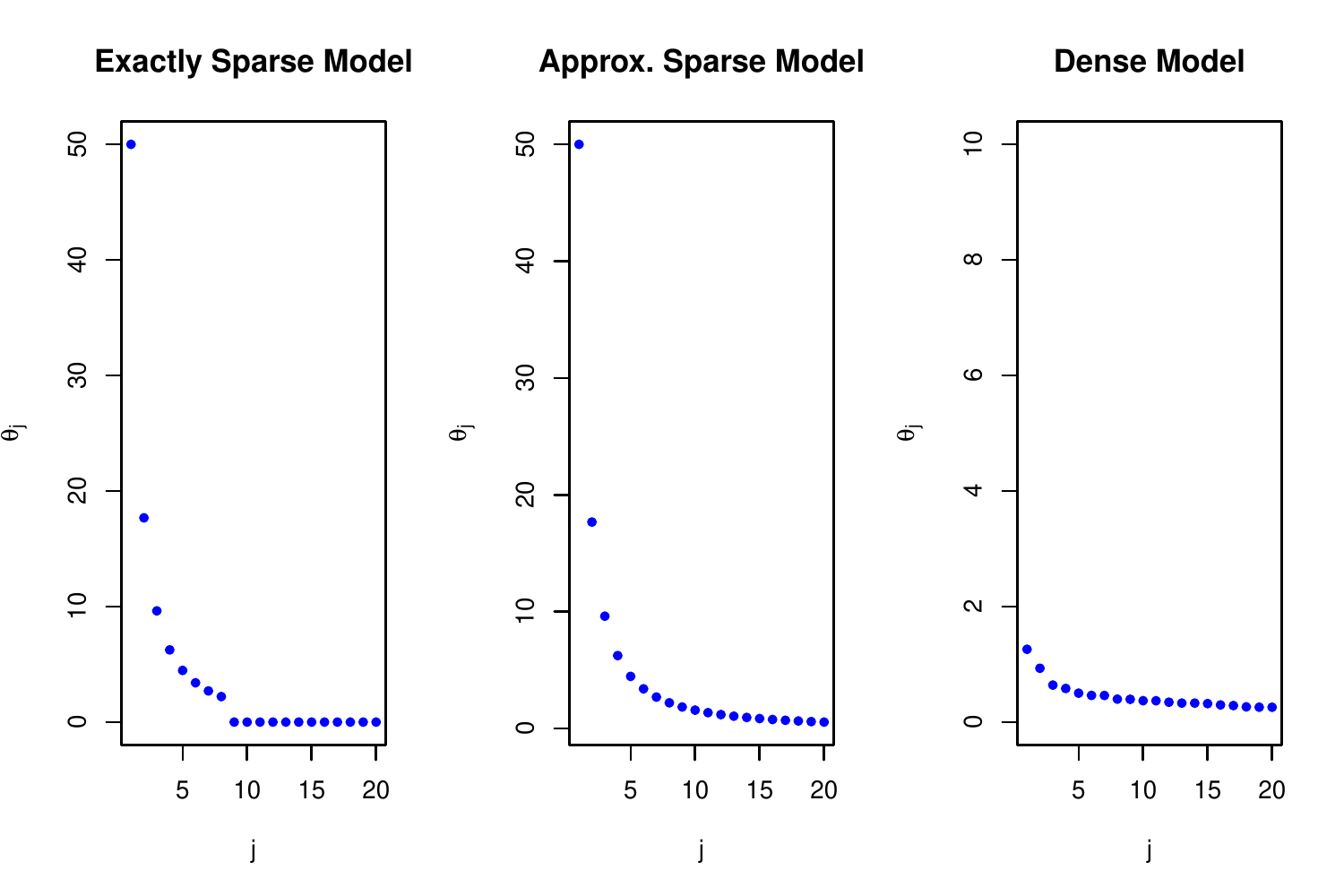}
\caption{A visual illustration of the ES, AS, and DM Models.
The true values $\theta_j$ are shown as blue points. In the Exactly Sparse (ES) Model, $\theta_j = 50/j^{3/2}$ for $j=1,...,8$ and $\theta_j = 0$ for $j>8$.  In the Approximately Sparse (AS) Model, $\theta_j = 10/j^{3/2}$.  In the Dense Model (DM), $\theta_j = 10/(2p)+ [10/(2p)]v_j$, where $\{v_j\}_{j=1}^p$ is a non-decreasing rearrangement
of $p$ i.i.d. draws of standard exponential variables and $ \| \theta \|_1  \leq 10$.} \label{FigureModels}
\end{center}
\end{figure}

Figure \ref{FigureModels} illustrates these conditions. Speaking informally,  ES can be thought of as a special case of AS, and AS can be thought of a special case of DM when $a>1$. More formally, Condition ES implies Condition AS with any $(a,A)$ such that $A\geq \|\theta_0\|_{\infty} s^a$; and, as long as $a>1$, Condition AS implies Condition DM with any $K\geq A a/(a - 1)$. While Conditions ES and AS require $\theta_0$ to be either sparse or approximately sparse, Condition DM allows $\theta_0$ to be ``dense" -- to have many elements that are all of similar, small size -- such that neither sparsity nor approximate sparsity holds. When working with Condition AS, it will be convenient to denote $s = \lceil(A/\lambda)^{1/a}\rceil$, which can be thought of as the "effective" dimension of the approximately sparse $\theta_0$.

Condition AS on approximate sparsity can also be compared with conditions typically imposed in the literature on nonparametric series estimation, e.g. \cite{N97}, where the {\em unordered} sequence of coefficients $(\theta_{0 j})_{j\in[p]}$ is often required to obey $|\theta_{0 j}| \leq A j^{-a}$, which is referred to as a smoothness condition. The approximate sparsity condition thus can be considered as a relaxation of the smoothness condition. 



We now establish a bound on the estimation error $\tilde\theta - \theta_0$ in the $\ell_q$ norm, where $q\geq 1$.

\begin{theorem}[Estimation Bounds for $\ell_1$-Regularized Many Means]\label{thm: estimation mm} Suppose that \eqref{eq: lambda quantile} holds. Then with probability
at least $1-\alpha$, we have $ | \tilde \theta_j| \leq | \theta_{0 j}|$ and $|\tilde \theta_j - \theta_{0 j} | \leq 2 \lambda$ for all $j\in[p]$ and
 \begin{itemize}
\item[(i)] under Condition DM:
$ \| \tilde \theta - \theta_0\|_q \leq 2 K^{1/q}\lambda^{1 - 1 / q}$ if $q\geq 1$;
\item[(ii)] under Condition ES:  $
\| \tilde \theta - \theta_0\|_q  \leq 2 s^{1/q} \lambda$ if $q\geq1$, and $\tilde \theta_{j} = 0$ for all $j \not \in T$;
\item[(iii)] under Condition AS:
$
\| \tilde \theta - \theta_0\|_q   \leq C_{a,q} s^{1/q} \lambda
$ if $q > a^{-1}$, where $s = \lceil(A/\lambda)^{1/a}\rceil$,
\end{itemize}
where $C_{a,q}$ is a constant depending only on $a$ and $q$.
\end{theorem}

\begin{corollary}[Estimation Bounds for $\ell_1$-Regularized Many Means]\label{cor: estimation mm} Suppose that \eqref{eq: lambda quantile} holds with $\alpha \to 0$ and that $\lambda \lesssim \sqrt{ \log p/n}$. 
Then
\begin{itemize}
\item[(i)] under Condition DM:
$ \| \tilde \theta - \theta_0\|_q \lesssim_P K^{1/q}(\log p/n)^{(q - 1)/(2q)}$ if $q\geq 1$;
\item[(ii)] under Condition ES:  $
\| \tilde \theta - \theta_0\|_q  \lesssim_P s^{1/q}\sqrt{\log p/n}$ if $q\geq 1$, and $\tilde \theta_{j} = 0$ for all $j\notin T$;

\item[(iii)] under Condition AS:
$
\| \tilde \theta - \theta_0\|_q  \lesssim_P C_{a,q}s^{1/q}\sqrt{\log p/n}
$
if $q > a^{-1}$, where $s = \lceil(A/\lambda)^{1/a}\rceil$,
\end{itemize}
where $C_{a,q}$ is a constant depending only on $a$ and $q$.
\end{corollary}

Note that the condition $\lambda \lesssim \sqrt{ \log p/n}$ in Corollary \ref{cor: estimation mm} can easily be satisfied in the ideal noise model and in many other models, as long as $\max_{j\in[p]}\En[\hat Z_{i j}^2] \lesssim 1$; see \eqref{eq: lambda choice self normalization} and \eqref{eq: lambda choice ideal noise}.

Corollary \ref{cor: estimation mm} illustrates the power of regularization in high-dimensional settings when
$\theta_0$ has some structure. In particular, Conditions ES and AS both imply that the estimation error of $\tilde\theta$ in the $\ell_2$ norm satisfies
$$
\| \tilde \theta- \theta_0 \|_2 \lesssim_P \sqrt{s \log p/n},
$$
where $s$ is the effective dimension in the case of Condition AS as long as $\lambda$ is chosen appropriately. Hence, the ambient dimension $p$ affects the rate only through a log factor, and the effective dimension appears in a very natural form through the ratio $\sqrt{s/n}$. Thus, under either of these two conditions, we have consistency in the $\ell_2$ norm as long as $s \log p/n$ tends to zero. Under Condition DM, the estimation error satisfies
$$
\|\tilde\theta - \theta_0\|_2 \lesssim_P \sqrt K(\log p/n)^{1/4},
$$
which tends to zero as long as $K^2\log p/n \to 0$.

\noindent
{\bf Proof of Theorem \ref{thm: estimation mm}.}  It follows from \eqref{eq: lambda quantile} that with probability at least $1-\alpha$, we have  $\| \hat \theta - \theta_0\|_{\infty} \leq \lambda$, in which case
\begin{equation}\label{eq: imply1}
| \tilde \theta_j| \leq | \theta_{0 j}|\text{ and }|\tilde \theta_j - \theta_{0 j} | \leq |\tilde \theta_j - \hat\theta_{j} | + |\hat \theta_j - \theta_{0 j} | \leq 2 \lambda,\quad\text{for all }j\in[p]
\end{equation}
by \eqref{eq: one-dimensional reduction}. This gives the first asserted claim.

To prove the second claim, we assume that \eqref{eq: imply1} holds. Then, under Condition DM, $\| \theta_0\|_1 \leq K$, and so
$$
\| \tilde \theta - \theta_0 \|_q^q \leq \|\tilde \theta - \theta_0\|_1 \|\tilde \theta - \theta_{0} \|^{q-1}_\infty \leq
(2 K) (2 \lambda)^{q-1} \leq 2^q K \lambda^{q - 1}
$$
by \eqref{eq: imply1}, which gives (i).  

Further, under Condition ES, \eqref{eq: imply1} implies that $|\tilde\theta_j| \leq |\theta_{0 j}| = 0$ for all $j\in T^c$, and so
$$
\| \tilde \theta - \theta_0 \|_q^q = \| (\tilde \theta - \theta_0)_T\|_q^q
\leq s \| \tilde \theta - \theta_0\|_\infty^q \leq s (2 \lambda)^q,
$$
which gives (ii).  

Finally, consider the case of AS and assume, without loss of generality, that components of $\theta_0$ are decreasing in absolute values, $|\theta_{ 0 1}| \geq \dots\geq |\theta_{0 p}|$, so that $|\theta_{0 j}| \leq A j^{-a}$ for all $j\in[p]$. Then, denoting $\bar T := \{ j \in [p]: A j^{-a} > \lambda \}$, it follows from the triangle inequality and \eqref{eq: imply1} that
\begin{equation}\label{eq: sparsity derivation under as 1}
\| \tilde \theta - \theta_0 \|_q \leq \| (\tilde \theta - \theta_0)_{\bar T} \|_q + \| (\tilde \theta - \theta_0)_{\bar{T}^c}\|_q \leq \| (\tilde \theta - \theta_0)_{\bar T} \|_q + 2 \| (\theta_0)_{\bar{T}^c}\|_q.
\end{equation}
Here, $ |\bar T| =  s -1 $, and so
\begin{equation}\label{eq: sparsity derivation under as 2}
\| (\tilde \theta - \theta_0)_{\bar T} \|_q < s^{1/q} \| \tilde \theta - \theta_0 \|_\infty \leq 2 s^{1/q}\lambda.
\end{equation}
Also,
\begin{equation}\label{eq: sparsity derivation under as}
\| (\theta_0)_{\bar{T}^c}\|_q^q
\leq \sum_{j=1}^p (A j^{-a} 1\{ A j^{-a}  \leq \lambda \})^q = A^q\sum_{j=1}^p j^{-a q}1\{ j \geq (A/\lambda)^{1/a} \} \leq \frac{2^{a q}s\lambda^q}{a q - 1},
\end{equation}
where the last inequality is established by replacing the sum by an integral, as formally shown in Lemma \ref{lem: summation-integration} in Appendix \ref{app: technical results}. Combining \eqref{eq: sparsity derivation under as 1}, \eqref{eq: sparsity derivation under as 2}, and \eqref{eq: sparsity derivation under as} gives (iii) 
and completes the proof of the theorem.
\qed

\section{Estimation with Many Parameters and Moments}\label{sec: many moments}

\subsection{Regularized Minimum Distance Estimation Problem}\label{Sec:RMD}
In this section, we consider the minimum distance estimation problem. We develop a Regularized Minimum Distance (RMD) estimator and study its properties.

Suppose that we have a target moment function $\theta \mapsto g(\theta)$,  mapping $ \Theta \subset \Bbb{R}^p$ to $\Bbb{R}^m$, and its empirical version $\theta \mapsto \hat g(\theta)$, also mapping $ \Theta \subset \Bbb{R}^p$ to $\Bbb{R}^m$, where both $p$ and $m$ may be large. Assume that $\theta_0$ is the unique solution of the following equation:
$$
g(\theta_0) = 0.
$$
We are interested in estimating $\theta_0$ using the empirical version $\theta\mapsto\hat g(\theta)$ of the function $\theta\mapsto g(\theta)$.

We define the RMD estimator $\hat \theta$ as a solution of the optimization problem
\begin{equation}\label{def:RMD}
\min_{\theta \in \Theta} \| \theta \|_1 : \| \hat g(\theta) \|_\infty \leq \lambda.
\end{equation}
where $\lambda$ is a regularization parameter. We will choose $\lambda$ so that a solution of \eqref{def:RMD} exists with large probability; and in the event that $\|\hat g(\theta)\|_{\infty} > \lambda$ for all $\theta\in\Theta$ so that the optimization problem \eqref{def:RMD} has no solution, we can set $\hat\theta$ to be equal to any particular element of $\Theta$. In the linear mean regression model, the RMD estimator reduces to the Dantzig Selector proposed by Cand\`{e}s and Tao in \cite{CT07}.

Let $\alpha\in(0,1)$ be a constant, which should be thought of as some small number. We will assume that the regularization parameter $\lambda$ satisfies the following condition:

\medskip

\noindent
{\bf Condition L}. \textit{The regularization parameter $\lambda$ is such that
\begin{equation}\label{def:lambda}
\|\hat g(\theta_0)\|_\infty \leq \lambda, \ \ \mbox{with probability at least $1-\alpha$.}
\end{equation}
}

When $\hat g_j(\theta_0)\sim N(0,\sigma^2/n)$ for all $j\in[m]$, setting $\lambda = n^{-1/2}\sigma\Phi^{-1}(1-\alpha/(2m))$ is sufficient to satisfy Condition L.  This Gaussianity of the moment conditions occurs, for example, in the high-dimensional linear regression model with homoscedastic Gaussian noise under appropriate normalization of the covariates; e.g. see \cite{BC11b}. More generally, we can use the self-normalization method or the bootstrap to approximate $\lambda$ as discussed in Section 2, provided that a preliminary estimator of $\theta_0$ is available.

The key consequence of Condition L is that $\theta_0$ is feasible in the optimization problem (\ref{def:RMD}) with probability at least $1 - \alpha$, in which case a solution $\hat\theta$ of this optimization problem exists and, by optimality, satisfies $\|\hat \theta\|_1\leq \|\theta_0\|_1$. As we discuss below, this property is crucial to handle high-dimensional models.

Denote
$$
\mathcal{R}(\theta_0) := \{\theta \in \Theta: \|\theta\|_1\leq \|\theta_0\|_1\},
$$
which we sometimes refer to as the restricted set. Also, let there be some sequences of positive constants $(\epsilon_n)_{n\geq 1}$ and $(\delta_n)_{n\geq 1}$ satisfying $\epsilon_n\searrow 0$ and $\delta_n\searrow 0$. To establish properties of the RMD estimator $\hat\theta$, we will use the following high-level conditions:

\medskip
\noindent
{\bf Condition EMC}. \textit{The empirical moment function concentrates around the target moment function: $$\sup_{\theta \in \mathcal{R}(\theta_0)}\| \hat g(\theta) - g(\theta)\|_\infty \leq \epsilon_n \text{  with probability at least } 1-\delta_n.$$}

\noindent
{\bf Condition MID}. \textit{The target moment function obeys the following identifiability condition: $$\{\| g(\theta) - g(\theta_0) \|_\infty  \leq \epsilon, \theta  \in \mathcal{R}(\theta_0)\}   \text{  implies  } \| \theta - \theta_0 \|_\ell \leq r( \epsilon; \theta_0, \ell), $$ for all $\epsilon>0$, where $\epsilon \mapsto r( \epsilon; \theta_0, \ell)$ is a weakly increasing rate function, mapping $[0,\infty)$ to $[0, \infty)$ and depending on the true value $\theta_0$ and the semi-norm of interest $\ell$.}
\medskip

Conditions EMC and MID encode the key blocks we need for the results. A large part of this section will be devoted to the verification of these conditions in particular examples. Condition EMC will be verified using empirical process methods, where contraction inequalities play a big role. Condition MID encodes both local and global identification of $\theta_0$. It states that if $g(\theta)$ is close to $g(\theta_0)$ in the $\ell_{\infty}$ norm and  $\theta$ is weakly smaller than $\theta_0$ in terms of the $\ell_1$ norm, then $\theta$ is close to $\theta_0$ in the semi-norm of interest $\ell$.  As we explain below, the validity of this condition depends on the interplay between the structure of $g$, the semi-norm $\ell$, and the structure of $\theta_0$. To appreciate the latter point, we note that if $\theta_0=0$, then Condition MID always holds with $r(\epsilon; 0, \ell) = 0$ for all semi-norms $\ell$. This means that $\theta_0 = 0$ is identified in the restricted set $\mathcal R(\theta_0)$ under no other assumptions on $g$.  The structure of $g$ and $\theta_0$ will begin to play an important role when $\theta_0 \neq 0$, as we discuss below.

Using these conditions, we can immediately obtain the following elementary but important result on the properties of the RMD estimator:

\begin{proposition}[Bounds on Estimation Error of RMD Estimator]\label{lemma:RMD} Assume that Conditions
L, EMC, and MID are satisfied for some semi-norm $\ell$. Then with probability at least $1- \alpha - \delta_n$, the RMD estimator $\hat \theta$ obeys
\begin{equation}\label{eq: simple bound for rmd}
\|\hat \theta - \theta_0\|_{\ell} \leq r( \lambda+\epsilon_n; \theta_0, \ell).
\end{equation}
\end{proposition}
\noindent
{\bf Proof of Proposition \ref{lemma:RMD}.} Consider the event that
$\lambda\geq \| \hat g(\theta_0)\|_\infty$, $\hat\theta\in\mathcal R(\theta_0)$, and $\| \hat g(\hat \theta) - g(\hat \theta)\|_\infty \leq \epsilon_n$. By the union bound and Condition EMC, this event occurs with probability at least $1-\alpha -\delta_n$ since Condition L implies that with probability at least $1 - \alpha$, we have $\lambda \geq \|\hat g(\theta_0)\|_{\infty}$ and $\hat\theta\in\mathcal R(\theta_0)$. On this event, we have, by the definition of the RMD estimator, $\|\hat g(\hat\theta)\|_{\infty} \leq \lambda$, and so
\begin{equation}\label{eq: simple bound for rmd derivation}
\| g(\hat \theta)\|_\infty \leq \| g(\hat \theta) - \hat g(\hat \theta)\|_\infty  + \|\hat g(\hat\theta)\|_{\infty} \leq \epsilon_n + \lambda
\end{equation}
by the triangle inequality. In turn, \eqref{eq: simple bound for rmd derivation} implies \eqref{eq: simple bound for rmd} via Condition MID since $g(\theta_0) = 0$, which gives the asserted claim. \qed

\begin{example}[Regularized GMM]\label{ex: RGMM}
An important special case of the minimum distance estimation problem is the Generalized Method of Moments (GMM) estimation problem, which corresponds to
$$
g(\theta) = A  \Ep g(X, \theta), \quad    \hat g(\theta) = \hat A \En g(X, \theta),
$$
where $(x, \theta) \mapsto g(x,\theta)$ is a measurable score function, mapping $\Bbb{R}^{d_x} \times \Theta$ to $\Bbb{R}^m$, $A$ is a positive definite weighting matrix, and $\hat A$ is an estimator of this matrix.  In practice, a simple choice of $A$ is the diagonal weighting matrix,
$$A^2 =   {\diag}({\text{Var}}( \En g(X, \tilde \theta)))^{-1},$$
where $\tilde \theta$ is a guess or a preliminary estimator of $\theta_0$. The analysis in this section will be suitable for the case where the weighting matrix is known, i.e. $\hat A = A$, and so, without much loss of generality, we set $A =I$ to simplify the exposition. We turn to other choices and estimation of $A$ in Section \ref{Sec:DRGMM}, where we consider optimal inference for individual components of $\theta_0$.
\qed
\end{example}

In what follows, let
$$
G:= (\partial/\partial \theta') g(\theta)|_{\theta = \theta_0}
$$
be the Jacobian matrix. This matrix plays an important role in encoding the local information about $\theta_0$. For convenience, for all $j\in[m]$, we will use $G_j(\theta)$ to denote the $j$th row of the matrix $G(\theta)$. Regarding the semi-norm $\ell$, we shall be focusing mainly on the $\ell_1$ norm $\|\cdot\|_1$, $\ell_2$ norm $\| \cdot \|_2$, and the Fisher norm $\| \cdot \|_{F}$, i.e. $\ell \in\{\ell_1,\ell_2,F\}$, where we define the Fisher norm by
$$
\| \delta \|_F := |\delta' (G' G)^{1/2} \delta|^{1/2},\quad \delta\in\mathbb R^p.
$$


To make Proposition \ref{lemma:RMD} operational, we need to verify Conditions EMC and MID and provide suitable choices of $\epsilon_n$, $\delta_n$, and the rate function $r$. We will do so separately in linear and non-linear models.

\subsubsection{Linear Case}
Here we consider the case where $\theta\mapsto g(\theta)$ is linear:
\begin{equation}\label{eq: linear case}
g(\theta) = G \theta + g(0).
\end{equation}
Two lead examples of this case are the linear mean regression model and the linear instrumental variables (IV) model:
\begin{example}[Linear regression models]\label{ex: linear regression models}
The linear mean regression model,
\begin{equation}\label{eq: linear regression model}
Y = W'\theta_0 + \epsilon,\quad \Ep \epsilon W=0,
\end{equation}
corresponds to \eqref{eq: linear case} with $g(\theta) = \Ep [ (Y- W'\theta) W]$, $G = - \Ep W W'$, and $g(0) = \Ep Y W$. The linear IV regression model,
\begin{equation}\label{eq: iv regression model}
Y = W'\theta_0 + \epsilon,\quad \Ep \epsilon Z = 0,
\end{equation}
corresponds to \eqref{eq: linear case} with $g(\theta) = \Ep [ (Y- W'\theta) Z]$, $G = - \Ep Z W'$,
and $g(0) = \Ep Y Z$. For both models, we will be able to show that the RMD estimator $\hat\theta$ has a fast rate of convergence as long as $\theta_0$ is either exactly sparse or approximately sparse under simple, intuitive conditions on the matrix $G$.\qed
\end{example}


We first discuss Condition MID. We consider three cases as in Section 2:
\begin{itemize}
\item exactly sparse model: $\theta_0$ obeys Condition ES;
\item approximately sparse model: $\theta_0$ obeys Condition AS;
\item dense model: $\theta_0$ obeys Condition DM.
\end{itemize}

In the dense model, where $\|\theta_0\|_1 \leq K$, we can immediately deduce that whenever $G$ is symmetric and non-negative definite, Condition MID holds with $\ell = F$ and
$$
r(\epsilon; \theta_0, F) \leq \sqrt{2K\epsilon}.
$$
Indeed, to prove this inequality, observe that for any $\theta\in\mathcal R(\theta_0)$ satisfying $\|g(\theta) - g(\theta_0)\|_{\infty}\leq \epsilon$, we have
\begin{align*}
\|\theta - \theta_0\|_{F}^2 &= |(\theta- \theta_0) 'G (\theta - \theta_0) | \leq  \|\theta- \theta_0\|_1 \| G (\theta - \theta_0)\|_{\infty} \\
&= \|\theta- \theta_0\|_1 \| g(\theta) - g(\theta_0)\|_{\infty}  \leq
2 K \epsilon,
\end{align*}
since $\theta\in\mathcal R(\theta_0)$ implies that $\|\theta - \theta_0\|_1 \leq \|\theta\|_1 + \|\theta_0\|_1 \leq 2\|\theta_0\|_1\leq 2K$ via the triangle inequality.  This result will imply ``slow" rates
of convergence of the estimator $\hat\theta$, though these slow rates will be sufficient in some applications.

In the exactly sparse or approximately sparse models, we proceed as follows. Define a modulus of continuity
$$
k(\theta_0, \ell) := \inf_{\theta \in \mathcal{R}(\theta_0): \| \theta - \theta_0\|_\ell >0 } \| G (\theta- \theta_0) \|_\infty/ \| \theta- \theta_0 \|_{\ell},
$$
which we can also call an identifiability factor. Whenever $k(\theta_0,\ell)\neq0$, we immediately obtain that Condition MID holds with
$$
r(\epsilon; \theta_0, \ell) \leq k(\theta_0, \ell)^{-1} \epsilon.
$$
There exist methods in the literature to show that $k(\theta_0,\ell)\neq0$ and to bound $k(\theta_0,\ell)$ from below in the exactly sparse model. When we consider the approximately sparse model, we will first sparsify $\theta_0$ to $\theta_0(\epsilon) = (\theta_{0 1}(\epsilon),\dots,\theta_{0 p}(\epsilon))'$ with
$$
\theta_{0 j}(\epsilon) := \begin{cases}
\theta_{0 j} + \text{sign}(\theta_{0 j})\Delta/(s - 1), & \text{ if }A j^{-a} > \epsilon,\\
0, & \text{ if }A j^{-a} \leq \epsilon,
\end{cases}
$$
where $s := \lceil (A/\epsilon)^{1/a}\rceil$ and $\Delta := \sum_{j=1}^p |\theta_{0 j}|1\{A j^{-a} \leq \epsilon\}$.  We will then provide a bound on $k^{-1}(\theta_0,\ell)\epsilon$ in terms of  $k^{-1}(\theta_0(\epsilon), \ell)\epsilon$
and the approximation error $\| \theta_0 - \theta_0(\epsilon)\|_\ell$.  Note that we assume that the components of $\theta_0$ are decreasing in absolute values, $|\theta_{ 0 1}| \geq \dots\geq |\theta_{0 p}|$ in this construction, which is without loss of generality because we do not use this information in the estimation.  We also inflate components of $\theta_{0 j}(\epsilon)$ with $A j^{-a} > \epsilon$  by using $\theta_{0 j} + \text{sign}(\theta_{0 j})\Delta/(s-1)$ instead of $\theta_{0 j}$ to make sure that $\theta\in\mathcal R(\theta_0)$ implies $\theta\in\mathcal R(\theta_0(\epsilon))$, which will be important in the verification of Condition MID.  Finally, we assume that $A > \epsilon$ to make sure that $s - 1 \geq 1$.

To present one possible lower bound for $k(\theta_0,\ell)$ in the exactly sparse model, we introduce some further notation. For $J\subset\{1,\dots,m\}$ and $H\subset\{1,\dots,p\}$, let $G_{J,H} = (G_{j,h})_{j\in J,h\in H}$ be the submatrix of $G$ consisting of all rows $j\in J$ and all columns $h\in H$ of $G$. Also, for an integer $l\geq s$, define the $l$-sparse smallest and $l$-sparse largest singular values of $G$ by
\begin{equation}\label{eq: sparse singular values def}
\sigma_{\min}(l):=\min_{|H|\leq l}\max_{|J|\leq l} \sigma_{\min}(G_{J,H}),  \quad \sigma_{\max}(l):=\max_{|H|\leq l}\max_{|J|\leq l} \sigma_{\max}(G_{J,H}),
\end{equation}
respectively, where $\sigma_{\min}(G_{J,H})$ and $\sigma_{\max}(G_{J,H})$ are the smallest and the largest singular values of the matrix $G_{J,H}$. We then have the following lower bound on $k(\theta_0,\ell)$, which is an immediate consequence of Theorem 1 in \cite{BCHN17}:  
\begin{lemma}[Lower Bound for $k(\theta_0,\ell)$ in Exactly Sparse Model]\label{lem: sparse singular values}
Under Condition ES, there exists a universal constant $c\geq 1$ such that if $\sigma_{\min}^2(l)/\sigma_{\max}^2(l) \geq c\mu_n$ and $\sigma_{\max}^2(l) \geq 1$ for $l = \lceil s/\mu_n\rceil$ and some $\mu_n\in(0,1/c)$, then
\begin{equation}\label{eq: basic bound for identifiability factor}
k(\theta_0,\ell_q) \geq s^{-1/q}\mu_n,\quad q\in\{1,2\}.
\end{equation}
\end{lemma}

In Lemma \ref{lem: sparse singular values}, the quantity $\mu_n$ appears as a key factor determining the modulus of continuity $k(\theta_0,\ell)$. When $\mu_n$ is bounded away from zero, we say that we have a strongly identified model. Otherwise, i.e. when $\mu_n>0$ is drifting towards zero, we say that we have a non-strongly identified model. The quantity $\mu_n$ in turn can be easily bounded in linear regression models:

\noindent
{\bf Example \ref{ex: linear regression models}} (Linear Regression Models, Continued). In the linear mean regression model \eqref{eq: linear regression model}, assume that all eigenvalues of $G = - \Ep W W'$ are bounded in absolute values from above and away from zero uniformly over $n$. Then for any integer $l\geq s$, we have
$$
\sigma_{\min}(l) = \min_{|H|\leq l}\max_{|J|\leq l} \sigma_{\min}(G_{J,H}) \geq \min_{|H|\leq l}\sigma_{\min}(G_{H,H}) \geq \sigma_{\min}(G)
$$
and, similarly,
$$
\sigma_{\max}(l) = \max_{|H|\leq l}\max_{|J|\leq l} \sigma_{\max}(G_{J,H})  \leq \sigma_{\max}(G)
$$
by the standard properties of eigenvalues of symmetric matrices. Hence, we can choose $\mu_n$ in Lemma \ref{lem: sparse singular values} to be bounded away from zero, which means that the linear mean regression model is strongly identified with $k(\theta_0,\ell_q)$ satisfying $k(\theta_0,\ell_q)^{-1} \leq C s^{1/q}$ for some constant $C>0$.

In the linear IV regression model \eqref{eq: iv regression model}, assume first that there exist constants $\bar\sigma \geq\underline\sigma > 0$ such that for all $s/\log n\leq l \leq s\log n$ and any combination of $l$ covariates $W_{H_1},\dots,W_{H_l}$ from the vector $W = (W_1,\dots,W_p)'$, there exists a combination of $l$ instruments $Z_{J_1},\dots,Z_{J_l}$ from the vector $Z = (Z_1,\dots,Z_m)'$ such that the matrix $\Ep[(Z_{J_1},\dots,Z_{J_l})'(W_{H_1},\dots,W_{H_l})]$ has singular values bounded from above by $\bar\sigma$ and from below by $\underline\sigma$, which means that $Z_{J_1},\dots,Z_{J_l}$ are strong instruments for the covariates $W_{H_1},\dots,W_{H_l}$. Then again we can choose $\mu_n$ in Lemma \ref{lem: sparse singular values} to be bounded away from zero, which again implies that we have a strongly identified model with $k(\theta_0,\ell_q)$ satisfying $k(\theta_0,\ell_q)^{-1} \leq C s^{1/q}$. On the other hand, if for some covariate $W_j$, we only have weak instruments, $\mu_n$ will drift towards zero, and we obtain a non-strongly identified model.\qed

In the analysis below, we use the bound \eqref{eq: basic bound for identifiability factor} as a starting point. In particular, letting there be a sequence of constants $(L_n)_{n\geq 1}$ satisfying $L_n \geq 1$ for all $n\geq 1$,  we impose the following assumption:

\medskip
\noindent
{\bf Condition LID}. \textit{For $q\in\{1,2\}$, either of the following conditions hold: (a) $\theta_0$ obeys Condition ES and  the bound $k (\theta_0, \ell_q)  \geq  s^{-1/q} \mu_n$ holds, or (b) $\theta_0$ obeys Condition AS, the bound $k (\theta_0(\epsilon), \ell_q)  \geq  s^{-1/q} \mu_n$  holds for
$s = \lceil (A/\epsilon)^{1/a} \rceil
$, and $\|G_j\|_1 \leq L_n$ for each $j\in[m]$, where $G_j$
denotes the $j$th row of $G$.}
\medskip

We use the notation LID as a shorthand for LID$(\theta_0,G)$ since the condition is indexed by both the parameter of interest $\theta_0$ and the matrix $G$. We will make the dependence explicit when we invoke the condition to other parameters. 

We can now verify Condition MID:

\begin{lemma}[Bounds on Rate Function in Linear Models]\label{ID:linear} (i) Under Condition LID(a), Condition MID holds with
$$
r(\epsilon; \theta_0, \ell_q) \leq   \epsilon s^{1/q} \mu_n^{-1},\quad q\in\{1,2\}.
$$
(ii)  Under Condition LID(b), Condition MID holds with
$$
r(\epsilon; \theta_0, \ell_q) \leq  C_{a,q}\Big( L_n \epsilon s^{1/q}\mu_n^{-1} + \epsilon s^{1/q} \Big),\quad q\in\{1,2\},
$$
as long as $a > 1$ and $A>\epsilon$, where $C_{a,q}$ is a constant depending only on $a$ and $q$.
\end{lemma}

\noindent
{\bf Proof of Lemma \ref{ID:linear}.} The first claim is immediate from the definition of the identifiability
factor $k(\theta_0,\ell_q)$. To show the second claim, assume that Condition AS holds, fix $q\in\{1,2\}$, and take any $\theta\in\mathcal R(\theta_0)$ such that $\|g(\theta) - g(\theta_0)\|_{\infty} = \|G(\theta - \theta_0)\|_{\infty} \leq \epsilon$. By the triangle inequality,
\begin{equation}\label{eq: verification of mid triangle}
\| \theta - \theta_0\|_q \leq \| \theta - \theta_0(\epsilon) \|_q + \| \theta_0(\epsilon) - \theta_0\|_q.
\end{equation}
Below, we bound the two terms on the right-hand side of this inequality.

To bound $\|\theta_0(\epsilon) - \theta_0\|_q$, we have
\begin{equation}\label{eq: delta bound}
\Delta = \sum_{j=1}^p |\theta_{0 j}| 1\{A j^{-a}\leq \epsilon\} \leq \sum_{j=1}^p A j^{-a} 1\{A j^{-a}\leq \epsilon\} \leq \frac{2^a \epsilon s}{a - 1}
\end{equation}
by Condition AS and Lemma \ref{lem: summation-integration} in Appendix \ref{app: technical results}. Thus,
\begin{align*}
\|\theta_0(\epsilon) - \theta_0\|_q^q &\leq \sum_{j=1}^p (\Delta/(s - 1))^q 1\{j < (A/\epsilon)^{1/a}\} + \sum_{j = 1}^p (A j^{-a})^q 1\{j\geq (A/\epsilon)^{1/a}\} \\
& \leq (s - 1)(\Delta/(s - 1))^q + \frac{2^{a q} \epsilon^q s}{a q - 1} \leq C_{a,q}^q \epsilon^q s.
\end{align*}
To bound $\| \theta - \theta_0(\epsilon) \|_q$, we have
\begin{align*}
\| G (\theta- \theta_0(\epsilon))\|_\infty
&\leq \| G (\theta- \theta_0)\|_\infty + \| G (\theta_0 - \theta_0(\epsilon))\|_\infty\\
&\leq \epsilon + \| G (\theta_0 - \theta_0(\epsilon))\|_\infty
\leq  \epsilon + \max_{j\in[m]} \| G_j\|_1 \| \theta_0 - \theta_0(\epsilon)\|_\infty \leq C_{a,q} L_n \epsilon,
\end{align*}
where the last inequality follows from \eqref{eq: delta bound} and Condition AS.
Therefore, since $\theta\in\mathcal R(\theta_0)$ implies $\theta\in\mathcal R(\theta_0(\epsilon))$, we have by Condition LID that
$$
\| \theta - \theta_0(\epsilon) \|_q  \leq C_{a,q} L_n\epsilon  s^{1/q} \mu_n^{-1}.
$$
Combining the bounds on $\|\theta_0(\epsilon) - \theta_0\|_q$ and $\| \theta - \theta_0(\epsilon) \|_q$ above with \eqref{eq: verification of mid triangle} gives the second asserted claim.
\qed
\begin{remark}[Bounding $k(\theta_0,F)$ via $k(\theta_0,\ell_1)$]
Note that once we have a lower bound on $k(\theta_0,\ell_1)$, we can always use it to obtain a lower bound on $k(\theta_0,F)$, whenever $G$ is symmetric. Indeed, for any $\theta\in\mathbb R^p$ such that $\theta \neq \theta_0$, we have
$$
(\theta - \theta_0)'G(\theta - \theta_0) \leq \|\theta - \theta_0\|_1 \|G(\theta - \theta_0)\|_{\infty}.
$$
Rearranging this expression gives
$$
\frac{\|G(\theta - \theta_0)\|_{\infty}}{\sqrt{(\theta - \theta_0)'G(\theta - \theta_0)}}\geq \sqrt{\frac{\| G(\theta - \theta_0) \|_{\infty}}{\|\theta - \theta_0\|_1}},
$$
which implies that $k(\theta_0,F) \geq k^{1/2}(\theta_0,\ell_1)$.\qed
\end{remark}

We next verify Condition EMC. Let there be a sequence of constants $(\ell_n)_{n\geq 1}$ satisfying $\ell_n\geq 1$ for all $n \geq 1$. Consider the following condition:

\medskip
\noindent
{\bf Condition ELM}.  {\em The empirical moment function is linear, $$\hat g(\theta) = \hat G \theta + \hat g(0),$$
and with probability at least $1-\delta_n$, we have
$$
\max_{j\in[m]} \|\hat G_{j} - G_{j} \|_{\infty} \vee \| \hat g(0) - g(0) \|_{\infty} \leq \ell_{n}/\sqrt n.
$$
}

\noindent
{\bf Example \ref{ex: RGMM}} (Regularized GMM, Continued).
When specialized to GMM problems, Condition ELM means that
the score function is linear,
$$
g(X, \theta) = G(X) \theta + g(X,0),
$$
where $G(X):=(\partial/\partial\theta')g(X,\theta)|_{\theta = \theta_0}$, and with probability at least $1-\delta_n$, we have
$$
\max_{j\in[m]} \|\Bbb{G}_n G_{j}(X)\|_{\infty} \vee \| \Bbb{G}_n g(X, 0)\|_{\infty} \leq \ell_{n}.
$$
This form will be useful below to verify Condition ELM in the linear regression models.
\qed

Condition ELM is plausible.  It is implied by many sufficient conditions based on self-normalized moderate deviations and high-dimensional central limit theorems, as reviewed in Section 2.
In particular, it is possible to choose
\begin{equation}\label{eq: ell typical rate}
\ell_n \propto \sqrt{ \log (p \vee m \vee n)}.
\end{equation}
in many cases, as illustrated in the case of linear regression models below. We highlight the slow growth of $\ell_n$ with respect to the number of parameters $p$ and the number moment functions $m$, which is critical to allow the analysis to handle high-dimensional models.

\begin{lemma}[Empirical Moment Concentration]\label{lemma:EMC1} If Conditions DM and ELM are satisfied, then Condition EMC holds with
$$
\epsilon_n = n^{-1/2} \ell_n(K+1).
$$
\end{lemma}
\noindent
{\bf Proof of Lemma \ref{lemma:EMC1}.}
Conditions DM and ELM imply that with probability at least $1 - \delta_n$,
\begin{align*}
&\sup_{\theta \in \mathcal R(\theta_0)}\| \hat g(\theta) - g(\theta)\|_\infty \leq
\sup_{\theta \in \mathcal R(\theta_0)} \| (\hat G - G)\theta + \hat g(0) - g(0) \|_\infty\\
&\qquad \leq  \max_{j\in[m]}  \|\hat G_{j} - G_j\|_{\infty} \| \theta_0\|_1 + \| \hat g(0) - g(0)\|_{\infty} \leq n^{-1/2} \ell_n (K+1),
\end{align*}
which gives the asserted claim. \qed

\noindent
{\bf Example \ref{ex: linear regression models}} (Linear Regression Models, Continued). Here, we verify Condition ELM for linear regression models under primitive conditions. Since the mean regression model is a special case of the IV regression model, we only consider the latter. In the linear IV model, we have $g(0)=\Ep[Y Z]$ and $G=-\Ep[Z W']$. Suppose that $Z \in \mathbb{R}^m$ and $W\in \mathbb{R}^p$ are such that the following moment condition holds for some $\sigma > 0$:
$$
\max_{j\in[m],k\in[p]} \Big(\Ep[Y^4]+\Ep[Z_j^4]+\Ep[W_k^4]\Big) \leq \sigma^2.
$$
Also, suppose that
$$
\Ep\left[\|(Y,Z',W')'\|_\infty^{8}\right] \leq M_n^{4}
$$
for some $M_n$, possibly growing to infinity, but such that $n^{-1/2} M_n^2\log(p\vee m \vee n)\leq \sigma^2$.
Let $(Y_i,Z_i,W_i)_{i=1}^n$ be a random sample from the distribution of $(Y,Z,W)$. Then by H\"{o}lder's inequality,
$$
\sum_{i=1}^n \Ep[(Y_i Z_{i j})^2] \leq \sum_{i=1}^n (\Ep[Y_i^4]\Ep[Z_{i j}^4])^{1/2} \leq n\sigma^2
$$
for all $j\in[m]$. Hence, by Lemma \ref{lem: maximal ineq},
\begin{align*}
\Ep\left[\max_{j\in[p]}\Big|\sum_{i=1}^n (Y_i Z_{i j} - \Ep[Y_i Z_{i j}]) \Big|\right]
&\leq A\left( \sigma\sqrt{n\log p} + \sqrt{\Ep\left[\max_{i\in[n]} \|Y_iZ_i\|_{\infty}^2 \right]} \log p \right)\\
&\leq A\left( \sigma\sqrt{n \log p} + n^{1/4} M_n \log p \right)\leq 2A\sigma\sqrt{n \log p}
\end{align*}
for some universal constant $A>0$. Therefore, applying Lemma \ref{lem: fuk-nagaev} with $s = 4$, $t = \sigma\sqrt{3n\log n}$, and $\sigma^2$ replaced by $n\sigma^2$ shows that there exist universal constants $c,C > 0$ such that with probability at least $1 - c/\log^2 n$,
$$
\|\mathbb G_n g(X,0)\|_{\infty} \leq C\sigma\sqrt{\log(p\vee n)},
$$
By the same argument, again with probability at least $1 - c/\log^2n$, we also have
$$
\max_{j\in[m]}\|\mathbb G_n G_j(X)\|_{\infty} \leq C\sigma\sqrt{\log(p\vee m\vee n)}.
$$
Condition ELM thus holds with
$\delta_n = 2c/\log^2 n$ and
$$
\ell_n = C\sigma\sqrt{\log(p\vee m\vee n)},
$$
which is in accord with \eqref{eq: ell typical rate}.
\qed

Summarizing the results in Proposition \ref{lemma:RMD}, Lemma \ref{ID:linear}, and Lemma \ref{lemma:EMC1}, we obtain the following theorem:

\begin{theorem}[Bounds on Estimation Error of RMD Estimator, Linear Case]\label{thm:BoundLinearRGMM}
In the linear case \eqref{eq: linear case}, assume that Conditions L, DM, LID, and ELM are satisfied and $\lambda \leq \ell_n/\sqrt n$. Then with probability at least $1 - \alpha - \delta_n$,
\begin{equation}\label{eq: linear bound part 1}
\|\hat\theta - \theta_0\|_q \leq \frac{\ell_n(K+2) s^{1/q}}{\mu_n \sqrt n},\quad q\in\{1,2\}
\end{equation}
in the case of LID(a) (exactly sparse model). Similarly, as long as $a > 1$ and $A > n^{-1/2}\ell_n(K+1)$,
\begin{equation}\label{eq: linear bound part 2}
\|\hat\theta - \theta_0\|_q \leq \frac{C_{a,q}(L_n + \mu_n)\ell_n(K+2) s^{1/q}}{\mu_n \sqrt n},\quad q\in\{1,2\}
\end{equation}
in the case of LID(b) (approximately sparse model), where $C_{a,q}$ is a constant depending only on $a$ and $q$. 
\end{theorem}

Theorem \ref{thm:BoundLinearRGMM} implies $\ell_q$-rates of convergence of the RMD estimator and characterizes how the sparsity of $\theta_0$ impacts these rates. The impact of the overall number of coefficients is bounded by the factor $\ell_n$ which  typically grows logarithmically with the number of coefficients $p$ and the number of moment conditions $m$. This is effectively the impact of not knowing the support of $\theta_0$. This implies that we can achieve consistent estimators even if $(p\vee m)\gg n$ since $\ell_n$ can grow much slower than root-$n$.

\begin{remark}[Improving Theorem \ref{thm:BoundLinearRGMM}]
The right-hand sides of the bounds in Theorem \ref{thm:BoundLinearRGMM} depend linearly on $K$, which may be suboptimal if $K$ is increasing with $n$. This dependence can be improved by reiterating the argument in Proposition \ref{lemma:RMD} one more time. In particular, we can introduce a doubly restricted set
$$
\bar{\mathcal R}(\theta_0):=\{\theta\in\Theta\colon \|\theta\|_1 \leq \|\theta_0\|_1\text{ and }\|\theta - \theta_0\|_1 \leq \gamma_n\},
$$ 
where $\gamma_n$ denotes the right-hand side of either \eqref{eq: linear bound part 1} or \eqref{eq: linear bound part 2}, depending on which part of Condition LID is imposed. Then by Theorem \ref{thm:BoundLinearRGMM} and Condition L, $\hat\theta\in\bar{\mathcal R}(\theta_0)$ with probability at least $1 - \alpha - \delta_n$, and we can replace the set $\mathcal R(\theta_0)$ in the proof of Proposition \ref{lemma:RMD} by $\bar{\mathcal R}(\theta_0)$. In turn, modifying slightly the proof of Lemma \ref{lemma:EMC1}, which provides the value of $\epsilon_n$ to be used in the proof of Proposition \ref{lemma:RMD}, we can write
\begin{align*}
&\epsilon_n = \sup_{\theta \in \bar{\mathcal R}(\theta_0)}\| \hat g(\theta) - g(\theta)\|_\infty \leq
\sup_{\theta \in \bar{\mathcal R}(\theta_0)} \| (\hat G - G)\theta_0 + (\hat G - G)(\theta - \theta_0) + \hat g(0) - g(0) \|_\infty\\
&\qquad \leq \|(\hat G - G)\theta_0\|_{\infty} +  \max_{j\in[m]}  \|\hat G_{j} - G_j\|_{\infty} \gamma_n + \| \hat g(0) - g(0)\|_{\infty},
\end{align*}
which is typically of order $n^{-1/2}\ell_n$ (as long as $\gamma_n\to 0$ and $\|(\hat G - G)\theta_0\|_{\infty}$ is of order $n^{-1/2}\ell_n$) and thus can be substantially smaller than $\epsilon_n = n^{-1/2} \ell_n (K+1)$ derived in the proof of Lemma \ref{lemma:EMC1}. Using these arguments in the proof of Proposition \ref{lemma:RMD} may lead to improved bounds in Theorem \ref{thm:BoundLinearRGMM} but we omit the formal statements for brevity of the chapter.
\qed
\end{remark}

\subsubsection{Restricted Non-Linear Case}\label{Sec:RestrictedNonLinearCase} We would like to find
some useful conditions for nonlinear models, where the rates
of convergence of the RMD estimator will be similar to what we have in the linear case.

\medskip
\noindent
{\bf Condition NLID}. \textit{Assume that Condition LID holds and that the target moment function $g$ satisfies a restricted non-linearity condition around $\theta = \theta_0$; namely
$$
\{\| g(\theta) - g(\theta_0) \|_\infty  \leq \epsilon, \theta  \in \mathcal{R}(\theta_0)\}   \text{  implies  } \| G (\theta- \theta_0) \|_\infty/2 \leq \epsilon
$$
for all $\epsilon \leq \epsilon^*$, where $\epsilon^*$
 is a tolerance parameter, measuring the degree of the linearity
 of the problem, with $\epsilon^* = \infty$ in the linear case.}
\medskip

This condition is the gradient version of restricted convexity for $\ell_1$-penalized M-estimators in \cite{BC11} and \cite{NRWY14}.

\begin{example} As a reference case, we can take nonlinear models with $$\epsilon^* = c \mu^2_n/( s^2)$$ for some constant $c>0$. This case arises, for example, in nonlinear moment condition models of the sort $\Ep \rho(Y, W'\theta_0) Z  = 0$, where $ \rho: \Bbb{R}^2 \to \Bbb{R}$ is a well-behaved residual function (e.g. as would arise in nonlinear regression and nonlinear IV regression). This nonlinearity creates an additional requirement on the effective sparsity of $\theta_0$, namely $$s^2 \ll n/\mu_n,$$ which arises when we analyze estimation.   \qed
\end{example}

\begin{lemma}[Bounds on Rate Function in Non-Linear Models]\label{ID:non-linear} (i) Under Conditions LID(a) and NLID, Condition MID holds for all $\epsilon \leq \epsilon^*$ with
$$
r(\epsilon; \theta_0, \ell_q) \leq  2 \epsilon s^{1/q} \mu_n^{-1},\quad q\in\{1,2\}.
$$
(ii) Under Conditions LID(b) and NLID, Condition MID holds for all $\epsilon \leq \epsilon^*$ with
$$
r(\epsilon; \theta_0, \ell_q) \leq  C_{a,q}\Big( L_n \epsilon s^{1/q}\mu_n^{-1} + \epsilon s^{1/q} \Big),\quad q\in\{1,2\}
$$
as long as $a > 1$ and $A>\epsilon$, where $C_{a,q}$ is a constant depending only on $a$ and $q$.
\end{lemma}

This lemma can be proven using the same argument as that leading to Lemma \ref{ID:linear}, so we omit the proof.

Next, let there be sequences of positive constants $(B_{1 n})_{n\geq 1}$ and $(B_{2 n})_{n\geq 1}$.
We consider an example of a sufficient condition that allows us to bound the empirical error in estimating $\theta_0$.

\medskip
\noindent
{\bf Condition ENM}. \textit{(i) The target and empirical moment functions have the form $g(\theta) = \Ep[g(X,\theta)]$ and $\hat g(\theta) = \En[g(X,\theta)]$, respectively, where $g(X,\theta) = (g_1(X,\theta),\dots,g_m(X,\theta))'$ is a vector of score functions, corresponding to the RGMM example. (ii) The score functions have the index form:
\begin{equation}\label{eq: UbarSingleIndices nonlinear case}
 g_j(X,\theta) = \tilde g_j(X, Z_{\uu(j)}(X)'\vartheta_{\uu(j)}),\quad j\in[m], \ \theta= (\vartheta_1',\dots,\vartheta_{\bar u}')'\in\mathbb R^p,
\end{equation}
where $\tilde g_j$ is a measurable map from $\Bbb{R}^{d_x} \times \Bbb{R}$ to $\Bbb{R}$ for all $j\in[m]$, $Z_u$ is a measurable map from $\mathbb R^{d_x}$ to $\mathbb R^{p_u}$ for all $u\in[\bar u]$, $u(j)\in[\bar u]$ for all $j\in[m]$, $\vartheta_u$ is $p_u$-dimensional subvector of $\theta$ for all $u\in[\bar u]$, and $p = p_1 + \dots + p_{\bar u}$. (iii) The score
functions  are Lipschitz in the second argument, namely
\begin{equation}\label{eq: lipschitz nonlinear case}
|\tilde g_j(X,v)- \tilde g_j(X,\tilde v)| \leq L_j(X)|v-\tilde v|,  \text{ for all }   (v,\tilde v) \in \mathbb R^2,
\end{equation}
with probability one, where $L_j$ is a measurable map from $\mathbb R^{d_x}$ to $\mathbb R_{+}$
for all $j\in [m]$. (iv) Finally, we have
\begin{equation}\label{eq: variance nonlinear case}
{\rm Var}(g_j(X,\theta)- g_j(X,\theta_0)) \leq B_{1 n}^2
\end{equation}
for all $\theta\in\mathcal R(\theta_0)$ and $j\in[m]$ and with probability at least $1 - \delta_n/6$,
\begin{equation}\label{eq: moments nonlinear case}
 \max_{j\in[m],k\in[p_{u(j)}]} \En[L_{j}^2(X)Z_{\uu(j)k}^2(X)]  \leq B^2_{2 n}, \  \|n^{-1/2}\Gn(g(X,\theta_0))\|_\infty \leq n^{-1/2} \ell_n.
 \end{equation}}
\medskip

In many applications, the number of indices $Z_\uu(X)'\vartheta_\uu$ is small.  As examples, the common class of single-index models clearly use just one index, $Z_\uu(X)'\vartheta_ \uu = Z(X)' \theta$; and we would have two indexes - the supply index $Z_1(X)'\vartheta_1$ and the demand index $Z_2(X)'\vartheta_2$ - if we estimate linear supply and demand equations.  The Lipschitz condition allows us to use Ledoux-Talagrand type contraction inequalities for bounding the error. Condition ENM, though plausible in a number of applications, is strong.  Its chief appeal is in immediately providing useful bounds on the empirical error. One could also obtain bounds on the empirical error through the use of maximal inequalities that carefully exploit the geometry and entropy properties of the set of functions $\{g(X,\theta)\colon \theta  \in \mathcal{R}(\theta_0)\}$.

\begin{lemma}[Empirical Moment Concentration]\label{lemma:EMC} Assume that Conditions DM and ENM are satisfied. Then Condition EMC holds with
$$
\epsilon_n = n^{-1/2} (\tilde \ell_n +  \ell_n),
$$
where $\tilde\ell_n = C(B_{1 n} + B_{2 n} K \log^{1/2}(p m/\delta_n))$ and $C$ is a universal constant.
\end{lemma}

Next, we demonstrate how Condition ENM can be verified in particular examples. Specifically, we consider logistic regression and nonlinear IV regression models.

\begin{example}[Logistic Regression Model]\label{Ex:Logistic}
Let $Y\in\{0,1\}$ be a binary outcome of interest and $W = (W_1,\dots,W_p)'\in \mathbb R^p$ a vector of covariates linked by a logistic model, namely
$$
\Ep[Y\mid W] = \Lambda(W'\theta_0),\quad \Lambda(t) = \exp(t)/\{1+\exp(t)\},\quad t\in\mathbb R.
$$
The vector of score functions associated with this model is
$$
g(X, \theta) =W(Y- \Lambda(W'\theta)),\quad X = (Y,W')',
$$
and so $g_j(X,\theta) = \tilde g_j(X,W'\theta)$, where
$$
\tilde g_j(X, t) = W_j(Y - \Lambda(t)),\quad t\in\mathbb R, \ j\in[p].
$$
Suppose that for some $\sigma > 0$,
\begin{equation}\label{eq: conditions for logistic model}
\max_{j\in[p]}\Ep[W_j^4] \leq \sigma^4,\ \frac{\log(p n)}{\sqrt n}(\Ep[\|W\|_{\infty}^4])^{1/2} \leq \sigma^2, \ \frac{\log(p n)}{\sqrt n}(\Ep[\|W\|_{\infty}^8])^{1/2} \leq \sigma^4.
\end{equation}
Then \eqref{eq: lipschitz nonlinear case} holds with $L_j(X)=|W_j|$ since $\Lambda$ is $1$-Lipschitz. Moreover, since $|\Lambda(t) - \Lambda(\tilde t)|\leq 1$ for any $(t,\tilde t)\in[0,1]^2$,
$$
{\rm Var}(g_j(X,\theta)- g_j(X,\theta_0)) \leq \Ep[W_j^2] \leq \sigma^2
$$
for any $\theta\in\mathbb R^p$ by the first inequality in \eqref{eq: conditions for logistic model}. Therefore, \eqref{eq: variance nonlinear case} holds with any $B_{1 n}^2 \geq \sigma^2$. Also, by the first and third inequalities in \eqref{eq: conditions for logistic model},
it follows from Lemmas \ref{lem: deviation inequality for positive rvs} and \ref{lem: maximal inequality for positive rvs}, where the latter is applied with $s = 2$ and $t = n\sigma^4$, that
$$
\max_{j\in[p], \ k\in[p]}\En[W_j^2 W_k^2] \leq \max_{j\in[p]}\En[W_j^4] \leq C_1 \sigma^4
$$
with probability at least $1 - c_1/\log^2 n$, where $c_1,C_1$ are some universal constants. Finally, by the same arguments as those in Example \ref{ex: linear regression models}, it follows from the first and second inequalities in \eqref{eq: conditions for logistic model} that
$$
\|\mathbb G_n(g(X,\theta_0))\|_{\infty} \leq C_2\sigma\sqrt{\log(p n)}
$$
with probability at least $1 - c_2/\log^2 n$, where $c_2,C_2$ are some universal constants. Conclude, by the union bound, that \eqref{eq: moments nonlinear case} holds probability at least $1 - \delta_n/6$ as long as we set
\begin{align*}
&B_{2 n}^2 \geq C_1\sigma^4, \ \delta_n = 6(c_1 + c_2)/\log^2 n, \ \ell_n = C_2\sigma\sqrt{\log(p n)}.
\end{align*}
We can thus establish all requirements of Condition ENM.
\qed
\end{example}

\noindent
{\bf Example \ref{Ex:NL-IV}} (Nonlinear IV Regression Model, Continued).
Consider the model
$$
\Ep[f(Y,W'\theta_0)\mid Z] = 0,
$$
where $Y\in\mathbb R$ is an outcome variable, $W\in\mathbb R^p$ is a vector of endogenous covariates, $Z\in\mathbb R^m$ is a vector of instruments, $f\colon \mathbb R^2\to\mathbb R$ is some known function, and $\theta_0\in\mathbb R^p$ is a vector of parameters of interest. A vector of score functions associated with this model is
$$
g(X,\theta) = Z f(Y,W'\theta),\quad X = (Y,W',Z')',
$$
and so $g_j(X,\theta) = \tilde g_j(X,\theta)$, where
$$
\tilde g_j(X, t) = Z_j f(Y, t),\quad t\in\mathbb R, \ j\in[m].
$$
Suppose that the function $f\colon\mathbb R^2\to\mathbb R$ is Lipschitz in its second argument:
$$
|f(Y,t) - f(Y,\tilde t)| \leq \gamma(Y)|t - \tilde t|,\quad \text{for all }t,\tilde t\in\mathbb R,
$$
with probability one. Suppose also that for some $\sigma > 0$,
\begin{align}
&\max_{j\in[m]}\Ep[Z_j^2 f^2(Y,W'\theta_0)] \leq \sigma^2, \ \frac{\log(m n)}{\sqrt n}(\Ep[\|Z\|_{\infty}^4 f^4(Y,W'\theta_0)])^{1/2} \leq \sigma^2,\label{eq: nonlinear iv condition 1}\\
&\max_{j\in[m], \ k\in[p]}\Ep[\gamma^2(Y)Z_j^2W_k^2] \leq \sigma^4, \ \frac{\log(p m n)}{\sqrt n}(\Ep[\gamma^4(Y)\|Z\|_{\infty}^4\|W\|_{\infty}^4])^{1/2} \leq \sigma^4.\label{eq: nonlinear iv condition 2}
\end{align}
Then \eqref{eq: lipschitz nonlinear case} holds with $L_j(X) = Z_j\gamma(Y)$ for all $j\in[m]$. Also, for any $j\in[m]$ and $\theta\in\mathcal R(\theta_0)$,
\begin{align*}
|g_j(X,\theta) - g_j(X,\theta_0)|
&\leq \gamma(Y)|Z_j W'(\theta - \theta_0)|\\
&\leq \gamma(Y)|Z_j|\|W\|_{\infty} \|\theta - \theta_0\|_1
\leq 2\gamma(Y)|Z_j|\|W\|_{\infty} \| \theta_0\|_1,
\end{align*}
and so \eqref{eq: variance nonlinear case} holds for all
\begin{equation}\label{eq: bn bound nonlinear case}
B_{1 n}^2\geq 4\|\theta_0\|_1^2 \max_{j\in[m]}\Ep[\gamma^2(Y)|Z_j|^2\|W\|_{\infty}^2].
\end{equation}
Further, like in Example \ref{Ex:Logistic}, by \eqref{eq: nonlinear iv condition 2}, it follows from Lemmas \ref{lem: deviation inequality for positive rvs} and \ref{lem: maximal inequality for positive rvs} that
$$
\max_{j\in[m], \ k\in[p]}\En[\gamma^2(Y)Z_j^2W_k^2] \leq C_1\sigma^4
$$
with probability at least $1 - c_1/\log^2(n)$, and by \eqref{eq: nonlinear iv condition 1}, it follows from Lemmas \ref{lem: fuk-nagaev} and \ref{lem: maximal ineq} that
$$
\|\mathbb G_n(g(X,\theta_0))\|_{\infty} \leq C_2 \sigma \sqrt{\log(m n)}
$$
with probability at least $1 - c_2/\log^2(n)$, where $c_1$, $C_1$, $c_2$ and $C_2$ are universal constants. Conclude, by the union bound, that \eqref{eq: moments nonlinear case} holds with probability at least $1 - \delta_n/6$ as long as we set
$$
B_{2 n}^2 \geq C_1\sigma^4, \ \delta_n = 6(c_1 + c_2)/\log^2 n, \ \ell_n = C_2\sigma\sqrt{\log(m n)}.
$$
We have thus verified all assumptions of Condition ENM. Note also that the conditions we give here are sufficient but sometimes are not necessary. For example, if we assume that the function $f\colon \mathbb R^2\to\mathbb R$ is bounded in absolute value by a constant $C$, then \eqref{eq: variance nonlinear case} holds for all
$$
B_n^2 \geq 4C\max_{j\in[m]}\Ep[Z_j^2].
$$
Depending on the setting, this bound can be better than \eqref{eq: bn bound nonlinear case}.
\qed

The following result is an immediate corollary of Proposition \ref{lemma:RMD} and Lemmas \ref{ID:non-linear} and \ref{lemma:EMC}.

\begin{theorem}[Bounds on Empirical Error for Non-Linear RGMM]\label{thm:NonLinear}
Consider the non-linear case and assume that Conditions L, DM, LID, NLID, and ENM are satisfied. Also, assume that $\lambda$ is chosen so that $\lambda \leq n^{-1/2}(\tilde\ell_n + \ell_n)$ and that the side condition $n^{-1/2}(\tilde\ell_n + \ell_n) \leq \epsilon^* /2$ holds. Then with probability at least $1 - \alpha - \delta_n$,
$$
\|\hat \theta - \theta_0 \|_q \leq  \frac{2(\tilde\ell_n + \ell_n)s^{1/q}}{\mu_n \sqrt n},\quad q\in\{1,2\},
$$
in the case of LID(a) (exactly sparse model); and, as long as $a > 1$, and $A > n^{-1/2}\ell_n(K + 1)$,
$$
\|\hat \theta - \theta_0 \|_q \leq \frac{C_{a,q}(L_n + \mu_n)(\tilde\ell_n + \ell_n)s^{1/q}}{\mu_n \sqrt n},\quad q\in\{1,2\}
$$
in the case of LID(b) (approximately sparse model), where $C_{a,q}$ is a constant depending only on $a$ and $q$.
\end{theorem}


Theorem \ref{thm:NonLinear} shows that, under sparsity conditions, the RGMM estimator can be consistent for $\theta_0$ in the $\ell_q$-norm in nonlinear models. Importantly, the dependence of the convergence rates on the total number of parameters and moment conditions is controlled by $\ell_n$ and $\tilde\ell_n$, which typically grow logarithmically with $p$ and $m$ as shown in Examples \ref{Ex:Logistic} and  \ref{Ex:NL-IV}. Thus consistency is possible even for high-dimensional models when the number of parameters exceeds the sample size. These results also highlight the different rates of convergence for different norms of interest. In particular, the RGMM estimator has good rates of convergence in the  $\ell_1$ and $\ell_2$ norms. However, the rate of convergence of the RGMM estimator in the max-norm is not optimal in many cases of interest; and additional tools, and estimators, are needed to obtain good estimators for that case.

\subsection{Double/De-Biased RGMM}\label{Sec:DRGMM}


Section \ref{Sec:RMD} focused on rates of convergence of regularized minimum distance estimators. We now turn to providing inferential statements for parameters of interest where we will leverage the obtained rates of convergence of the RGMM estimator. Our development will emphasize inference via the Neyman orthogonality principle. Specifically, we aim to construct moment equations $M(\alpha; \eta) = 0$ for the target parameters $\alpha \in \Bbb{R}^{p_1}$ given the nuisance parameter $\eta \in \Bbb{R}^p$ such that the true value $\alpha_0$ of the parameter $\alpha$ obeys
\begin{equation}\label{neyman1}
M(\alpha_0; \eta_0) = 0,
\end{equation}
where $\eta_0$ is the true value of the nuisance parameter and such that the equations are first-order insensitive to local perturbations of the nuisance parameter $\eta$ around the true value:
\begin{equation}\label{neyman2}
\partial_{\eta'} M(\alpha_0, \eta) \Big |_{\eta = \eta_0} = 0.
\end{equation}
We refer to the latter property as the Neyman orthogonality condition.

Given $M$, estimation of or inference about $\alpha_0$ will then be based on some estimator $\hat M$
of $M$, where we plug-in a potentially biased, regularized estimator $\hat \eta$ in place of the unknown
$\eta_0$. The role of the Neyman orthogonality condition, (\ref{neyman2}), is precisely to mitigate the impact of the use of such biased estimators (or other non-regular estimators) of $\eta_0$ on the estimation of $\alpha_0$. In many settings, basing estimation and inference for $\alpha_0$ on estimating equations with the Neyman orthogonality property will allow us
\begin{itemize}
\item to construct a high-quality estimator of $\alpha_0$ with an approximately Gaussian distribution or
\item to carry out high-quality inference on $\alpha_0$ via hypothesis testing (in cases where $\alpha_0$ is weakly identified).
\end{itemize}

\subsection{Construction of Approximate Mean Estimators for $\theta$ via Orthogonal Moments}\label{Sec: Neyman Estimation}

Here we take the target parameter and the nuisance parameter
to be the same, namely
$$
\alpha = \theta, \quad \eta = \theta.
$$
The estimator for the nuisance parameter will be $\hat \eta = \hat \theta$, the RGMM estimator from
the previous subsections.

We can construct the Neyman orthogonal equations $M(\alpha; \eta)$ for the pair $(\alpha, \eta)$ as follows. First we define an optimal moment selection matrix:
$$
\gamma_0 = G ' \Omega^{-1},
$$where  $\Omega = \Ep g(X,\theta_0)g(X,\theta_0)'$ and $G=\left.(\partial/\partial \theta')g(\theta)\right|_{\theta=\theta_0}$.
This $p \times m$ moment selection matrix can be used to collapse $m$-dimensional moment equations $g(\theta_0) = 0$
to $p$-dimensional moment equations $\gamma_0 g(\theta_0) = 0$.  We will focus on the optimal moment selection matrix, although, in principle, sub-optimal moment selection matrices could be used in practice as well.  For example, estimation of the $m$ by $m$ matrix $\Omega$ and its inverse might be a limiting factor in some settings, and one might consider using ${\rm diag}(\Omega)$ instead as its inverse is trivial to compute.

Given the moment selection matrix, we define
\begin{equation} \label{neyman:GMM}
M(\alpha; \eta) = \gamma_0 G (\alpha - \eta) + \gamma_0 g(\eta).
\end{equation}
We then have that at the true parameter values 
$$
M(\theta_0; \theta_0) = 0,
$$
and the Neyman orthogonality condition holds:
$$
\partial_{\eta'} M(\theta_0; \eta) \Big |_{\eta = \theta_0} = - \gamma_0 G + \gamma_0 G = 0.
$$

Heuristically, if we somehow knew $\gamma_0$, $G$, and $\eta= \theta_0$,
we could define an "oracle" linear estimator of $\alpha_0$, $\bar \theta$, as the root of
$$
\bar M(\bar\theta; \theta_0) = \gamma_0 G (\bar \theta - \theta_0) + \gamma_0 \hat g( \theta_0) = 0;
$$
that is
$$
\sqrt{n}(\bar \theta - \theta_0) = - (\gamma_0 G)^{-1} \gamma_0 \sqrt{n}\hat g( \theta_0).
$$
This estimator is linear, so it obeys
$$
\sqrt{n}(\bar \theta - \theta_0) \approx_d  N(0, V), \quad   V = (G'\Omega^{-1} G)^{-1},
$$
over the sets in $\mathcal{A}$, the class of all rectangles in $\mathbb{R}^p$, under the CLT conditions of Section 2. The variance matrix $V$ here is the optimal variance matrix for GMM.

The above construction is infeasible because of the many unknowns, including the true value of the target parameter, appearing in it.  To make construction feasible, we can plug-in estimators corresponding to the unknowns:
instead:\begin{enumerate}
\item plug-in the RGMM estimator $\hat \theta$ in place
of $\theta_0$,
\item  plug-in an estimator  $\hat G = \partial_{\theta'} \hat g(\hat \theta)$ for $G$ (or a regularized version);
\item  plug-in an estimator  $\hat \Omega = \En g(X,\hat \theta) g(X,\hat \theta)'$ for $\Omega$ (or a regularized version);
\item plug-in
a regularized estimator $\hat \gamma$ for $\gamma_0 := G'\Omega^{-1}$, such that $\hat \gamma$ is well-behaved;
\item plug-in a regularized estimator $\hat \mu$ for $\mu_0 := (\gamma_0 G)^{-1}$, such that $\hat \mu$ is well-behaved.

\end{enumerate}

Specific choices of estimators $\hat \mu$ and $\hat \gamma$ will be discussed later.  Given estimators of all unknown quantities, we define a ``two-step" estimator of the target parameter by solving the estimated Neyman-orthogonal equation:
$$
\hat M(\theta; \hat \theta) =  \hat \mu^{-1} (\theta - \hat \theta) + \hat \gamma \hat g(\hat \theta) = 0.
$$
The resulting solution to this equation, $\check \theta$, provides an estimator of the target parameter which we refer to as the double/debiased regularized GMM (DRGMM) estimator:
\begin{equation}\label{DRGMM}
\sqrt{n}(\check \theta - \hat \theta) = - \hat \mu \hat \gamma \sqrt{n}\hat g( \hat \theta) \ \ \mbox{or equivalently} \ \ \check \theta = \hat \theta -  \hat \mu \hat \gamma \hat g( \hat \theta)
\end{equation}

By exploiting the Neyman orthogonality property and further assumptions on the problem,
we can show that this estimator approximates the infeasible ``oracle" estimator defined above in the sense that
\begin{eqnarray*}
\sqrt{n}(\check \theta - \theta_0) & =  & \sqrt{n}(\bar \theta - \theta_0)  + o_P(1/\sqrt{\log p}).
\end{eqnarray*}
Hence, the DRGMM estimator is also approximately linear and is therefore an approximate mean, so we have
\begin{equation}\label{DRGMM:N}
\sqrt{n}(\check \theta - \theta_0) \approx_d  N(0, V), \quad   V = (G'\Omega^{-1} G)^{-1},
\end{equation}
over the class of all rectangles in $\mathbb{R}^p$ under the conditions of Theorem 2.1 in Section 2. Indeed, given this construction, we are back to the MAM framework.  We can thus use the inferential tools from Section 2 for immediate construction of simultaneous confidence bands and hypothesis testing with control of FWER or FDR.  Inference done in this way will be optimal in the sense that the variance matrix $V$ can not be generally improved by using any other moment selection matrix $\bar \gamma$ in place of $\gamma_0$.  Optimality may also be attained in other semi-parametric senses, which we do not discuss.

\subsection{Testing Parameters $\alpha$ with Nuisance Parameters $\eta = \theta$
via Neyman-Orthogonal Scores}

Here we take the target parameter, $\alpha$, and the nuisance parameter, $\eta$,
to be the different:
$$
\alpha = \alpha, \quad \eta = \theta.
$$
The true value of the parameter is given by $(\alpha_0', \eta_0')'$ and solves
$$
g(\alpha_0, \theta_0) = 0.
$$
Here, we are thinking of a situation where $\eta_0$ is strongly identified and can be well-estimated
by RGMM while $\alpha_0$ is only weakly or partially identified.  We would thus like to use a robust
testing approach to test values of $\alpha_0$ and then invert to construct a confidence set for $\alpha_0$.

The estimator for the nuisance parameter will be $\hat \eta = \hat \theta$, the RGMM estimator from
Section \ref{Sec:RMD}.  We can construct the Neyman orthogonal equations $M(\alpha; \eta)$ for the pair $(\alpha, \eta)$ as follows. First, we define a $p' \times m$ moment selection matrix for $\alpha_0$,
$$
\xi_0, \text{ e.g. } \xi_0 = I \text{ or } \xi_0 = G_\alpha \Omega^{-1}, \quad  G_\alpha = \partial_{\alpha'} g(\alpha_0, \theta_0),
$$
where $p' \geq \dim(\alpha)$. The latter matrix will be optimal when $\alpha_0$ is strongly identified.

Given the moment selection matrix, we define
\begin{equation} \label{neyman:GMM}
M(\alpha; \theta) =   (\xi_0 -   \xi_0 G ( G' \Omega^{-1} G )^{-1} G \Omega^{-1})  g(\alpha, \theta) =
  (\xi_0 -   \xi_0 G \mu_0 \gamma_0) g(\alpha, \theta_0)
\end{equation}
as in \cite{CHS}.
Using this estimating equation, we have that, at the true values of the parameters,
$$
M(\alpha_0; \theta_0) = 0
$$
using $g(\alpha_0, \theta_0) = 0$ and that the Neyman orthogonality condition holds:
$$
\partial_{\eta'} M(\theta_0; \theta) \Big |_{\theta= \theta_0} =  (\xi_0 -   \xi_0 G ( G' \Omega^{-1} G )^{-1} G \Omega^{-1}) G = 0.
$$

Heuristically, if we knew $\theta_0$ and all the extra parameters used in forming $M$, we could use the oracle
Neyman-orthogonal score for testing $\alpha_0$:
$$
\sqrt{ n} \bar M(\alpha_0; \theta_0) =  (\xi_0 -   \xi_0 G \mu_0 \gamma_0) \sqrt{n}\hat g(\alpha_0, \theta_0).
$$
This quantity is clearly linear in $\sqrt{n}\hat g(\alpha_0, \theta_0)$, so it obeys
$$
\sqrt{ n} \bar M(\alpha_0; \theta_0) \approx_d  N(0, V_M), \quad   V_M =  (\xi_0 -   \xi_0 G \mu_0 \gamma_0)
\Omega (\xi_0 -   \xi_0 G \mu_0 \gamma_0)',
$$
over the sets in $\mathcal{A}$, the class of all rectangles in $\mathbb{R}^p$, under the CLT conditions of Section 2.

The above construction is again clearly infeasible.  As outlined in Section \ref{Sec: Neyman Estimation}, we can make the construction feasible by plugging in estimators for the various missing unknowns.  Specifically, we will
\begin{enumerate}
\item plug-in the RGMM estimator $\hat \theta$ in place
of $\theta_0$,
\item  plug-in an estimator  $\hat G = \partial_{\theta'} \hat g(\hat \theta)$ for $G$ (or a regularized version);
\item  plug-in an estimator  $\hat \Omega = \En g(X,\hat \theta) g(X,\hat \theta)'$ for $\Omega$ (or a regularized version);
\item plug-in
a regularized estimator $\hat \gamma$ for $\gamma_0 := G'\Omega^{-1}$, such that $\hat \gamma$ is well-behaved;
\item plug-in a regularized estimator $\hat \mu$ for $\mu_0 := (\gamma_0 G)^{-1}$, such that $\hat \mu$ is well-behaved;
\item plug-in a regularized estimator $\hat \xi$ for $\xi_0$.
\end{enumerate}
We discuss estimation of $\hat \mu$, $\hat \gamma$, and $\hat \xi$ further in Section \ref{Sec:AnalysisDRGMM}. 

Given plug-in estimates of the unknown quantities, we define the Neyman-orthogonal score function for testing $\alpha_0$:
$$
\sqrt{n} \hat M(\alpha_0; \hat \theta) =  \sqrt{n} (\hat \xi_0 -   \hat \xi_0 \hat G \hat \mu_0 \hat \gamma_0) \hat g(\alpha_0, \hat \theta).
$$
By exploiting Neyman orthogonality property and further assumptions on the problem,
we can show that this feasible score approximates the infeasible ``oracle" score,
\begin{eqnarray*}
\sqrt{n}\hat M(\alpha_0; \hat \theta) & =  & \sqrt{n}\bar M(\alpha_0; \hat \theta)   + o_P(1/\sqrt{\log p}).
\end{eqnarray*}
Hence, the feasible score is approximately linear and is therefore an approximate mean.  We then have
\begin{equation}\label{DRGMM-test}
\sqrt{n}\bar M(\alpha_0; \hat \theta)  \approx_d  N(0, V_M),
\end{equation}
over the class of all rectangles in $\mathbb{R}^p$ under the conditions of Theorem 2.1 in Section 2.  We are thus back within the setting outlined in Section 2 and may use the inferential tools from Section 2 for construction of simultaneous confidence bands and hypothesis testing with control of FWER or FDR.  Given that we can provide valid inferential statements based on (\ref{DRGMM-test}) for any $\alpha_0$, we can invert to obtain confidence regions.  

\subsection{Analysis of DRGMM}\label{Sec:AnalysisDRGMM}

In order to analyze the estimator (\ref{DRGMM}), we can write, using elementary expansions and some algebra,
\begin{equation}\label{eq:LinearExpansionDeBiasedRGMM}
\sqrt{n} ( \check \theta - \theta_0)   =  - \mu_0 \gamma_0 \sqrt{n} \hat g(\theta_0) + r,
\end{equation}
where
\begin{equation}\label{def:r1r2r3}
 r  =   r_1 + r_2 + r_3  \quad \left | \begin{array}{ccl}
r_1 & = &  \sqrt{n} (I - \hat \mu \hat \gamma \hat G) (\hat \theta - \theta_0) \\
r_2 & = & \sqrt{n}(\hat \mu \hat \gamma (\hat G - \tilde G)) (\hat \theta - \theta_0). \\
r_3 & = & \sqrt{n}(\hat \mu \hat \gamma - \mu_0 \gamma_0) \hat g (\theta_0)
\end{array} \right.
\end{equation}
In (\ref{def:r1r2r3}), $\tilde G = \{- \partial_{\theta'} \hat g_k(\hat \theta^*_k)\}_{k=1}^m$ denotes a $m\times p$ matrix with rows $ - \partial_{\theta'} \hat g_k(\hat \theta^*_k)$, $k \in [m]$, where each row
is evaluated at a point $\hat \theta^*_k$ on the line between $\hat \theta$ and $\theta_0$.

Note that because of Neyman orthogonality property we expect the term $r_1$ to be small,
in fact if we knew $(\mu_0, \gamma_0, G)$ the term would vanish.  In linear moment models,
$\hat G = \tilde G$, so that the second term vanishes, $r_2 = 0$. The last term can also vanish under mild conditions.  We analyze the structure of these remainder terms in the lemma given below.

To fix ideas, we record a trivial proposition.

\begin{proposition}[Approximate Linearity and Normality of DRGMM]\label{DRGMM linearity} If the remainder term $r$ obeys the conditions of Section 2, then DRGMM is an approximate mean estimator:
$$
\sqrt{n}(\check \theta - \theta_0) = \frac{1}{\sqrt{n}} \sum_{i=1}^n Z_i + r,  \quad Z_i = -\mu_0\gamma_0 g(X_i, \theta_0).
$$
Each component of the estimator is approximately normally distributed and satisfies self-normalized moderate deviations results of Section 2, provided the $Z_i$'s obey the regularity conditions of Section 2. The distribution of $\sqrt{n}(\check\theta-\theta_0)$ over rectangles can be approximated by bootstrapping the scores $$\hat Z_i = \hat\mu\hat \gamma g(X_i, \hat \theta),$$ provided these obey the regularity conditions of  Section 2. Consequently, the results on simultaneous inference with FWER control or pointwise testing with FDR control apply.
\end{proposition}

A crucial step to using Proposition \ref{DRGMM linearity} is to ensure the approximation error $r$ is small. The following lemma is useful for thinking about estimators of $\gamma_0$ and $\mu_0$ which are suitably well-behaved to obtain small approximation errors.  Of course, the lemma only suggests one possible direction, and there are other strategies for estimating $\gamma_0$ and $\mu_0$ to explore. In what follows, we use $v_j$ to denote the $j$th row of some matrix $v$.

\begin{lemma}\label{lemma:linearize} We have that the approximation errors $r_1$, $r_2$ and $r_3$ as defined in (\ref{def:r1r2r3}) satisfy
$$
\begin{array}{rl}
\|r_1\|_\infty & \leq  \bar r_1 = \sqrt{n} \| I - \hat \mu \hat \gamma \hat G\|_\infty \| \hat \theta- \theta_0\|_1\\
\\
\| r_2 \|_\infty & \leq  \bar r_2=  \sqrt{n} \max_{j\in[p]} \| \hat \mu_j\|_1 \max_{j\in[p]}\| \hat \gamma_j\|_1 \| \hat G - \tilde G\|_\infty \| \hat \theta - \theta_0\|_1 \\
\\
 \|r_3\|_\infty & \leq\bar r_3 = \max_{j\in[p]} \| \hat \mu_j \|_1 \max_{j\in[p]}\|\hat \gamma_j - \gamma_{0j}\|_1 \|\sqrt{n}\hat g(\theta_0)\|_\infty \\
&  \ \ \ \ \ \ \ \ \ + \max_{j\in[p]} \| \hat \mu_j - \mu_{0j} \|_1 \max_{j\in[p]} \| \gamma_{0j}\|_1 \| \sqrt{n}\hat g(\theta_0)\|_\infty\\
\end{array}
$$
with simplifications occurring in the linear case, when $\tilde G = \hat G$, and in the case
where the moment selection matrix is known, $\hat \gamma = \gamma_0$.
\end{lemma}

\begin{proof}[Proof of Lemma \ref{lemma:linearize}] In what follows we denote by $e_j$ the coordinate vector, with $1$ in the $j$-th position and $0$ in the other positions. To obtain the first bound, we have that by H\"{o}lder's inequality
$$
\|r_1\|_\infty \leq \sqrt{n} \max_{j\in [p]} \| e_j' - \hat \mu_j \hat \gamma \hat G\|_\infty \| \hat \theta - \theta_0\|_1 \leq \bar r_1.
$$

The second bound follows from multiple applications of the H\"{o}lder's inequality
$$
\begin{array}{rl}
\|r_2\|_\infty & \leq \|\hat \mu \hat \gamma (\hat G - \tilde G)\|_\infty \sqrt{n} \| \hat \theta - \theta_0\|_1\\
& \leq \max_{j,k\in[p]} \|\hat \mu_{j\cdot}\|_1 \| [\hat \gamma (\hat G - \tilde G)]_{\cdot,k}\|_\infty \sqrt{n} \| \hat \theta - \theta_0\|_1\\
& = \max_{j,k\in[p]} \|\hat \mu_{j\cdot}\|_1  \max_{l\in[p]}\| \hat \gamma_{l\cdot} (\hat G - \tilde G)_{\cdot,k}| \sqrt{n} \| \hat \theta - \theta_0\|_1\\
& \leq \max_{j\in[p]} \|\hat \mu_{j\cdot}\|_1  \max_{l\in[p]}\| \hat \gamma_{l\cdot}\|_1 \|\hat G - \tilde G\|_\infty \sqrt{n} \| \hat \theta - \theta_0\|_1.\\
\end{array}$$

Finally, we bound $r_3$. It follows from H\"{o}lder's inequality and the triangle inequality that

$$\begin{array}{rl}
\|r_3\|_\infty & \leq \max_{j\in[p]} | (\hat \mu_j \hat \gamma - \mu_{0j}\gamma_0) \sqrt{n}\hat g(\theta_0)|\\
& \leq \max_{j\in[p]} | \hat \mu_j (\hat \gamma - \gamma_0) \sqrt{n}\hat g(\theta_0)| + \max_{j\in[p]} | (\hat \mu_j - \mu_{0j})\gamma_0 \sqrt{n}\hat g(\theta_0)|\\
& \leq \max_{j\in[p]} \| \hat \mu_j \|_1 \|(\hat \gamma - \gamma_0) \sqrt{n}\hat g(\theta_0)\|_\infty + \max_{j\in[p]} \| \hat \mu_j - \mu_{0j} \|_1 \| \gamma_0 \sqrt{n}\hat g(\theta_0)\|_\infty\\
& \leq \max_{j\in[p]} \| \hat \mu_j \|_1 \max_{j\in[p]}|(\hat \gamma - \gamma_0)_{j\cdot} \sqrt{n}\hat g(\theta_0)|\\
& \qquad + \max_{j\in[p]} \| \hat \mu_j - \mu_{0j} \|_1 \max_{j\in[p]} | (\gamma_0)_{j,\cdot} \sqrt{n}\hat g(\theta_0)|\\
& \leq \max_{j\in[p]} \| \hat \mu_j \|_1 \max_{j\in[p]}\|(\hat \gamma - \gamma_0)_{j\cdot}\|_1 \|\sqrt{n}\hat g(\theta_0)\|_\infty \\
& \qquad + \max_{j\in[p]} \| \hat \mu_j - \mu_{0j} \|_1 \max_{j\in[p]} \| (\gamma_0)_{j,\cdot}\|_1 \| \sqrt{n}\hat g(\theta_0)\|_\infty.\\
\end{array}$$


\end{proof}

A plausible approach is then to construct estimators $\hat \gamma$, $\hat \mu$, and $\hat \theta$ such that the upper bounds $\bar r_1$, $\bar r_2$, and $\bar r_3$ given in Lemma \ref{lemma:linearize} approach zero sufficiently fast.  A natural choice of $\hat\theta$ is given by the RGMM estimator discussed in Section \ref{Sec:RMD}. We can also obtain estimators $\hat G$ and $\hat \Omega$ by plug-in expressions.  We now turn to estimating $\gamma_0$ and $\mu_0$. 

First, we consider one potential estimator for $\gamma_0$.  We do not attempt to use a standard plug-in estimate since $\hat\Omega$ will not be full rank in high-dimensional settings, so its inverse will be ill-posed even if $\Omega^{-1}$ is well-behaved. Instead, we define the estimator $\hat \gamma$  as 
the solution to the program:
\begin{equation}\label{def:hatgamma}
\min_{\gamma \in \Bbb{R}^{p \times m}} \sum_{j \in [p]} \| \gamma_j \|_1 :   \quad \|  \gamma_j \hat \Omega - (\hat G')_j  \|_\infty \leq \lambda^{\gamma}_j, \quad j \in [p],
\end{equation}
where $\lambda^\gamma$ is a vector of regularization parameters.  By allowing $\lambda^\gamma>0$, the constraint in (\ref{def:hatgamma}) requires solving the inverse problem only approximately, which is needed to handle the rank deficiency of $\hat \Omega$.


We proceed similarly in our proposed estimator for $\mu_0$. We define the estimator $\hat \mu$ as the solution to the following program:
\begin{equation}\label{def:hatmu}
\min_{ \mu \in \Bbb{R}^{p \times m} } \sum_{j\in[p]} \| \mu_j\|_1:  \quad \| \mu_j \hat \gamma \hat G - e_j' \|_\infty \leq \lambda^\mu_j, \quad j \in [p]
\end{equation}
where $e_j$ is a coordinate vector with 1 in the $j$-th position and 0 elsewhere and $\lambda^\mu$ is a vector of regularization parameters.  Again, the use of positive regularization parameters $\lambda^\mu>0$ allows us to work with approximate solutions which are needed to cope with the high-dimensionality.

We now summarize an algorithm for constructing the estimator $\check\theta$.

\noindent {\bf Algorithm for DRGMM.}\\
{\it Step 1. Compute the RGMM estimator $\hat \theta$.\\
Step 2. Use the plug-in rules $\hat G = \partial_{\theta'} \hat g(\hat \theta)$ and  $\hat \Omega = \En g(X,\hat \theta) g(X,\hat \theta)'$.\\
Step 3. Obtain the estimator $\hat\gamma$ as defined in (\ref{def:hatgamma}).\\
Step 4. Obtain the estimator $\hat\mu$ as defined in (\ref{def:hatmu}).\\
Step 5. Update the initial RGMM estimator $\check\theta = \hat\theta - \hat\mu\hat\gamma \hat g(\hat\theta)$.
 }

We note that the regularized problems (\ref{def:hatgamma}) and (\ref{def:hatmu}) can be cast as linear programming problems and can be solved separately by row $j\in [p]$.  Both features are convenient from a computational perspective. 

Next we proceed to analyze the properties of the estimators. The following lemma provides high-level conditions on the estimators of $G$ and $\Omega$ and on the penalty choices to derive the needed $\ell_1$-rates of convergence for the rows of $\hat \gamma$ and $\hat \mu$.

\begin{lemma}\label{lem:HL-OmegaG}
Let $n^{1/2}\|\hat \Omega - \Omega\|_\infty\leq \ell^\Omega_n$ and $n^{1/2}\|\hat G - G\|_\infty \leq  \ell_n^G$ with probability $1-\delta_n$ and suppose that $\max_{j\in[p]}\|\gamma_{0j}\|_1 \leq K$. Let the penalty parameters satisfy $n^{1/2}\lambda_j^\gamma \geq K \ell^\Omega_n + \ell_n^G$ and $\lambda_j^\gamma \leq n^{-1/2}\ell_n$ for $j\in[p]$. Suppose that Condition LID$(\gamma_{0j},\Omega)$ holds for each ${j\in[p]}$.
 Then, with probability $1-3\delta_n$ we have
$$ \max_{j\in[p]}\|\hat \gamma_j - \gamma_{0j}\|_1 \leq \frac{C_{a,1}s\ell_n(2+K) (L_n+\mu_n)}{\mu_n \sqrt{n}}. $$
Suppose that $\max_{j\in[p]}  \|\mu_{0j}\|_1 \leq K$. Let the penalty parameters in (\ref{def:hatmu}) satisfy $n^{1/2}\lambda_j^\mu \geq 2K^2\ell_n^G+K^3\ell_n^\Omega+K^2\max_{j\in[m]}n^{1/2}\lambda_j^\gamma$ and $ \lambda_j^\mu \leq n^{-1/2} \ell_n'$ for $j\in [p]$. Suppose  Condition LID$(\mu_{0j},G'\Omega^{-1}G)$ holds for each ${j\in[p]}$. Then with probability $1-\delta_n$ we have
$$ \max_{j\in[p]}\|\hat \mu_j - \mu_{0j}\|_1 \leq\frac{C_{a,1}s\ell_n'(2+K) (L_n+\mu_n)}{\mu_n \sqrt{n}}  $$
\end{lemma}

Lemma \ref{lem:HL-OmegaG} builds upon the theory of RGMM with linear score functions established in Section \ref{Sec:RMD}. As expected, condition LID is assumed to hold for the different Jacobian matrices. Lemma \ref{lem:HL-OmegaG} also highlights sufficient conditions on how the penalty parameters should be chosen. Moreover, it assumes that the estimators $\hat \Omega$ and $\hat G$ have good rates of convergence in the $\ell_\infty$-norm.

The following lemmas complement the result of Lemma \ref{lem:HL-OmegaG} by providing conditions and explicit bounds on the rates of convergence for the estimators of $\Omega$ and $G$ in the linear and non-linear case.

\begin{lemma}[Linear Score]\label{lem:PrimitiveGandOmegaLinear}
Consider the case of linear score, $g(\theta)=G\theta+g(0)$ where $G_{kj}=\Ep[G_{kj}(X)]$, $g_k(0)=\Ep[g_k(X,0)]$. Let $\Omega = \Ep[g(X,\theta_0)g(X,\theta_0)']$ and $\hat\Omega=\En[g(X,\hat\theta)g(X,\hat\theta)']$. Suppose that:
\begin{itemize}
\item[(i)] $\max_{k\in[m],j\in[p]}\Ep[ G_{kj}^2(X) ]\leq \sigma^2$, $\max_{k\in[m]} \Ep[g_k^2(0)]\leq \sigma^2$; 
\item[(ii)] $n^{-1/2}\Ep[\max_{i\in[n]}\|G(X_i)\|_\infty^2] \leq \delta_n \log^{-1/2}(2m)$;
\item[(iii)] with probability $1-\delta_n$ we have $\max_{k\in[m]}\En[\{G_k(X)'(\hat \theta - \theta_0)\}^2] \leq \Delta_{2n}^2$; 
\item[(iv)] $c\leq \max_{k\in[m]}\Ep[g_k^4(X,\theta_0)]\leq C$, $n^{-1/2}\Ep[\max_{i\in[n]}\|g(X,\theta_0)\|_\infty^4] \leq \delta_n \log^{-1/2}(2m).$
\end{itemize}
Then, with probability $1-C\delta_n$ we have $$\begin{array}{rl}
  \|\hat G - G\|_\infty & \leq  C\sigma\sqrt{n^{-1}\log(2m)}\\
  \|\hat\Omega - \Omega\|_\infty & \leq C' \sqrt{n^{-1}\log(2m)} + C\Delta_{2n} + \Delta_{2n}^2
 \end{array}$$
\end{lemma}

The moment assumptions in Lemma \ref{lem:PrimitiveGandOmegaLinear} are quite standard and allow for $m\gg n$. Requirement (iii) relies on the rate of convergence of $\hat \theta$ which impacts the estimation of $\Omega$ only. We note that there are examples in which we can bypass this term such as the homoskedastic linear instrumental variable case discussed in Theorem \ref{thm:MainInference:Linear}.

Next we provide conditions to derive bounds on the estimation error of $G$ and $\Omega$ for the non-linear case. The conditions will assume a Lipschitz condition on the score and its derivative as stated in (\ref{eq: UbarSingleIndices nonlinear case}) and (\ref{eq: lipschitz nonlinear case}) of condition ENM.

\begin{lemma}[Non-Linear Score]\label{lem:PrimitiveGandOmega}
Let the score $(g_k, k\in[m])$ and $(\partial_j g_k, j\in [p],k\in[m])$, $m\geq p$, satisfy conditions (\ref{eq: UbarSingleIndices nonlinear case})) and (\ref{eq: lipschitz nonlinear case}), with $(L_k(X),Z_k(X))_{k=1}^m$ and $(\tilde L_{kj}(X),\tilde Z_{kj}(X))_{k\in[m],j\in[p]}$ respectively. Suppose further that:\\
(i) for all ${j\in[m],l\in[p]}$, $\Ep[L_j^2(X)Z_{jl}^2(X)]\leq B_n^2$,  $\Ep[ |L_j(X)Z_j'v|^2\{1+ g_k^2(X,\theta_0)\}] \leq C\|v\|^2$; \\
and for each ${k,j\in[m],l\in[p]}$, $\Ep[\tilde L_{kj}^2(X)\tilde Z_{kjl}^2(X)]\leq B_n^2$, and  $\Ep[ |\tilde L_{kj}(X)\tilde Z_{kj}'v|^2] \leq C\|v\|^2$; \\
(ii) with probability $1-\delta_n$ we have $\max_{k,j\in[m], l\in[p]}\En[L_j^2(X) Z_{kl}^2\{1+ g_k^2(X,\theta_0)\}]\leq B_n^2$;\\
(iii) with probability $1-\delta_n$ we have $\|\hat \theta - \theta_0\|_\ell \leq \Delta_{\ell n}$  for $\ell \in \{1,2\}$; \\
(iv) $\Ep[g_k^4(X,\theta_0)]\leq C$, $k\in [m]$, $n^{-1/2}\Ep[\max_{i\in[n]}\|g(X_i,\theta_0)\|_\infty^4] \leq \delta_n \wedge \log^{-1/2}m; $\\
(v) $\Ep[G_{kj}^2(X,\theta_0)]\leq C$, $k\in [m], j\in[p]$; $n^{-1/2}\Ep[\max_{i\in[n]}\|G(X_i,\theta_0)\|_\infty^2] \leq \delta_n \wedge \log^{-1/2}m.$
\\
Then, with probability $1-C'\delta_n$ we have
$$\begin{array}{rl}
 \|\hat G - G\|_\infty  & \leq C'  \sqrt{n^{-1}\log(2m)} + C' B_n\Delta_{1n} \sqrt{n^{-1}\log(m/\delta_n)} +  C'\Delta_{2n} \\
 \|\hat G - \tilde  G\|_\infty  & \leq  C'  \sqrt{n^{-1}\log(2m)} + C'B_n\Delta_{1n} \sqrt{n^{-1}\log(m/\delta_n)} + C'\Delta_{2n}\\
 \|\hat\Omega - \Omega\|_\infty & \leq C' \sqrt{n^{-1}\log(2m)} + C'B_n\Delta_{1n} \sqrt{n^{-1}\log(m  p/\delta_n)} + 2 B_n^2\Delta_{1n}^2 + 2 C\Delta_{2n}.
  \end{array}$$
\end{lemma}

The moment conditions are quite standard. The bounds depend on the $\ell_1$ and $\ell_2$-rates of convergence of the initial RGMM estimator $\hat \theta$.

The following theorems provide results that builds upon the RGMM estimator discussed in Section \ref{Sec:RMD} and builds upon the previous lemmas to deliver the approximate linear expansion (\ref{eq:LinearExpansionDeBiasedRGMM}). 

We begin with a result for the homoskedastic linear instrumental variables model. Let
$Y = W'\theta_0 + \epsilon$ with $\Ep \epsilon Z = 0$, $\Ep \epsilon^2 ZZ' = \sigma^2 \Ep ZZ'$.  Then, using the moment function $g(\theta) = \Ep [ (Y- W'\theta) Z]$, we have $G = -\Ep Z W'$, $g(0) = \Ep Y Z$, and $\Omega = \sigma^2\Ep ZZ'$. 
In this homoskedastic setting, we will compute $\hat \gamma$ in (\ref{def:hatgamma}) with $\hat\Omega =\En ZZ'$.

\begin{theorem}[Homoskedastic Linear IV Model]\label{thm:MainInference:Linear}
Consider the homoskedastic high-dimensional linear instrumental variable model. Suppose:\\
(1) $\|\theta_0\|_1 \leq K$, $\max_{j\in[p]}\|\gamma_{0j}\|_1 \leq K$ and $\max_{j\in[p]}\|\mu_{0j}\|_1\leq K$;\\
(2) Conditions LID$(\theta_0,G)$, LID$(\gamma_{0j},\Omega)$ and LID$(\mu_{0j},G'\Omega^{-1}G)$ hold for $j\in [p]$;\\
(3) $\max_{k\in[m],j\in[p]}\Ep[ Z_k^2W_j^2 ]+\Ep[ Y^2Z_k^2]\leq C$, $c\leq \max_{k\in[m]}\Ep[\epsilon^4Z_k^4]\leq C$; \\
(4) $n^{-1/2}\{\Ep[\max_{i\in[n]}\|Z_iW_i'\|_\infty^2] + \Ep[\max_{i\in[n]}\|\epsilon_i Z_i\|_\infty^4]\} \leq \delta_n \log^{-1/2}(2m)$;\\
(5) $K + L _n+ \sigma^2  +\mu_n^{-1} \leq C$.\\
 Then for $\bar \lambda = C'(1+\sigma)\sqrt{n^{-1}\log(2mn)}$ for some fixed $C'$ sufficiently large, setting $\lambda_j^\gamma = \frac{1}{2}\lambda_j^\mu = \bar \lambda$, with probability $1-C\delta_n$ we have
$$\begin{array}{rl}
\sqrt{n}(\check\theta - \theta_0) = \frac{(G'\Omega^{-1}G)^{-1}G'\Omega^{-1}}{\sqrt{n}}\sum_{i=1}^n \epsilon_i Z_i + r, \ \ \mbox{with} \ \ \|r\|_\infty \leq C u_n
\end{array}
$$
provided that $n^{-1}s^2\log^2(pmn) \leq u_n^2$.
\end{theorem}

Theorem \ref{thm:MainInference:Linear} derives an approximate linear representation for the estimator that immediately allows us to construct simultaneous confidence regions for all parameters under the conditions of Section 2. The result exploits the homoskedasticity and  bypasses the need to estimate $\sigma^2$. In turn this allows the representation to hold under the mild sparsity requirement of $n^{-1}s^2\log^2(pmn) \leq u_n^2$.

Under more stringent requirements, the next result considers the non-linear case where the functions $g_k$ satisfies a Lipschitz condition; see condition ENM in Section \ref{Sec:RestrictedNonLinearCase}.

\begin{theorem}[Non-Linear Case]\label{thm:MainInference:NonLinear}
Suppose the following conditions hold:\\
(1) $\|\theta_0\|_1 \leq K$, $\max_{j\in[p]}\|\gamma_{0j}\|_1 \leq K$ and $\max_{j\in[p]}\|\mu_{0j}\|_1\leq K$;\\
(2) Conditions LID$(\theta_0,G)$, LID$(\gamma_{0j},\Omega)$ and LID$(\mu_{0j},G'\Omega^{-1}G)$ hold for $j\in [p]$;\\
(3) Condition ENM holds for the score $(g_k, k\in[m])$  with $(L_k(X),Z_k(X))_{k=1}^m$;\\
(4) Condition ENM holds for $(\partial_j g_k, j\in [p],k\in[m])$, $m\geq p$, with $(\tilde L_{kj}(X),\tilde Z_{kj}(X))_{k\in[m],j\in[p]}$;\\
(5) for all ${j\in[m],l\in[p]}$, $\Ep[L_j^2(X)Z_{jl}^2(X)]\leq B_n^2$,  $\Ep[ |L_j(X)Z_j'v|^2\{1+ g_k^2(X,\theta_0)\}] \leq C\|v\|^2$; \\
and for each ${k,j\in[m],l\in[p]}$, $\Ep[\tilde L_{kj}^2(X)\tilde Z_{kjl}^2(X)]\leq B_n^2$, and  $\Ep[ |\tilde L_{kj}(X)\tilde Z_{kj}'v|^2] \leq C\|v\|^2$; \\
(6) with probability $1-\delta_n$ we have $\max_{k,j\in[m], l\in[p]}\En[L_j^2(X) Z_{kl}^2\{1+ g_k^2(X,\theta_0)\}]\leq B_n^2$;\\
(7) $\Ep[g_k^4(X,\theta_0)]\leq C$, $k\in [m]$, $n^{-1/2}\Ep[\max_{i\in[n]}\|g(X_i,\theta_0)\|_\infty^4] \leq \delta_n \wedge \log^{-1/2}m; $\\
(8) $\Ep[G_{kj}^2(X,\theta_0)]\leq C$, $k\in [m], j\in[p]$; $n^{-1/2}\Ep[\max_{i\in[n]}\|G(X_i,\theta_0)\|_\infty^2] \leq \delta_n \wedge \log^{-1/2}m.$\\
(9) $B_n + K + L _n+ \mu_n^{-1} \leq C$.\\
For $\bar a\geq 0$ and $C'\geq 1$, let $\bar\lambda = C'n^{-\frac{1}{2}+\bar a}\Phi^{-1}(1-(pmn)^{-1})$. Then, setting $\lambda_j^\gamma = \frac{1}{2}\lambda_j^\mu = \bar \lambda$, with probability $1-C\delta_n$ we have
$$\begin{array}{rl}
\sqrt{n}(\check\theta - \theta_0) = - \mu_0\gamma_0\hat g(\theta_0) + r, \ \ \mbox{with} \ \ \|r\|_\infty \leq C u_n
\end{array}
$$
provided that $n^{-1+2\bar a}s^2\log^2(pmn) \leq u_n^2$ and $\bar \lambda \geq Cn^{-1/2}s^{1/2}\log^{1/2}(2m)$ for some large $C>0$.
\end{theorem}

Theorem \ref{thm:MainInference:NonLinear} provides one approach to constructing estimators with suitable linearization based on the RGMM estimator and the estimators (\ref{def:hatgamma}) and (\ref{def:hatmu}) for the nuisance parameters. Under suitable choice of penalty parameters, the requirement $n^{-1}s^3\log^2(pmn) \leq u_n$ where $u_n = o(\log^{-1/2}(p))$ suffices to ensure that approximation errors do not distort the asymptotic coverage of (rectangular) confidence regions. We note that the derivation of practical choices of penalty parameters has drawn considerable attention in the literature. Although Theorem \ref{thm:MainInference:NonLinear} allows us to postulate a choice $\bar \lambda$ that allows us to cover a class of $s$-sparse models, it would be of interest to obtain adaptive rules that are theoretically valid and practical. 
See Remark \ref{rem:PracticalConsiderations} below for some initial discussion.
%



%


\begin{remark}[Practical Considerations for the DRGMM Estimator]\label{rem:PracticalConsiderations}
We note that the penalty choices discussed in Lemma \ref{lem:HL-OmegaG} rely on using upper bounds of some unknown quantities. Relying on upper bounds is common in deriving theoretical results, though using upper bounds in finite-samples may result in overpenalization and a deterioration in performance.  An alternative approach attempts to make the estimators adaptive to the relevant unknown quantities.  For example, we can define an alternative estimator for $\gamma_0$ as the solution to the following optimization problem
\begin{equation}\label{def:hatgammaAdapt}
\min_{\gamma \in \Bbb{R}^{p \times m}} \sum_{j \in [p]} \| \gamma_j \|_1 :   \quad \|  \gamma_j \hat \Omega - (\hat G')_j  \|_\infty \leq \|\gamma_j\|_1n^{-1/2}\ell_n^\Omega + n^{-1/2}\ell_n^G, \quad j \in [p].
\end{equation}
The optimization problem in (\ref{def:hatgammaAdapt}) can still be written as a linear programming problem after adding additional variables, \cite{BRT:ells}, \cite{BCKTR:eiv}, and \cite{BCK:eiv}. The benefit of using (\ref{def:hatgammaAdapt}) is that it avoids trying to guess $K$. Similarly, we can have
\begin{equation}\label{def:hatmuAdapt}
\min_{ \mu \in \Bbb{R}^{p \times m} } \sum_{j\in[p]} \| \mu_j\|_1:  \quad \| \mu_j \hat \gamma \hat G - e_j' \|_\infty \leq \|\mu_j\|_1\lambda^\mu_j, \quad j \in [p]
\end{equation} 
where $\lambda^\mu_j = 2\max_{j\in[p]}\|\hat\gamma_j\|_1 n^{-1/2}\ell_n^G+\max_{j\in[p]}\|\hat\gamma_j\|_1^2 n^{-1/2}\ell_n^\Omega+\max_{j\in[p]}\|\hat\gamma_j\|_1\|\lambda^\gamma\|_\infty$.
This approach is justified by H\"{o}lder's inequality. Other approaches motivated by  self-nor\-ma\-li\-zation moderate deviation theory leads to non-linear problems that can be handled via conic programming in special cases; see \cite{BRT:coniceiv} and \cite{BCHN17}.
\end{remark}

\section{Bibliographical Notes and Open Problems}
The literature on CLT with increasing dimensions is quite broad, and we refer the reader to Appendix I in \cite{CCK13} for an extensive review on this topic prior to the publication of \cite{CCK13}. In the following discussion, we mainly focus on the development after \cite{CCK13}. 
The high-dimensional CLT and bootstrap results in this chapter, namely Theorems \ref{thm: clt for mam}--\ref{thm: empirical bootstrap for mam}, build upon Proposition 2.1, Corollary 4.2, and Proposition 4.3, respectively, in \cite{CCK17}, which, in turn, improves on the results of their earlier paper \cite{CCK13}. Both papers use some important technical tools, such as anti-concentration inequalities and Gaussian comparison theorems, obtained in \cite{CCK15}. \cite{W14} provides a helpful exposition of these results targeting mathematically oriented graduate students. \cite{CCK13} also provides several useful applications of these results, including the choice of the regularization parameter for the Dantzig selector, specification testing with a parametric model under the null and a nonparametric one under the alternative, and multiple testing with FWER control. Another useful application is testing many moment inequalities, where the number of moment inequalities is larger than the sample size, which is studied in  \cite{CCK13b} in detail. 

There are several extensions of high-dimensional CLT and bootstrap results of \cite{CCK13,CCK17}. 
\cite{DZ17} show that Condition E, which is imposed in Theorems \ref{thm: clt for mam}--\ref{thm: empirical bootstrap for mam}, can be slightly improved if we are only concerned with inference based on the empirical bootstrap or if we use the multiplier bootstrap with Gaussian weights $e_i$ replaced by weights satisfying
\begin{equation}\label{eq: third moment matching}
\Ep[e_i] = 0,\quad \Ep[e_i^2] = 1,\quad \Ep[e_i^3] = 1,\quad i\in[n].
\end{equation}
In particular, they show that the term $\log^7(p n)/n$ in Condition E can be replaced by the term $\log^5(p n)/n$. To obtain their results, they circumvent the Gaussian approximation and work directly with the bootstrap approximation.  \cite{ZW17,CCK13b,ZC17b} develop time series extensions of the high-dimensional CLT.
\cite{Ch17,ChKa17} develop extensions of the high-dimensional CLT and bootstrap theorems to $U$-statistics and randomized incomplete $U$-statistics, respectively (\cite{Ch17} focuses on the second order case). \cite{Ko17} studies Gaussian approximation to a high-dimensional vector of smooth Wiener functionals by combining the techniques developed in \cite{CCK13,CCK15,CCK17} and Malliavin calculus.

An important feature of Theorems \ref{thm: clt for mam}--\ref{thm: empirical bootstrap for mam} is that they provide distributional approximation results for the class of {\em rectangles}. There are also many related results in the literature if we are interested in other classes of sets. For example, \cite{B03,Be05} show that a result like \eqref{eq: clt simple implication} with $\mathcal A$ being the class of all {\em convex sets} is possible under certain moment conditions if $p = o(n^{2/7})$. More formally, \cite{Be05} proves the following: Let $Z_{1},\dots,Z_{n}$ be zero-mean independent random vectors  
in $\RR^{p}$ and suppose that the covariance matrix of $S_{n}^{Z} = n^{-1/2} \sum_{i=1}^{n}Z_{i}$, $V= n^{-1}\sum_{i=1}^{n} \Ep[Z_{i}Z_{i}']$, is invertible;  then 
\begin{equation}
\sup_{A \subset \RR^{p}: \text{convex}} | \Pr (S_{n}^{Z} \in A) - \Pr (N(0,V) \in A) | \le \frac{K p^{1/4}}{n^{3/2}} \sum_{i=1}^{n} \Ep[\| V^{-1/2} Z_{i} \|_{2}^{3}], \label{eq: Bentkus}
\end{equation}
where $K$ is a universal constant. In the simplest case where $V=I$ and $\| Z_{i} \|_{2} \le C \sqrt{p}$ for all $i\in[n]$ and some constant (independent of $n$), the right-hand side on (\ref{eq: Bentkus}) is $O(p^{7/4}n^{-1/2})$, which is $o(1)$ if $p=o(n^{2/7})$. This result is straightforward to use for inference since the covariance matrix $V$ can be accurately estimated as long as $p\ll n$.
Recently, \cite{Zh17} improves on the Bentskus condition in the case where $Z_{1},\dots,Z_{n}$ are i.i.d. with identity covariance matrix, and such that $\| Z_{i} \|_{2} \le C \sqrt{p}$ for all $i\in[n]$ and some constant $C$; under these conditions \cite{Zh17} shows that the left-hand side of (\ref{eq: Bentkus}) is approaching zero provided that $p=o(n^{2/5})$ up to log factors.
\cite{Z16} shows that the multiplier bootstrap inference over all {\em centered balls} with weights $e_i$ satisfying \eqref{eq: third moment matching} is possible if $p = o(n^{1/2})$. See also \cite{M93} for some early results regarding the multiplier bootstrap with weights $e_i$ satisfying \eqref{eq: third moment matching}.

Theorem \ref{thm: moderate deviations for mam} on moderate deviations of self-normalized sums extends the results of \cite{JSW03} to show that the moderate deviation inequality holds, up to some corrections, even if the original random variables in the denominator of the self-normalized sum are replaced by suitable estimators. This is particularly helpful when we allow for many approximate means as opposed to many exact means; see \cite{BCCH12} for an application. Textbook-level treatment of the theory of self-normalized sums can be found in \cite{PLS09}.

Theorem \ref{thm: simultaneous conf intervals} on simultaneous confidence intervals is a rather simple application of Theorems \ref{thm: clt for mam}--\ref{thm: empirical bootstrap for mam}. Similar results formulated in terms of particular applications can be found for example in \cite{BCK:biometrika}, \cite{BCCW17}, and \cite{BCHN17}. Simultaneous confidence intervals also constitute an important research topic in the literature on nonparametric estimation and inference. Useful references on this literature are provided in \cite{CCK14b}.

Theorems \ref{thm: fwer bonferroni} and \ref{thm: MHT} on multiple testing with FWER control build upon \cite{RW05} with the key difference that we allow for $p\to\infty$, and in particular $p/n\to\infty$, as $n\to\infty$. A closely related analog of these theorems can be found in \cite{CCK13} but our conditions here are somewhat weaker than those in \cite{CCK13}. \cite{W00} explains importance of FWER control in multiple testing. A clean textbook-level treatment of multiple testing with FWER control can be found in \cite{LR05}. A useful discussion of multiple testing in experimental economics can be found in \cite{LSX15}. \cite{RS10} draws a connection between multiple testing with the FWER control and constructing sets covering the identified sets with a prescribed probability in partially identified models and provide a stepdown procedure for constructing such sets. Note also that our formulation of the Bonferroni-Holm procedure is slightly different but equivalent to the commonly used formulation, e.g. in \cite{LR05}. We have changed the formulation to facilitate the comparison between the Bonferroni-Holm and Romano-Wolf procedures.

Theorem \ref{thm: fdp} on multiple testing with FDR control generalizes the results of \cite{LS14} to allow for many approximate means. In turn, \cite{LS14} generalizes the original results of \cite{BH95} on the Benjamini-Hochberg procedure to allow for the unknown distribution of the data and also to allow for some dependence between the $t$-statistics. Moreover, \cite{LS14} uses moderate deviation for self-normalized sums theory to allow for testing in ultra-high dimensions. Theorem \ref{thm: fdp} is also closely related to the results in \cite{LL14}, who considers multiple testing with FDR control for a specific setting: variable selection in a high-dimensional regression model. A useful discussion of multiple testing with FDR control and other types of control can be found in \cite{RSW08}. A textbook-level treatment is provided in \cite{G15}.

Our Condition C for the analysis of the Benjamini-Hochberg procedure requires that each $t$-statistic is correlated with a relatively small set of other $t$-statistics. The procedure, however, remains valid under the so-called positive regression dependency condition; see \cite{BY01} for details. There also has been a lot of research about related procedures; see e.g. \cite{RSW08b}. A radically different procedure, which also allows for FDR control but which we did not consider in this chapter, is the knockoff filter of \cite{BC15}. This alternative procedure is designed specifically for the variable selection problem in the linear mean regression model and requires the number of covariates to be smaller than the sample size but does not restrict dependence between covariates in any way. See also \cite{CFJL17} where the knockoff filter is modified to allow for the high-dimensional regression model, with the number of covariates exceeding the sample size, in exchange for some other conditions. 

We present our high-dimensional estimation results within the context of $\ell_1$-regularized minimum distance estimation focusing on the case where the parameter vector exhibits an approximately sparse structure.  The estimation results clearly build upon fundamental work for $\ell_1$-penalized regression of \cite{FF93} and \cite{T96}.  This initial work has been expanded in many directions; see, for example, the textbook treatment of \cite{BvdG11} as well as \cite{CT07}, \cite{BRT09}, \cite{BC11}, \cite{BCW11}, \cite{BCCH12}, \cite{BCH14}, and \cite{BCFVH17} for results most closely related to the approach taken in this review.  

More generally, providing methods for estimating high-dimensional models has been an active area of research for quite some time, and there is a large collection of methods available within the literature.  \cite{ESL} and \cite{CASI} provide useful textbook introductions to a wide array of methods that are useful in high-dimensional contexts.  Developing new techniques for estimation in high-dimensional settings is also still an active area of research, so the list of methods available to researchers continues to expand.  Further exploring the use of these procedures in economic applications and the impact of their use on inference about structural parameters seems like a useful avenue to pursue. 

Methods for obtaining valid inferential statements following regularization in high-dimensional settings has been an active area of research in the recent statistics and econometrics literature.  Early work on inference in high-dimensional settings focused on the exact sparsity structure with separation (from zero) discussed in Section \ref{Sec: RegularizedMM}; see, e.g., \cite{FanLi2001} for an early paper or \cite{FanLv2010} for a more recent review.  A consequence of sparsity with strong separation from zero is that model selection does not impact the asymptotic distribution of the parameters estimated in the selected model, under regularity conditions.  This property allows one to do inference using standard approximate distributions for the parameters of the selected model ignoring that model selection was done.  While convenient, inferential results obtained relying on this structure may perform very poorly in more realistic approximately sparse structures as was noted in a series of papers; see, for example, \cite{LP08a} and \cite{LP08b}.  

The more recent work on inference about model parameters following the use of regularization, including the procedure outlined in this chapter, has focused on providing procedures that remain valid without maintaining exact sparsity with separation. As noted in Section \ref{Sec:DRGMM}, a key element in obtaining valid inferential statements is the use of estimating equations that satisfy the Neyman orthogonality condition.  This idea dates at least to \cite{N59} who used the idea of projecting the score that identifies the parameter of interest onto the ortho-complement of the tangent space for nuisance parameters in the construction of the $C(\alpha)$, or orthogonal score, statistic.  This idea also plays a key role in semiparametric and targeted learning theory; see, for example, \cite{Andrews94}, \cite{Newey94}, \cite{vdv98}, \cite{SRR99}, and \cite{vdLR11}.  

Within the high-dimensional context, much of the work on inference focuses on inference for prespecified low-dimensional parameters in the presence of high-dimensional nuisance parameters when $\ell_1$ regularization or variable selection methods are used to estimate the nuisance parameters.  \cite{BCH10b} considers inference about parameters on a low-dimensional set of endogenous variables following selection of instruments from a high-dimensional set using lasso in a homoscedastic, Gaussian IV model.  Their approach relies on the fact that the moment condition underlying IV estimation is Neyman orthogonal.  These ideas were further developed in the context of providing uniformly valid inference about the parameters on endogenous variables in the IV context with many instruments to allow non-Gaussian heteroscedastic disturbances in \cite{BCCH12}.  \cite{BCH14}, which to our knowledge provides the first formal statement of the Neyman orthogonality condition in the high-dimensional setting, covers inference on the parametric components of the partially linear model and average treatment effects.  See also \cite{BCH10a}, \cite{F15}, \cite{K18}, \cite{BCHK16}, and \cite{BCFVH17}, among others, for further applications and generalizations explicitly making use of Neyman orthogonal estimating equations.  As noted above, Neyman orthogonal estimating equations are closely related to Neyman's $C(\alpha)$-statistic. The use of $C(\alpha)$ statistics for testing and estimation with high-dimensional approximately sparse models was first explored in the context of quantile regression in \cite{BCK:biometrika} and in the context of high-dimensional generalized linear models by \cite{BCW16}.  Other uses of $C(\alpha)$-statistics or close variants include those in \cite{VSW:score}, \cite{NL:sparc}, \cite{YNL:score}, and \cite{NL17}.  Finally, a different strand of the literature has focused on ex-post ``de-biasing'' of estimators to enable valid inference as opposed to directly basing estimation and inference on orthogonal estimating equations.  While seemingly distinct, the de-biasing approach is the same as approximately solving orthogonal estimating equations; see, for example, discussion in \cite{CHS}.  Important seminal contributions following the de-biasing approach are \cite{ZZ14}, \cite{vdGBRD14}, and \cite{JM14}.

Rather than focus on a pre-specified low-dimensional parameter, one may also do inference for high-dimensional parameters.  \cite{WR09} and \cite{MMB09} use sample splitting to provide procedures for multiple inference in high-dimensional settings that can control FWER and FDR under strong conditions.  \cite{NvdG13} also consider the construction of confidence sets for the entire parameter vector in a sparse high-dimensional regression using sample splitting ideas. \cite{vdGBRD14} suggest using Bonferroni-Holm in conjunction with the de-sparsified lasso using a limiting distribution derived under homoscedastic Gaussian errors.  As discussed in this chapter, the high-dimensional CLT and bootstrap results of \cite{CCK13} are broadly applicable for inference about high-dimensional parameters.  \cite{BCK:biometrika} provides an early use of these results for construction of a simultaneous confidence rectangle for many target parameters within a rich class of models estimated using orthogonal estimating equations; see also \cite{CHS:hdm} which implements inference for high-dimensional treatment or structural effects following \cite{BCK:biometrika}.  \cite{DBZ16} and \cite{ZC17} also consider bootstrap inference for many parameters estimated with debiased estimators in high-dimensional models building on \cite{CCK13}.  More recently, \cite{BCCW17} extends \cite{BCK:biometrika} to provide valid inference for many functional parameters, and \cite{BCK:eiv} consider inference for many parameters in a high-dimensional linear model with errors in variables.  Finally, \cite{CG16}, \cite{ZB:lin}, \cite{ZB:proj}, and \cite{HKM} consider different approaches which allow testing hypotheses about functionals that may involve the entire high-dimensional parameter vector within different high-dimensional contexts. 

There is also a rapidly growing body of research focused on learning economically interesting parameters that uses different high-dimensional methods and/or aims to provide reliable inferential statement under relatively weaker conditions. Data-adaptive estimation of nuisance functions, with an emphasis on using high-dimensional methods, is advocated in the targeted learning literature under the nomenclature ``super learner'' though many formal results in this literature are obtained in low-dimensional settings; see, e.g., \cite{vdl:super}, \cite{vdLR11}, and \cite{ZvdL11}.  \cite{AI16} is an important example that uses tree-based methods and sample splitting for estimating and performing inference about heterogeneous treatment effects.  \cite{WA:rf} considers estimation and inference for heterogeneous treatment effects using a variant of random forests with formal results established in a low-dimensional context.  In a high-dimensional setting, \cite{CCDDHNR18} use Neyman orthogonal estimating equations and sample splitting to provide a generic procedure for inference about low-dimensional parameters under weak conditions that allow for the use of wide variety of high-dimensional, machine learning methods.  Sample splitting and orthogonal estimating equations are also employed in \cite{CGST18}, which considers inference for high-dimensional conditional heterogeneous treatment effects.  These ideas are also extended in \cite{CDDF18} which provides inference for a variety of useful functionals of heterogeneous treatment effects, such as the best linear predictor of the conditional average treatment effect function, estimated via generic high-dimensional methods using sample splitting while accounting for uncertainty introduced from the sample splits under very mild conditions.  \cite{BCHN17} provides multipurpose inference methods for high-dimensional causal effects which cover endogenous treatments.  
\cite{AIW:bal} use a reweighting after regression adjustment via lasso which allows valid inference for the average treatment effect to be performed under very weak conditions on the propensity score as long as treatment and control conditional mean functions are linear.  \cite{CNR17} considers inference for linear functionals of conditional expectations with estimated Riesz representers under very weak conditions.  Note that among applications of \cite{CNR17} is estimation of average treatment effects, and that the conditions of \cite{CNR17}  also impose weak assumptions on the propensity score.  The advantage of  
\cite{CNR17}'s approach  over that in \cite{AIW:bal}  is that the former explicitly allows the tradeoff in the rate of estimating the inverse propensity score with the rate of estimating the regression function, allowing misspecification of both functions.  The regularity conditions are also substantively weaker, for example allowing regressions functions be completely non-sparse when the inverse propensity score is well-approximated.

Most of the work in the recent literature on high-dimensional estimation and inference relies on approximate sparsity to provide dimension reduction and the corresponding use of sparsity-based estimators. Dense models are appealing in many settings and may be usefully employed in more moderate-dimensional settings.  The many weak-instrument regime, popular in econometrics since at least \cite{bekker}, provides one such example.  Approaches which provide valid inference for structural parameters within this setting can all be viewed as making use of regularization to avoid dramatic overfitting in the relationship between endogenous variables and the many available instruments.  See, for example, \cite{CI04}, \cite{Okui11}, \cite{Carr12}, and \cite{HK14} for approaches that explicitly use regularized first-stage estimation within a dense model framework.  \cite{CJN16} and \cite{CJN17} consider inference for a low-dimensional set of coefficients in a linear model with number of variables proportional to but smaller than the sample size in a framework allowing for the coefficients on the nuisance variables to be dense.  Within the same framework, \cite{CJM17} extend this work to address inference about parameters estimated using two-step procedures where the first step is a linear regression with many variables.  


We conclude these bibliographic notes by noting that the references to the high-dimensional literature provided above are necessarily selective.  The literature on high-dimensional estimation and inference is large and rapidly expanding, and it is impractical to give more than a cursory overview highlighting a few examples.  The goal of this review is to provide readers with a few key papers in several areas and a taste of existing results.

\clearpage

\appendix

\section{Fundamental Tools}
\subsection{Tool Set 1: Moderate Deviation Inequality for Self-normalized Sums}

\begin{lemma}\label{lem: moderate deviations}
Let $\xi_{1},\dots, \xi_{n}$ be independent mean-zero random variables with $n^{-1}\sum_{i=1}^n\Ep [ \xi_{i}^{2} ] \geq 1$ and $\Ep [| \xi_{i} |^{2+\nu} ] < \infty$ for all $i\in[n]$ where $0 < \nu \leq 1$.
Let $S_{n}:=\sum_{i=1}^{n} \xi_{i}, V_{n}^2:=\sum_{i=1}^{n} \xi_{i}^2$, and $D_{n,\nu}:=(n^{-1}\sum_{i=1}^{n} \Ep [ |\xi_{i}|^{2+\nu}])^{1/(2+\nu)}$. Then uniformly in $0\leq x\leq n^{\frac{\nu}{2(2+\nu)}}/D_{n,\nu}$,
\[
\left | \frac{\Pr(S_n/V_n\geq x)}{1 - \Phi(x)} - 1 \right | \leq  K n^{-\nu/2} D_{n,\nu}^{2+\nu} (1+x)^{2+\nu},
\]
where $K$ is a universal constant.
\end{lemma}
\begin{proof}
See Theorem 7.4 in \cite{PLS09} or the original paper, \cite{JSW03}. Note that the formulation in these sources requires that $n^{-1}\sum_{i=1}^n\Ep[\xi_i^{2}] = 1$ but it it is trivial to show that we can instead assume that $n^{-1}\sum_{i=1}^n\Ep[\xi_i^{2}] \geq 1$.
\end{proof}

\subsection{Tool Set 2: Maximal and Deviation Inequalities}

\begin{lemma}\label{lem: fuk-nagaev}
Let $X_{1},\dots,X_{n}$ be independent random vectors in $\R^{p}$. In addition, define $\sigma^{2} := \max_{j \in [p]} \sum_{i=1}^{n} \Ep [X_{i j}^{2}]$. Then for every $s > 1$ and $t > 0$,
\begin{multline*}
\Pr \left  ( \max_{j \in [p]} \Big | \sum_{i=1}^{n} (X_{ij} - \Ep [ X_{ij} ]) \Big | \geq 2\Ep \Big [ \max_{ j \in [p]} \Big | \sum_{i=1}^{n} (X_{ij} - \Ep [ X_{ij} ]) \Big | \Big  ]+t \right )
\\ \leq e^{-t^{2}/(3\sigma^{2})} + \frac{K_{s}}{t^s} \sum_{i=1}^{n} \Ep \left[ \max_{j \in [p]} | X_{ij} - \Ep[X_{i j}] |^{s} \right],
\end{multline*}
where $K_{s}$ is a constant depending only on $s$.
\end{lemma}
\begin{proof}
See Lemma E.2 in \cite{CCK17}.
\end{proof}



\begin{lemma}
\label{lem: maximal ineq}
Let $X_{1},\dots,X_{n}$ be independent random vectors in $\R^{p}$ with $p \geq 2$. Define $M := \max_{i \in [n]} \max_{j \in [p]} | X_{i j} |$ and $\sigma^{2} := \max_{j \in [p]} \sum_{i=1}^{n} \Ep [ X_{i j}^{2} ]$.
Then
\begin{equation*}
\Ep \left [\max_{j \in [p]} \Big  | \sum_{i=1}^{n} (X_{i j} - \Ep [X_{i j}]) \Big  | \right ] \leq K (\sigma \sqrt{\log p} + \sqrt{\Ep [ M^{2} ]} \log p),
\end{equation*}
where $K$ is a universal constant.
\end{lemma}

\begin{proof}
See Lemma 8 in \cite{CCK15}.
\end{proof}


\begin{lemma}\label{lem: deviation inequality for positive rvs}
Let $X_1,\dots,X_n$ be independent random vectors in $\mathbb R^p$ with $p\geq 2$ such that $X_{i j}\geq 0$ for all $i\in[n]$ and $j \in[p]$. Define $M:=\max_{i\in[n]}\max_{j\in[p]}X_{i j}$. Then for any $s\geq 1$ and $t>0$,
$$
\Pr\left(\max_{j\in[p]}\sum_{i=1}^n X_{i j} \geq 2\Ep\Big[\max_{j\in[p]}\sum_{i=1}^n X_{i j} \Big] + t \right) \leq K\Ep[M^s]/t^s.
$$
where $K$ is a constant depending only on $s$.
\end{lemma}
\begin{proof}
See Lemma E.4 in \cite{CCK17}.
\end{proof}

\begin{lemma}\label{lem: maximal inequality for positive rvs}
Let $X_1,\dots,X_n$ be independent random vectors in $\mathbb R^p$ with $p\geq 2$ such that $X_{i j}\geq 0$ for all $i\in[n]$ and $j \in[p]$. Define $M:=\max_{i\in[n]}\max_{j\in[p]}X_{i j}$. Then
$$
\Ep\left[ \max_{j\in[p]}\sum_{i=1}^n X_{i j} \right] \leq K\left( \max_{j\in[p]}\Ep\left[\sum_{i=1}^n X_{i j}\right] + \Ep[M]\log p \right),
$$
where $K$ is a universal constant.
\end{lemma}
\begin{proof}
See Lemma 9 in \cite{CCK15}.
\end{proof}

\begin{lemma}\label{lem: critical value bound}
Let $(Y_1,\dots,Y_p)^{T}$ be a Gaussian random vector with $\Ep [ Y_{j} ]= 0$ and $\Ep[ Y_{j}^{2} ]\leq 1$ for all $j \in[p]$. For $\alpha\in(0,1)$, let $c(\alpha)$ denote the $(1-\alpha)$ quantile of the distribution of $\max_{j\in[p]}Y_j$. Then
$c(\alpha)\leq \sqrt{2\log p}+\sqrt{2 \log(1/ \alpha)}$.
\end{lemma}
\begin{proof}
By the Borell-Sudakov-Tsirel'son inequality (Theorem A.2.1 in \cite{VW96}), for every $r > 0$,
\[
\Pr\left(\max_{j\in[p]}Y_j\geq \Ep\Big[\max_{j\in[p]}Y_j\Big]+r\right)\leq e^{-r^2/2},
\]
which implies that
\begin{equation}\label{eq: critical value bound}
c(\alpha)\leq \Ep\left[\max_{j\in[p]}Y_j\right]+\sqrt{2\log (1/\alpha)}.
\end{equation}
In addition, by Proposition A.3.1 in \cite{T11},
\begin{equation}\label{eq: gaussian maximal inequality}
\Ep\left[\max_{j\in[p]}Y_j\right]\leq \sqrt{2\log p}
\end{equation}
Combining \eqref{eq: critical value bound} and \eqref{eq: gaussian maximal inequality} leads to the desired result. 
\end{proof}

\subsection{Tool Set 3: High-dimensional Central Limit and Bootstrap Theorem, Gaussian Comparison
and Anti-Concentration Inequalities}

In Theorems \ref{thm: hdclt}--\ref{thm: empirical bootstrap} below, we will follow the following setting: Let $X_{1},\dots,X_{n}$ be independent zero-mean random vectors in $\RR^{p}$ , and consider $S_{n}^{X} = n^{-1/2} \sum_{i=1}^{n} X_{i}$. Denote by $V=n^{-1} \sum_{i=1}^{n} \Ep[X_{i}X_{i}']$ the covariance matrix of $S_{n}^{X}$ (which we assume to exist). 
In addition, let $\mathcal A$ denote the class of all rectangles in $\mathbb R^p$, i.e. sets of the form
$$
A = \Big\{w = (w_1,\dots,w_p)'\in\mathbb R^p\colon w_{l j} \leq w_j\leq w_{r j}\text{ for all }j\in[p]\Big\},
$$
where $w_l = (w_{l 1},\dots,w_{l p})'$ and $w_r = (w_{r 1},\dots,w_{r p})'$ are $p$-dimensional vectors with components in $(\mathbb R\cup\{-\infty\}\cup\{+\infty\})^p$.

\begin{theorem}[High-dimensional CLT]\label{thm: hdclt}
Assume that for some constants $b>0$ and $B\geq 1$,
\begin{equation}\label{eq: m1}
\frac{1}{n}\sum_{i=1}^n\Ep[X_{i j}^2]\geq b \quad \text{and} \quad \frac{1}{n}\sum_{i=1}^n\Ep[|X_{i j}|^{2+k}] \leq B \quad \text{for all} \ j \in[p] \ \text{and} \ k=1,2.
\end{equation}
Then the following claims hold: (i) if
\begin{equation}\label{eq: e1}
\Ep[\exp(|X_{i j}|/B)]\leq 2\text{ for all }i\in[n]\text{ and }j\in[p],
\end{equation}
then
\begin{equation}\label{eq: hdclt under e1}
\sup_{A\in\mathcal A}|\Pr(S_n^X \in A) - \Pr(N(0,V) \in A)|\leq K_1\left(\frac{B^2\log^7(p n)}{n}\right)^{1/6},
\end{equation}
where $K_1$ is a constant depending only on $b$; (ii) if for some constant $q \in [3,\infty)$,
\begin{equation}\label{eq: e2}
\Ep\left[\max_{j \in [p]}(|X_{i j}|/B)^q\right] \leq 1 \text{ for all }i\in[n],
\end{equation}
then
\begin{equation}\label{eq: hdclt under e2}
\sup_{A\in\mathcal A}|\Pr(S_n^X \in A) - \Pr(N(0,V) \in A)|\leq K_2\left(\left( \frac{B^2\log^7(p n)}{n} \right)^{1/6} + \left( \frac{B^2\log^3(p n)}{n^{1 - 2/q}} \right)^{1/3}\right),
\end{equation}
where $K_2$ is a constant depending only on $b$ and $q$.
\end{theorem}

\begin{proof}
See Proposition 2.1 in \cite{CCK17}.
\end{proof}



\begin{theorem}[Multiplier Bootstrap]\label{thm: multiplier bootstrap}
Let $e_1,\dots,e_n$ be i.i.d. $N(0,1)$ random variables that are independent of $X_1,\dots,X_n$ and consider $S_{n}^{eX} =n^{-1/2} \sum_{i=1}^{n}e_{i}X_{i}$. Assume that \eqref{eq: m1}  holds for some constants $b>0$ and $B\geq 1$. Recall $X_{1}^{n} =(X_{1},\dots,X_{n})$. Then the following claims hold: (i) if \eqref{eq: e1} is satisfied, then  for any constant $\beta\in(0,e^{-1})$ with probability at least $1 - \beta$,
$$
\sup_{A\in\mathcal A}|\Pr(S_n^{e X}\in A\mid X_1^n) - \Pr(N(0,V) \in A)| \leq K_1\left(\frac{B^2 \log^5(p n)\log^2(1/\beta)}{n}\right)^{1/6},
$$
where $K_1$ is a constant depending only on $b$; (ii) if \eqref{eq: e2} is satisfied for some constant $q \in [4,\infty)$, then for any constant $\beta\in(0,e^{-1})$ with probability at least $1 - \beta$,
\begin{align*}
&\sup_{A\in\mathcal A}|\Pr(S_n^{e X}\in A\mid X_1^n) - \Pr(N(0,V) \in A)| \\
&\qquad  \leq K_2\left(\left( \frac{B^2 (\log p + \log(1/\beta)) \log^4p}{n} \right)^{1/6} + \left( \frac{B^2 (\log p + \beta^{-2/q})\log^2 p}{n^{1 - 2/q} }\right)^{1/3}\right),
\end{align*}
where $K_2$ is a constant depending only on $b$ and $q$.
\end{theorem}
\begin{proof}
The claim under \eqref{eq: e1} is proven in Corollary 4.2 of \cite{CCK17}. (Note that there, $S_n^{e X} = n^{-1/2}\sum_{i=1}^n e_i X_i$ is replaced by $n^{-1/2}\sum_{i=1}^n e_i(X_i - \bar X_n)$, where $\bar X_n := n^{-1}\sum_{i=1}^n X_i$ is the sample mean, but since $X_i$'s have mean zero, this change is not essential).

The claim under \eqref{eq: e2} improves upon Corollary 4.2 of \cite{CCK17}. To prove the asserted claim, note that, by Remark 4.1 in \cite{CCK17}, there exists a constant $K>0$ depending only on $b$ such that on the event $\Delta_n \leq \bar\Delta_n$,
\begin{equation}\label{eq: gaussian comparison sharp}
\sup_{A\in\mathcal A}|\Pr(S_n^{e X}\in A\mid X_1^n) - \Pr(N(0,V) \in A)| \leq K\bar\Delta_n^{1/3} \log^{2/3} p,
\end{equation}
where
$$
\Delta_n := \max_{1\leq j,k\leq p}\left|\frac{1}{n}\sum_{i=1}^n (X_{i j}X_{i k} - \Ep[X_{i j}X_{i k}])\right|.
$$
Also, as in the proof of Proposition 4.1 in \cite{CCK17}, case (E.2), for all $t > 0$,
\begin{align*}
&\Pr\left( \Delta_n > C\left(\left( \frac{B^2 \log p}{n} \right)^{1/2} + \frac{B^2 \log p}{n^{1 - 2/q}} \right) + t \right) \leq \exp\left(-\frac{n t^2}{3 B^2}\right) +\frac{c B^q}{ t^{q/2} n^{q/2 - 1}},
\end{align*}
where $c,C>0$ are constants depending only on $b$ and $q$. Thus, setting
$$
t = \bar C\left( \left( \frac{B^2 \log(1/\beta)}{n} \right)^{1/2} + \frac{B^2}{\beta^{2/q} n^{1 - 2/q}} \right)
$$
for sufficiently large $\bar C$, which can be chosen to depend only on $b$ and $q$, shows that with probability at least $1 - \beta$,
$$
\Delta_n \leq (C + \bar C)\left(\left( \frac{B^2(\log p + \log(1/\beta))}{n} \right)^{1/2} + \frac{B^2(\log p + \beta^{-2/q})}{n^{1 - q/2}}\right).
$$
The asserted claim follows by combining this bound with \eqref{eq: gaussian comparison sharp}.
\end{proof}
\begin{theorem}[Empirical Bootstrap]\label{thm: empirical bootstrap}
Let $X_1^*,\dots,X_n^*$ be i.i.d. random vectors from the empirical distribution of $X_1,\dots,X_n$ and consider $S_{n}^{X^*} =  n^{-1/2}\sum_{i=1}^n(X_i^* - \bar X_n)$,
where $\bar X_n = n^{-1}\sum_{i=1}^n X_i$ is the sample mean. Assume that \eqref{eq: m1} holds for some constants $b>0$ and $B\geq 1$. Then the following claims hold: (i) if \eqref{eq: e1} is satisfied, then  for any constant $\beta\in(n^{-2},e^{-1})$ with probability at least $1 - \beta$,
$$
\sup_{A\in\mathcal A}|\Pr(S_n^{X^{*}}\in A\mid X_1^n) - \Pr(N(0,V) \in A)| \leq K_1\left(\frac{B^2 \log^7(p n)}{n}\right)^{1/6},
$$
where $K_1$ is a constant depending only on $b$; (ii) if \eqref{eq: e2} is satisfied for some constant $q \in [4,\infty)$, then for any constant $\beta\in(n^{-2},e^{-1})$ with probability at least $1 - \beta$,
\begin{align*}
&\sup_{A\in\mathcal A}|\Pr(S_n^{X^{*}}\in A\mid X_1^n) - \Pr(N(0,V) \in A)| \\
&\qquad  \leq K_2\left(\left( \frac{B^2\log^7 (pn)}{n} \right)^{1/6} + \left( \frac{B^2\log^3 (pn)}{\beta^{2/q}n^{1 - 2/q} }\right)^{1/3}\right),
\end{align*}
where $K_2$ is a constant depending only on $b$ and $q$.
\end{theorem}
\begin{proof}
This is Proposition 4.3 in \cite{CCK17}.
\end{proof}

\begin{lemma}[Gaussian Comparison]
Let $Y = (Y_1,\dots,Y_p)'$ and $Z = (Z_1,\dots,Z_p)'$ be mean-zero Gaussian random vectors in $\mathbb R^p$ with covariance matrices $\Sigma^{Y} = (\Sigma_{jk}^{Y})_{1 \le j,k \le p}$ and $\Sigma^{Z} = (\Sigma_{jk}^{Z})_{1 \le j,k \le p}$, respectively. Let $\Delta := \max_{j,k\in[p]}| \Sigma_{jk}^{Y} - \Sigma_{jk}^{Z} |$
and let $\mathcal A$ be the class of all rectangles in $\mathbb R^p$.
Suppose that $\min_{j \in [p]} \Sigma_{jj}^{Y} \bigvee \min_{j \in [p]}\Sigma_{jj}^{Z} \ge \underline{\sigma}^{2}$ for some constant $\underline{\sigma}> 0$. Then 
$$
\sup_{A\in\mathcal A}\left| \Pr\left( Y \in A\right) - \Pr\left( Z \in A\right) \right| \leq C\Delta^{1/3}\log^{2/3}p,
$$
where $C$ is a constant that depends only on $\underline\sigma$.
\end{lemma}

\begin{proof}
The proof is implicit in the proof of Theorem 4.1 in \cite{CCK17}. 
This Gaussian comparison inequality here improves slightly on the original Gaussian comparison inequality in \cite{CCK15}.
\end{proof}



\begin{lemma}[Anti-concentration Inequality]\label{lem: anticoncentration 2}
Let $Z = (Z_1,\dots,Z_p)'$ be a mean-zero Gaussian random vector in $\mathbb R^p$ with $\sigma_j^2 := \Ep[Z_j^2] >  0$ for $j\in[p]$. Denote $\underline\sigma :=\min_{j\in[p]}\sigma_j$. Then for all $\epsilon > 0$ and $x = (x_{1},\dots,x_{p})' \in \RR^{p}$, we have 
\begin{equation}\label{eq: anticoncentration 3}
\Pr(Z \leq x+\epsilon) - \Pr(Z \leq x)\leq \frac{\epsilon}{\underline{\sigma}} (\sqrt{2\log p}+2),
\end{equation}
where $x+\epsilon = (x_{1}+\epsilon,\dots,x_{p}+\epsilon)'$. 
\end{lemma}

\begin{proof}
This is Nazarov's inequality stated in Lemma A.1 in \cite{CCK17}. For a detailed proof of Lemma \ref{lem: anticoncentration 2}, see \cite{CCK17arxiv}. 
\end{proof}


\section{Proofs for Section 2}

\noindent
{\bf Proof of Theorem \ref{thm: clt for mam}.}
Take any rectangle $A\in\mathcal A$. Then there exist vectors $w_l = (w_{l 1},\dots,w_{l p})'$ and $w_r = (w_{r 1},\dots,w_{r p})'$ in $\mathbb R^p$ such that
$$
A = \Big\{ w = (w_1,\dots,w_p)'\in\mathbb R^p\colon w_{l j} \leq w_j\leq w_{r j}\text{ for all }j\in[p] \Big\}.
$$
Also, denote $\Delta_n := \delta_n / \sqrt{\log p}$ and let
$$
A^{+} := \Big\{w = (w_1,\dots,w_p)'\in\mathbb R^p\colon w_{j l} - \Delta_{n} \leq w_j \leq w_{j r} + \Delta_{n}\text{ for all }j\in[p]\Big\},
$$
so that $A^{+}\in\mathcal A$ is also a rectangle and $A\subset A^{+}$. Then, by Condition A and linearization \eqref{linearize},
\begin{equation}\label{eq: thm1 bound 1}
\Pr(\sqrt n(\hat\theta - \theta_0) \in A) \leq \left(\frac{1}{\sqrt n}\sum_{i=1}^n Z_i \in A^{+}\right) + \beta_{n}.
\end{equation}
In addition, under Conditions M and E, it follows from Theorem \ref{thm: hdclt} that
\begin{equation}\label{eq: thm1 bound 2}
\Pr\left( \frac{1}{\sqrt n}\sum_{i=1}^n Z_i \in A^{+} \right) \leq \Pr(N(0,V) \in A^{+}) + K_1\delta_n,
\end{equation}
where $K_1$ is a constant that depends only on $q$.

Further, for any vector $a = (a_1,\dots,a_p)'\in\mathbb R^p$ and any number $b \in \mathbb R$, let $a + b$ and $a - b$ be vectors in $\mathbb R^p$ defined by $(a+b)_j = a_j + b$ and $(a-b)_j = a_j - b$ for all $j\in[p]$, respectively. Note that $w = (w_1,\dots,w_p)'\in A^{+}\backslash A$ implies that either $w_j \leq w_{j r} + \Delta_{n}$ for all $j$ but $w_j > w_{j r}$ for some $j$ or $w_{j} \geq w_{j l} - \Delta_{n}$ for all $j$ but $w_j < w_{j l}$ for some $j$. Hence, by the union bound,
\begin{align*}
\Pr\left( N(0,V) \in A^{+}\backslash A \right)
& \leq \Pr\left( N(0,V) \leq w_r + \Delta_n \right) - \Pr\left( N(0,V) \leq w_r \right) \\
&\quad + \Pr\left( N(0,V) \geq w_l - \Delta_n \right) - \Pr\left( N(0,V) \geq w_l \right).
\end{align*}
In turn, by the anti-concentration inequality, Lemma \ref{lem: anticoncentration 2},
$$
\Pr\left( N(0,V) \leq w_r + \Delta_n \right) - \Pr\left( N(0,V) \leq w_r \right) \leq K_2\Delta_{n} \sqrt{\log p} \leq K_2\delta_n
$$
and
\begin{align*}
&\Pr\left( N(0,V) \geq w_l - \Delta_n \right) - \Pr\left( N(0,V) \geq w_l \right)\\
&\qquad = \Pr\left( - N(0,V) \leq - w_l + \Delta_n \right) - \Pr\left( - N(0,V) \leq - w_l \right) \leq K_2\delta_n
\end{align*}
for some universal constant $K_2$. Therefore,
$$
\Pr\left( N(0,V) \in A^{+} \right) \leq \Pr\left( N(0,V) \in A\right) + 2K_2\delta_n.
$$
Combining this inequality with \eqref{eq: thm1 bound 1} and \eqref{eq: thm1 bound 2} gives
\begin{equation}\label{eq: thm 1 upper bound}
\Pr(\sqrt n(\hat\theta - \theta_0) \in A) \leq \Pr\left(N(0,V) \in A\right) + \beta_n + (K_1 + 2K_2)\delta_n,
\end{equation}
which is the upper bound for $\Pr(\sqrt n(\hat\theta - \theta_0)\in A)$. To establish the lower bound, one can use the same argument with the rectangle
$$
A^{-} := \Big\{w = (w_1,\dots,w_p)'\in\mathbb R^p\colon w_{l j} + \Delta_{n} \leq w_j \leq w_{r j} - \Delta_{n}\text{ for all }j\in[p]\Big\},
$$
and obtain
\begin{equation}\label{eq: thm 1 lower bound}
\Pr(\sqrt n(\hat\theta - \theta_0) \in A) \geq \Pr\left(N(0,V) \in A\right) - \beta_n - (K_1 + 2K_2)\delta_n.
\end{equation}
Combining \eqref{eq: thm 1 upper bound} and \eqref{eq: thm 1 lower bound} gives the asserted claim with $C := 1 + K_1 +2 K_2$. \qed

\noindent
{\bf Proof of Theorem \ref{thm: bootstrap for mam}.}
First, we consider the case of E.1. Take any rectangle $A\in\mathcal A$. Then there exist vectors in $\mathbb R^p$, $w_l = (w_{l 1},\dots,w_{l p})'$ and $w_r = (w_{r 1},\dots,w_{r p})'$, such that
$$
A = \Big\{ w = (w_1,\dots,w_p)'\in\mathbb R^p\colon w_{l j} \leq w_j\leq w_{r j}\text{ for all }j\in[p] \Big\}.
$$
Also, denote $\Delta_n := \delta_n / \sqrt{\log (p n)}$. Now, observe that conditional on $(Z_i,\hat Z_i)_{i=1}^n$, for all $j\in[p]$, the random variable $n^{-1/2}\sum_{i=1}^n e_i (\hat Z_{i j} - Z_{i j})$ is zero-mean Gaussian with variance $n^{-1}\sum_{i=1}^n (\hat Z_{i j} - Z_{i j})^2$. Hence, by the Borell-Sudakov-Tsirel'son inequality, for any $\alpha\in(0,1)$,
$$
\Pr_e\left( \max_{j\in[p]}\left| \frac{1}{\sqrt n}\sum_{i=1}^n e_i(\hat Z_{i j} - Z_{i j}) \right| > \max_{j\in[p]}\sqrt{\En[(\hat Z_{i j} - Z_{i j}])^2}\left(\sqrt{2\log p} + \sqrt{2\log(1/\alpha)}\right)  \right)
$$
is bounded from above by $2\alpha$. Therefore, setting $\alpha = \delta_n/2\geq 1/(2 n)$ and using Condition A shows that there exists a universal constant $K_1$ such that with probability at least $1 - \beta_n$,
\begin{equation}\label{eq: proof of thm 2 main event}
\Pr_e\left( \max_{j\in[p]}\left| \frac{1}{\sqrt n}\sum_{i=1}^n e_i(\hat Z_{i j} - Z_{i j}) \right| > K_1\Delta_n \right) \leq \delta_n.
\end{equation}
Hence, letting
$$
A^{+} := \Big\{w = (w_1,\dots,w_p)'\in\mathbb R^p\colon w_{l j} - K_1\Delta_{n} \leq w_j \leq w_{r j} + K_1\Delta_{n}\text{ for all }j\in[p]\Big\},
$$
so that $A^{+}\in\mathcal A$ is also a rectangle and $A\subset A^{+}$, implies that with probability at least $1 - \beta_n$,
\begin{equation}\label{eq: proof of thm 2 bound 1}
\Pr_e\left( \frac{1}{\sqrt n}\sum_{i=1}^n e_i \hat Z_i \in A \right) \leq \Pr_e\left( \frac{1}{\sqrt n}\sum_{i=1}^n e_i Z_i \in A^{+} \right) + \delta_n.
\end{equation}
Further, by Theorem \ref{thm: multiplier bootstrap}, there exists a universal constant $K_2$ such that
\begin{equation}\label{eq: proof of thm 2 bound 2}
\Pr_e\left( \frac{1}{\sqrt n}\sum_{i=1}^n e_i Z_i \in A^{+} \right) \leq \Pr(N(0,V) \in A^{+}) + K_2\delta_n
\end{equation}
holds with probability at least $1 - (2 n)^{-1}$. Also, like in the proof of Theorem \ref{thm: clt for mam},
$$
\Pr(N(0,V) \in A^{+}) \leq \Pr(N(0,V)\in A) + K_1 K_3\Delta_n \sqrt{\log p} \leq \Pr(N(0,V)\in A) + K_1 K_3\delta_n,
$$
where $K_3$ is a universal constant. Combining this inequality with \eqref{eq: proof of thm 2 bound 1} and \eqref{eq: proof of thm 2 bound 2} and noting that $\sqrt n(\hat\theta^* - \hat\theta) = n^{-1/2}\sum_{i=1}^n e_i \hat Z_i$ shows that with probability at least $1 - \beta_n - (2 n)^{-1}$,
\begin{equation}\label{eq: thm 2 upper bound}
\Pr_e(\sqrt n(\hat\theta^* - \hat \theta) \in A) \leq \Pr(N(0,V) \in A) + C\delta_n,
\end{equation}
where $C = 1 + K_2 + K_1 K_3$, which is the upper bound for $\Pr_e(\sqrt n(\hat\theta^* - \hat\theta) \in A)$. To establish the lower bound, one can use the same argument with the rectangle
$$
A^{-} := \Big\{w = (w_1,\dots,w_p)'\in\mathbb R^p\colon w_{l j} + K_1 \Delta_{n} \leq w_j \leq w_{r j} -  K_1\Delta_{n}\text{ for all }j\in[p]\Big\},
$$
and obtain that with probability at least $1 - \beta_n - (2 n)^{-1}$,
\begin{equation}\label{eq: thm 2 lower bound}
\Pr_e(\sqrt n(\hat\theta^* - \hat\theta) \in A) \geq \Pr\left(N(0,V) \in A\right) - C\delta_n.
\end{equation}
Combining \eqref{eq: thm 2 upper bound} and \eqref{eq: thm 2 lower bound} and noting that the same event \eqref{eq: proof of thm 2 main event} is used to establish both \eqref{eq: thm 2 upper bound} and \eqref{eq: thm 2 lower bound} gives the asserted claim for the Gaussian bootstrap in the case of E.1.

In the case of E.2, we proceed through the same steps but we note that \eqref{eq: proof of thm 2 bound 2} only holds with probability at least $1 - (\log n)^{-2}/2$. This completes the proof of the theorem.
\qed

\noindent
{\bf Proof of Theorem \ref{thm: empirical bootstrap for mam}.}
Note that we can assume without loss of generality that $\delta_n \leq 1$ since otherwise the claims of the theorem are trivial. Therefore, since $(B_n^2 \log^7(p n)/n)^{1/6} \leq \delta_n$ by Condition E and $B_n\geq 1$, it follows that $\log^7(p n)/n \leq 1$, and so $\delta_n \geq \log^3(p n)/\sqrt n$. This bound will be useful later in the proof.

Now, conditional on $(Z_i,\hat Z_i)_{i=1}^n$, the random vector
$$
\frac{1}{\sqrt n}\sum_{i=1}^n (e_i - 1)(\hat Z_i - Z_i)
$$
is equal in distribution to
$$
\frac{1}{\sqrt n}\sum_{i=1}^n\left( (\hat Z_i - Z_i)^* - \frac{1}{n}\sum_{l=1}^n(\hat Z_l - Z_l) \right),
$$
where $(\hat Z_i - Z_i)^*$, $i\in[n]$, are i.i.d. random vectors from the empirical distribution of $\hat Z_1 - Z_1,\dots,\hat Z_n - Z_n$. Also, by Condition A,
\begin{equation}\label{eq: bernstein 1}
\max_{j\in[p]}( \En[(\hat Z_{i j} - Z_{i j})^2] )^{1/2} \leq \delta_n/\log(p n)
\end{equation}
holds with probability at least $1 - \beta_n$. In addition, by the assumption of the theorem,
\begin{equation}\label{eq: bernstein 2}
\max_{i\in[n]}\max_{j\in[p]}\left| \hat Z_{i j} - Z_{i j} - \frac{1}{n}\sum_{l=1}^n(\hat Z_{l j} - Z_{l j}) \right| \leq 2
\end{equation}
holds with probability at least $1 - \beta_n$. Therefore, applying the union bound and Bernstein's inequality (cf. \cite{VW96}, Lemma 2.2.9 or \cite{BLM13}, p.36) on the intersection of events \eqref{eq: bernstein 1} and \eqref{eq: bernstein 2} shows that with probability at least $1 - 2\beta_n$, for any $t > 0$,
$$
\Pr_e\left(\max_{j\in[p]}\left| \frac{1}{\sqrt n}\sum_{i=1}^n (e_i - 1)(\hat Z_i - Z_i) \right| > t \right) \leq 2p\exp\left( - \frac{n t^2/2}{n\delta_n^2/\log^2(p n) + 2\sqrt n t/3} \right).
$$
Hence, since $\delta_n \geq  \log^3(p n)/ \sqrt n$, there exists a universal constant $K_1>0$ such that for $\Delta_n := \delta_n /\sqrt{\log(p n)}$, with probability at least $1 - 2\beta_n$,
\begin{equation}\label{eq: proof of thm 3 main event}
\Pr_e\left(\max_{j\in[p]}\left| \frac{1}{\sqrt n}\sum_{i=1}^n (e_i - 1)(\hat Z_i - Z_i) \right| > K_1 \Delta_n \right) \leq \delta_n.
\end{equation}
The rest of the proof proceeds through the same steps as those in the proof of Theorem \ref{thm: bootstrap for mam} with \eqref{eq: proof of thm 3 main event} replacing \eqref{eq: proof of thm 2 main event} and applying Theorem \ref{thm: empirical bootstrap} instead of \ref{thm: multiplier bootstrap}. Note that in the case of E.2, we obtain $\delta_n \log^{1/3} n$ on the right-hand side of the bound \eqref{eq: bootstrap for mam e2} instead of $\delta_n$, as we had for the Gaussian bootstrap in Theorem \ref{thm: bootstrap for mam}, because the bound in Theorem \ref{thm: empirical bootstrap} is slightly worse than that in Theorem \ref{thm: multiplier bootstrap}.\qed

\noindent
{\bf Proof of Theorem \ref{thm: moderate deviations for mam}.}
We split the proof into three steps.

{\bf Step 1.} Here we show that there exists $n_{1}$ depending only on $\bar C$ such that for all $n\geq n_{1}$,
\begin{equation}\label{eq: thm 4 variance lower bound}
\Pr\left( \frac{1}{n}\sum_{i=1}^n Z_{i j}^2 \leq \frac{1}{2}\text{ for some }j\in[p]\right) \leq \frac{1}{p n}.
\end{equation}
To prove \eqref{eq: thm 4 variance lower bound}, we will use the following inequality: If $X_1,\dots,X_n$ are independent nonnegative random variables, then for any $t > 0$,
\begin{equation}\label{eq: massart}
\Pr\left( \sum_{i=1}^n (X_i - \Ep[X_i]) \leq - t \right) \leq \exp\left( -\frac{t^2}{2\sum_{i=1}^n \Ep[X_i^2]}\right),
\end{equation}
which can be found in Exercise 2.9 of \cite{BLM13}. Specifically, since
$$
\frac{(2\bar C)^3 B_n \log^{3/2}(p n)}{\sqrt n} \leq \delta_n\leq 1,
$$
it follows that there exists $n_{1}$ depending only on $\bar C$ such that for all $n\geq n_{1}$,
\begin{equation}\label{eq: thm 4 step 1 n1}
\frac{16 B_n^2\log(p n)}{n} = \frac{16 B_n^2\log^3(p n)}{n \log^2(p n)} \leq \frac{16}{(2\bar C)^6 \log^2(p n)} \leq 1.
\end{equation}
Then for all $n\geq n_{1}$ and $j=1,\dots,p$,
\begin{align}
&\Pr\left( \frac{1}{n}\sum_{i=1}^n Z_{i j}^2 \leq \frac{1}{2} \right) \leq \Pr\left( \frac{1}{n}\sum_{i=1}^n Z_{i j}^2 \leq \frac{1}{2 n}\sum_{i=1}^n \Ep[Z_{i j}^2] \right) = \Pr\left( \sum_{i=1}^n Z_{i j}^2 \leq \frac{1}{2}\sum_{i=1}^n \Ep[Z_{i j}^2] \right)\nonumber\\
&\qquad = \Pr\left( \sum_{i=1}^n(Z_{i j}^2 - \Ep[Z_{i j}^2]) \leq -\frac{1}{2}\sum_{i=1}^n \Ep[Z_{i j}^2] \right) \leq \exp\left( -\frac{(\sum_{i=1}^n \Ep[Z_{i j}^2])^2}{ 8\sum_{i=1}^n \Ep[Z_{i j}^4] } \right)\nonumber\\
&\qquad \leq \exp\left( - \frac{n^2}{8 n B_n^2}\right) \leq \exp(-2\log(p n)) = (p n)^{-2},\label{eq: thm 4 step 1}
\end{align}
where the inequality in the first line follows from Condition M, the inequality in the second line from \eqref{eq: massart} with $X_i = Z_{i j}^2$ and $t = 2^{-1}\sum_{i=1}^n\Ep[Z_{i j}^2]$, the first inequality in the third line from Condition M, and the second inequality in the third line from \eqref{eq: thm 4 step 1 n1}. Combining \eqref{eq: thm 4 step 1} with the union bound gives \eqref{eq: thm 4 variance lower bound}.

{\bf Step 2.} Here we show that for all $n\geq n_{1}$,
\begin{equation}\label{eq: thm 4 step 2}
\Pr\left(\max_{j\in[p]} \left| \frac{(\sum_{i=1}^n \hat Z_{i j}^2)^{1/2}}{(\sum_{i=1}^n Z_{i j}^2)^{1/2}} - 1\right|  > \frac{\delta_n\sqrt 2}{\log(p n)} \right)  \leq \beta_n + (p n)^{-1}.
\end{equation}
To show \eqref{eq: thm 4 step 2}, note that for all $n\geq n_{1}$, the left-hand side of \eqref{eq: thm 4 step 2} is equal to
\begin{align*}
& \Pr\left(\max_{j\in[p]} \frac{|(\sum_{i=1}^n \hat Z_{i j}^2)^{1/2} - (\sum_{i=1}^n Z_{i j}^2)^{1/2}|}{(\sum_{i=1}^n Z_{i j}^2)^{1/2}}  > \frac{\delta_n \sqrt 2}{\log(p n)} \right)\\
&\qquad
\leq \Pr\left(\max_{j\in[p]} \frac{(\sum_{i=1}^n (\hat Z_{i j} - Z_{i j})^2)^{1/2}}{(\sum_{i=1}^n Z_{i j}^2)^{1/2}}  > \frac{\delta_n \sqrt 2}{\log(p n)} \right)\\
&\qquad
\leq \Pr\left( \max_{j\in[p]} \left( \frac{1}{n}\sum_{i=1}^n (\hat Z_{i j} - Z_{i j})^2 \right)^{1/2} > \frac{\delta_n }{\log(p n)} \right) + (p n)^{-1} \leq \beta_n + (p n)^{-1},
\end{align*}
where the second line follows from the triangle inequality, and the third one from Step 1 and Condition A. This gives \eqref{eq: thm 4 step 2}.

{\bf Step 3.} Here we complete the proof. Fix $1 \leq x \leq \bar C\log^{1/2}(p n)$ and denote
\begin{equation}\label{eq: thm4 step 3-0}
\Delta_n(x):= K B_n (1 + x)^3 / \sqrt n \leq K(2\bar C)^3 B_n \log^{3/2}(p n)/\sqrt{n} \leq K\delta_n,
\end{equation}
where $K$ is a universal constant from Lemma \ref{lem: moderate deviations}. Now, fix $j\in[p]$ and observe that for all $n\geq n_{1}$,
\begin{align}
\Pr\left( \frac{n(\hat\theta_j - \theta_{0 j})}{( \sum_{i=1}^n \hat Z_{i j}^2 )^{1/2}} > x \right)
& \leq \Pr\left( \frac{n(\hat\theta_j - \theta_{0 j})}{( \sum_{i=1}^n Z_{i j}^2 )^{1/2}} >x - x\left| \frac{(\sum_{i=1}^n \hat Z_{i j}^2)^{1/2}}{( \sum_{i=1}^n Z_{i j}^2 )^{1/2}} - 1 \right| \right) \nonumber\\
& \leq \Pr\left(  \frac{n(\hat\theta_j - \theta_{0 j})}{( \sum_{i=1}^n Z_{i j}^2 )^{1/2}} >x - \frac{\sqrt 2 x \delta_n}{\log(p n)} \right) + \beta_n + (p n)^{-1},\label{eq: thm 4 step 3-1}
\end{align}
where the second inequality follows from Step 2. Further, by linearization \eqref{linearize},
$$
n(\hat\theta_j - \theta_{0 j}) = \sum_{i=1}^n Z_{i j} + \sqrt n r_{n j},
$$
and by Condition A,
$$
\Pr\left( \max_{j\in[p]}|r_{n j}| > \delta_n/\sqrt{\log (p n)} \right) \leq \beta_n.
$$
Hence, for all $n\geq n_{1}$, the probability in \eqref{eq: thm 4 step 3-1} is bounded from above by
\begin{equation}\label{eq: thm4 step 3-2}
\Pr\left( \frac{\sum_{i=1}^n Z_{i j}}{(\sum_{i=1}^n Z_{i j}^2)^{1/2}} > x - \frac{\sqrt 2 x \delta_n}{\log(p n)} - \frac{\sqrt 2 \delta_n}{\sqrt{\log (p n)}} \right) + \beta_n + (p n)^{-1},
\end{equation}
where we used Step 1 to bound $\sum_{i=1}^n Z_{i j}^2$. Next, since $\delta_n \leq 1$ and $x\geq 1$, there exists universal $n_{2}$ such that for all $n\geq n_2$, we have $2\sqrt 2\delta_n \leq \log(p n)$ and $x \geq 2\sqrt 2 \delta_n / \sqrt{\log (p n)}$. Then for all $n\geq n_{2}$,
$$
0 \leq x - \frac{\sqrt 2 x \delta_n}{\log(p n)} - \frac{\sqrt 2\delta_n}{\sqrt{\log (p n)}} \leq x \leq \bar C\log^{1/2}(p n) \leq n^{1/6}/B_n^{1/3},
$$
and so by Lemma \ref{lem: moderate deviations}, for all $n\geq n_{2}$, the probability in \eqref{eq: thm4 step 3-2} is bounded from above by
\begin{equation}\label{eq: thm4 step 3-3}
\left( 1 - \Phi\left( x - \frac{\sqrt 2 x \delta_n}{\log(p n)} - \frac{\sqrt 2 \delta_n}{\sqrt{\log p}} \right) \right)(1 + \Delta_n(x)).
\end{equation}
Further, for any $y,z \geq 0$, we have $\Phi(y + z) - \Phi(y) \leq z\phi(y)$, where $\phi$ is the pdf of the standard normal distribution, and for any $y \geq 1$, $\phi(y) \leq 2y(1 - \Phi(y))$; see Proposition 2.5 in \cite{D14}. Hence, denoting $\gamma_n := \sqrt 2 x \delta_n / \log(p n) + \sqrt 2\delta_n / \sqrt{\log (p n)}$, we have for all $n\geq n_{2}$ that
\begin{align*}
-\Phi\left( x - \frac{\sqrt 2 x \delta_n}{\log(p n)} - \frac{\sqrt 2 \delta_n}{\sqrt{\log (p n)}} \right)
&= -\Phi(x - \gamma_n) \leq - \Phi(x) + \gamma_n\phi(x - \gamma_n)\\
& \leq -\Phi(x) + \gamma_n\phi(x)e^{x\gamma_n} \leq -\Phi(x) + 2x\gamma_n e^{x\gamma_n}(1 - \Phi(x)).
\end{align*}
Therefore, for all $n\geq n_{2}$, the expression in \eqref{eq: thm4 step 3-3} is bounded from above by
$$
(1 - \Phi(x))(1 + 2 x \gamma_n e^{x \gamma_n})(1 + \Delta_n(x)).
$$
Thus, using
$$
x\gamma_n = \frac{\sqrt 2 x^2\delta_n}{\log(p n)} + \frac{ \sqrt 2 x \delta_n}{\sqrt{\log (p n)}} \leq  \sqrt 2(\bar C^2 + \bar C)\delta_n \leq  \sqrt 2(\bar C^2 + \bar C),
$$
and \eqref{eq: thm4 step 3-0} shows that there exists a constant $C$ depending only on $\bar C$ such that for all $n\geq n_0:=n_{1}\vee n_{2}$,
\begin{align*}
&\Pr\left( \frac{n(\hat\theta_j - \theta_{0 j})}{( \sum_{i=1}^n \hat Z_{i j}^2 )^{1/2}} > x \right) - (1 - \Phi(x))\leq C\Big(\delta_n (1 - \Phi(x)) + \beta_n + (p n)^{-1} \Big).
\end{align*}
This provides the upper bound in \eqref{eq: SNMD for mam 1}. To establish the lower bound, we use the same argument and note that for all $n\geq n_{2}$, we have
$$
0 \leq x + \frac{\sqrt 2 x \delta_n}{\log(p n)} + \frac{\sqrt 2\delta_n}{\sqrt{\log (p n)}} \leq 2 x \leq 2\bar C\log^{1/2}(p n) \leq n^{1/6}/B_n^{1/3},
$$
which ensures that we can use Lemma \ref{lem: moderate deviations}. This completes the proof of \eqref{eq: SNMD for mam 1}.

To prove \eqref{eq: thm 4 part 2}, we proceed like in the beginning of this step to show that for all $n\geq n_0$ and all $1\leq x\leq \bar C\log^{1/2}(p n)$,
$$
\Pr\left( \max_{j\in[p]}\frac{n(\hat\theta_j - \theta_{0 j})}{( \sum_{i=1}^n \hat Z_{i j}^2 )^{1/2}} > x \right) \leq \Pr\left(\max_{j\in[p]} \frac{\sum_{i=1}^n Z_{i j}}{( \sum_{i=1}^n Z_{i j}^2 )^{1/2}} > x \right) + C(\beta_n + (p n)^{-1}).
$$
In addition, by the union bound,
$$
\Pr\left(\max_{j\in[p]} \frac{\sum_{i=1}^n Z_{i j}}{( \sum_{i=1}^n Z_{i j}^2 )^{1/2}} > x \right) \leq p\max_{j\in[p]} \Pr\left(\frac{\sum_{i=1}^n Z_{i j}}{( \sum_{i=1}^n Z_{i j}^2 )^{1/2}} > x \right),
$$
and, by the arguments above, for all $n\geq n_0$, $j \in[p]$, and $1 \leq x\leq \bar C\log^{1/2}(p n)$,
$$
\Pr\left(\frac{\sum_{i=1}^n Z_{i j}}{( \sum_{i=1}^n Z_{i j}^2 )^{1/2}} > x \right) \leq (1 - \Phi(x))(1 + 2 x \gamma_n e^{x \gamma_n})(1 + \Delta_n(x)) \leq 1 - \Phi(x) + C\delta_n(1 - \Phi(x)),
$$
so that
$$
\Pr\left( \max_{j\in[p]}\frac{n(\hat\theta_j - \theta_{0 j})}{( \sum_{i=1}^n \hat Z_{i j}^2 )^{1/2}} > x \right) \leq p(1 - \Phi(x)) + C\Big(p\delta_n (1 - \Phi(x)) + \beta_n + (p n)^{-1} \Big).
$$
Setting $x = \Phi(1 - \alpha/p)$ here gives \eqref{eq: thm 4 part 2} and completes the proof of the theorem.
\qed

\noindent
{\bf Proof of Theorem \ref{lem: quantile comparison}.}
Without loss of generality, we can assume that $\delta_n \leq 1$ since otherwise the result is trivial (by choosing $C$ large enough). Also, there exists a universal constant $n_1$ such that $\log^4(p n)\geq (2\sqrt 2)^6$ for all $n\geq n_1$, and so under Condition E, for all $n\geq n_1$,
$$
\frac{(2\sqrt 2)^6 B_n^2 \log^3(p n)}{n} \leq \frac{(2\sqrt 2)^6 \delta_n^6}{\log^4(p n)} \leq \frac{(2\sqrt 2)^6 \delta_n^2}{\log^4(p n)} \leq \delta_n^2.
$$
In addition, let $n_0$ be a universal constant from Corollary \ref{cor: maximal inequality for mam}. Clearly, it suffices to prove the result for $n\geq n_0\vee n_1$ since the result for $n < n_0\vee n_1$ is trivial. Then, by Corollary \ref{cor: maximal inequality for mam},
$$
\Pr\left( \max_{j\in[p]} \frac{|\sqrt n(\hat\theta_j - \theta_{0 j})|}{(\En[\hat Z_{i j}^2])^{1/2}} > \sqrt{2\log(p n)}\right) \leq 2C(\beta_n + n^{-1}),
$$
where $C$ is a universal constant, and so with probability at least $1 - \beta_n - 2C(\beta_n + n^{-1})$, for all $j \in[p]$,
\begin{align*}
|\sqrt n(\hat w_j - w_j)(\hat\theta_j - \theta_{0 j})|
&= \frac{|\sqrt n(\hat\theta_j - \theta_{0 j})|}{(\En[(\hat Z_{i j})^2])^{1/2}} \times |\hat w_j - w_j|(\En[(\hat Z_{i j})^2])^{1/2}\\
&\leq \sqrt{2\log(p n)} \times (\delta_n/\log(p n)) = \sqrt 2\delta_n/\sqrt{\log(p n)},
\end{align*}
where we used Condition W. Also, by Conditions A and W, with probability at least $1 - \beta_n$, for all $j \in[p]$,
$$
\left| \sqrt n w_j(\hat\theta_j - \theta_{0 j}) - \frac{1}{\sqrt n}\sum_{i=1}^n w_j Z_{i j} \right| = w_j|r_{n j}| \leq C_W|r_{n j}| \leq C_W\delta_n / \sqrt{\log(p n)}.
$$
Conclude that
\begin{equation}\label{eq: linearization bands}
\sqrt n \hat W(\hat\theta - \theta_0) = \frac{1}{\sqrt n}\sum_{i= 1}^n W Z_i + \bar r_n,
\end{equation}
where $\bar r_n = (\bar r_{n 1},\dots\bar r_{n p})'$ is such that
\begin{equation}\label{eq: a1}
\Pr(\|\bar r_n\|_{\infty} > (C_W + \sqrt 2)\delta_n/\sqrt{\log(p n)}) \leq 2\beta_n + 2C(\beta_n + n^{-1}).
\end{equation}
Next, by the triangle inequality,
$$
(\En[(\hat w_j \hat Z_{i j} - w_j Z_{i j})^2])^{1/2} \leq |\hat w_j - w_j|(\En[(\hat Z_{i j})^2])^{1/2} + w_j(\En[(\hat Z_{i j} - Z_{i j})^2])^{1/2},
$$
and so by Conditions A and W,
\begin{equation}\label{eq: a2}
\Pr\left(\max_{j\in[p]} \En[(\hat w_j \hat Z_{i j} - w_j Z_{i j})^2] > (C_W + 1)^2 \delta_n^2/\log^2(p n)\right) \leq 2\beta_n.
\end{equation}
Thus, applying Theorem \ref{thm: clt for mam} for linearization \eqref{eq: linearization bands}, with \eqref{eq: a1} and \eqref{eq: a2} playing the role of Condition A, shows that
\begin{equation}\label{eq: conf thm 1 bands}
\sup_{A\in\mathcal A}\Big| \Pr(\sqrt n \hat W(\hat\theta - \theta_0) \in A) - \Pr(W N(0,V) \in A) \Big| \leq (C/2)(\delta_n + \beta_n + n^{-1}) \leq C(\delta_n + \beta_n)
\end{equation}
for some constant $C$ depending only on $C_W$, and applying Theorem \ref{thm: bootstrap for mam} shows that
\begin{equation}\label{eq: conf thm 2 bands}
\sup_{A\in\mathcal A}\Big| \Pr_e(\sqrt n \hat W(\hat\theta^* - \hat\theta) \in A) - \Pr(W N(0,V) \in A) \Big| \leq C\delta_n
\end{equation}
holds with probability at least $1 - 2\beta_n - n^{-1}$ in the case of E.1 and $1 - 2\beta_n - (\log n)^{-2}$ in the case of E.2 since the proof of Theorem \ref{thm: bootstrap for mam} only requires the second part of Condition A.

We are now able to prove \eqref{eq: quantile comparison 1} and \eqref{eq: quantile comparison 2}. To prove \eqref{eq: quantile comparison 1}, denoting $\epsilon_n := C(\delta_n + \beta_n)$, we have by \eqref{eq: conf thm 1 bands} that
\begin{align*}
\Pr(\| \sqrt n \hat W (\hat \theta - \theta_0) \|_{\infty} \leq \lambda^g(1 - \alpha - \epsilon_n))
& \leq \Pr(\|N(0,V)\|_{\infty} \leq \lambda^g(1 - \alpha - \epsilon_n)) + \epsilon_n\\
& \leq 1 - \alpha - \epsilon_n + \epsilon_n = 1 - \alpha,
\end{align*}
and, similarly,
$$
\Pr(\| \sqrt n \hat W (\hat \theta - \theta_0) \|_{\infty} \leq \lambda^g(1 - \alpha + \epsilon_n)) \geq 1 - \alpha.
$$
Thus,
$$
\lambda^g(1 - \alpha - \epsilon_n) \leq \lambda(1 - \alpha) \leq \lambda^g(1 - \alpha + \epsilon_n),
$$
which gives \eqref{eq: quantile comparison 1}. To prove \eqref{eq: quantile comparison 2}, we use the same argument but apply \eqref{eq: conf thm 2 bands} instead of \eqref{eq: conf thm 1 bands}.

Finally, we prove \eqref{eq: gaussian quantile bound 1}. Let $\xi$ be a $N(0,1)$ random variable. Then for any $a\in(0,1)$,
\begin{align*}
\Pr(\| W N(0,V) \|_{\infty} > \bar\sigma\Phi^{-1}(1 - a/(2p)))
&\leq\sum_{j=1}^p \Pr(|w_j V_{j j}^{1/2} \xi | > \bar\sigma\Phi^{-1}(1 - a/(2p)))\\
& \leq \sum_{j=1}^p \Pr(|\xi | > \Phi^{-1}(1 - a/(2 p))) = a,
\end{align*}
which implies that $\lambda^g(1 - a) \leq \bar\sigma \Phi^{-1}(1 - a/(2 p))$. To prove that $\Phi^{-1}(1 - a/(2 p)) \leq (2\log(2p/a))^{1/2}$, we apply the inequality $1 - \Phi(x) \leq \exp(-x^2/2)$, which holds for all $x>0$ as discussed in Proposition 2.5 of \cite{D14}, with $x = (2\log(2p /a))^{1/2}$. This completes the proof of the lemma.
\qed

\noindent
{\bf Proof of Theorem \ref{thm: simultaneous conf intervals}.}
By Theorem \ref{lem: quantile comparison} and its proof,
\begin{align*}
&\Pr(\| \sqrt n \hat W(\hat\theta - \theta_0) \|_{\infty} \leq \hat\lambda(1 - \alpha))
\leq \Pr(\| \sqrt n \hat W(\hat\theta - \theta_0) \|_{\infty} \leq \lambda^g(1 - \alpha + \epsilon_n)) + o(1)\\
&\quad = \Pr(\| W N(0,V) \|_{\infty} \leq \lambda^g(1 - \alpha + \epsilon_n)) + o(1) = 1 - \alpha + \epsilon_n + o(1) = 1 - \alpha + o(1),
\end{align*}
and, similarly,
\begin{align*}
&\Pr(\| \sqrt n \hat W(\hat\theta - \theta_0) \|_{\infty} \leq \hat\lambda(1 - \alpha))
\geq \Pr(\| \sqrt n \hat W(\hat\theta - \theta_0) \|_{\infty} \leq \lambda^g(1 - \alpha - \epsilon_n)) + o(1)\\
&\quad = \Pr(\|W N(0,V) \|_{\infty} \leq \lambda^g(1 - \alpha - \epsilon_n)) + o(1) = 1 - \alpha - \epsilon_n + o(1) = 1 - \alpha + o(1).
\end{align*}
Hence,
$$
\Pr(\| \sqrt n \hat W(\hat\theta - \theta) \|_{\infty} \leq \hat\lambda(1 - \alpha)) = 1 - \alpha + o(1),
$$
which implies that the confidence intervals in \eqref{eq: simultaneous conf intervals} satisfy \eqref{eq: conf intervals main property}, which is the first asserted claim. To prove the second asserted claim, simply note that by Theorem \ref{lem: quantile comparison}, with probability $1 - o(1)$, $\hat\lambda(1 - \alpha) \leq \lambda^g(1 - \alpha + \epsilon_n)$. This completes the proof of the theorem.
\qed

\noindent
{\bf Proof of Theorem \ref{thm: fwer bonferroni}.}
Since the set of null hypotheses $H_j$ rejected by the Bonferroni-Holm procedure always contains the set of null hypotheses rejected by the Bonferroni procedure, it suffices to prove the result for the Bonferroni-Holm procedure. To do so, note that whenever $w'\subset w''$, it follows that $|w'| \leq |w''|$, and so
$$
c_{1-\alpha,w'} = \Phi^{-1}(1 - \alpha/|w'|) \leq \Phi^{-1}(1 - \alpha/|w''|) = c_{1-\alpha,w''}.
$$
Thus, \eqref{eq: critical value property 1} is satisfied, and we only need to verify \eqref{eq: conditions MHT}. To this end, take any $w\subset \mathcal W$ and $P\in\mathcal P^w$.
Denote $\bar p := |w|$. By discussion immediately after Theorem \ref{thm: moderate deviations for mam},
$$
\Phi^{-1}(1 - \alpha/\bar p) \leq (2\log(\bar p/\alpha))^{1/2} \leq (2\log(\bar p n))^{1/2}
$$
for all $n\geq 1/\alpha$. Together with $(2\sqrt 2)^3B_n\log^{3/2}(\bar p n)/\sqrt n \leq \delta_n$, this makes it possible to apply Theorem \ref{thm: moderate deviations for mam}, \eqref{eq: thm 4 part 2}, with $\bar C = \sqrt 2$, to show that
\begin{align}
&\Pr_P\left(\max_{j\in w} \frac{\sqrt n(\hat\theta_j - \bar\theta_{0 j})}{(\frac{1}{n}\sum_{i=1}^n \hat Z_{i j}^2)^{1/2}} > \Phi^{-1}(1-\alpha/\bar p) \right) \nonumber\\
&\quad \leq \Pr_P\left(\max_{j\in w} \frac{\sqrt n(\hat\theta_j - \theta_{0 j})}{(\frac{1}{n}\sum_{i=1}^n \hat Z_{i j}^2)^{1/2}} > \Phi^{-1}(1-\alpha/\bar p) \right) \leq \alpha + C \Big(\alpha \delta_n + \beta_n + n^{-1} \Big)\label{eq: bonferroni procedure proof}
\end{align}
for some universal constant $C$ and all $n \geq n_0\vee n_1$, where $n_0$ is a universal constant appearing in Theorem \ref{thm: moderate deviations for mam} and $n_1$ is such that for all $n\geq n_1$, we have $\delta_n \leq 1$ (recall that $\delta_n\searrow0$). Since the right-hand side of \eqref{eq: bonferroni procedure proof} and $n_0\vee n_1$ do not depend on $(w,P)$ and $C(\alpha\delta_n + \beta_n + n^{-1}) = o(1)$, \eqref{eq: conditions MHT} holds, and the asserted claim follows.
\qed

\noindent
{\bf Proof of Theorem \ref{thm: MHT}.}
Since \eqref{eq: critical value property 1} holds trivially, it suffices to prove \eqref{eq: conditions MHT}. To do so, we will use the following lemma.

\begin{lemma}\label{lem: verification of w}
For $j \in[p]$, denote $\hat w_j := (n^{-1}\sum_{i=1}^n \hat Z_{i j}^2)^{-1/2}$ and $w_j := V_{j j}^{-1/2}$. Then Conditions M, E, and A imply that there exists universal constants $n_0\in\mathbb N$ and $C\geq 1$ such that for all $n\geq n_0$, Condition W holds with $C_W = 1$ and with $\delta_n$ and $\beta_n$ replaced by $C\delta_n$ and $\beta_n + \delta_n + (p n)^{-1}$, respectively.
\end{lemma}
\begin{proof}
The first part of Condition W holds trivially with $C_W = 1$. To prove the second part, we will assume, without loss of generality, that $\delta_n \leq 1$. Observe that by the triangle inequality and Condition M, for all $j \in [p]$,
\begin{align}
& |\hat w_j - w_j| (\En[\hat Z_{i j}^2])^{1/2} = \frac{| (\frac{1}{n}\sum_{i=1}^n \hat Z_{i j}^2)^{1/2} - (\frac{1}{n}\sum_{i=1}^n \Ep[Z_{i j}^2])^{1/2}|}{(\frac{1}{n}\sum_{i=1}^n \Ep[Z_{i j}^2])^{1/2}}\nonumber\\
&\qquad \leq \frac{| (\frac{1}{n}\sum_{i=1}^n \hat Z_{i j}^2)^{1/2} - (\frac{1}{n}\sum_{i=1}^n Z_{i j}^2)^{1/2} |}{(\frac{1}{n}\sum_{i=1}^n \Ep[Z_{i j}^2])^{1/2}} + \frac{| (\frac{1}{n}\sum_{i=1}^n Z_{i j}^2)^{1/2} - (\frac{1}{n}\sum_{i=1}^n \Ep[Z_{i j}^2])^{1/2} |}{(\frac{1}{n}\sum_{i=1}^n \Ep[Z_{i j}^2])^{1/2}}\nonumber\\
&\qquad \leq \left( \frac{1}{n}\sum_{i=1}^n (\hat Z_{i j} - Z_{i j})^2 \right)^{1/2} + \left| \frac{1}{n}\sum_{i=1}^n Z_{i j}^2 - \frac{1}{n}\sum_{i=1}^n \Ep[Z_{i j}^2] \right|.\label{eq: thm 7 proof of w}
\end{align}
The first term in \eqref{eq: thm 7 proof of w} is bounded from above by $\delta_n/\log(p n)$ with probability at least $1 - \beta_n$ uniformly over all $j \in[p]$ by Condition A. Also, by Condition E and Lemma \ref{lem: maximal ineq},
$$
\Ep\left[ \max_{j\in[p]}\left| \frac{1}{n}\sum_{i=1}^n Z_{i j}^2 - \frac{1}{n}\sum_{i=1}^n \Ep[Z_{i j}^2] \right| \right] \leq K_1\left( \sqrt{\frac{B_n^2\log p}{n}} + \left( \Ep\left[ \max_{i\in[n]}\max_{j\in[p]} Z_{i j}^4 \right] \right)^{1/2}\frac{\log p}{n} \right),
$$
where $K_1$ is a universal constant. In addition, in the case of E.1, for some universal constant $K_2$,
$$
\left( \Ep\left[ \max_{i\in[n]}\max_{j\in[p]} Z_{i j}^4 \right] \right)^{1/2} \leq K_2 B_n^2\log^2(p n),
$$
and in the case of E.2,
$$
\left( \Ep\left[ \max_{i\in[n]}\max_{j\in[p]} Z_{i j}^4 \right] \right)^{1/2} \leq B_n^2\sqrt n.
$$
Thus, in both cases, since $\delta_n \leq 1$, by Condition E,
$$
\Ep\left[ \max_{j\in[p]}\left| \frac{1}{n}\sum_{i=1}^n Z_{i j}^2 - \frac{1}{n}\sum_{i=1}^n \Ep[Z_{i j}^2] \right| \right] \leq C\delta_n^2/\log(p n),
$$
for some universal constant $C$, so that by Markov's inequality, with probability at least $1 - \delta_n$,
$$
\max_{j\in[p]}\left| \frac{1}{n}\sum_{i=1}^n Z_{i j}^2 - \frac{1}{n}\sum_{i=1}^n \Ep[Z_{i j}^2] \right| \leq C\delta_n/\log(p n).
$$
Hence,
$$
\Pr\left(\max_{j\in[p]}|\hat w_j - w_j|^2\En[\hat Z_{i j}^2] > (1 + C)^2\delta_n^2/\log^2(p n)\right)\leq \beta_n + \delta_n.
$$
Moreover, by Conditions E and A, as in Step 1 of the proof of Theorem \ref{thm: moderate deviations for mam}, there exists a universal constant $n_0$ such that for all $n\geq n_0$,
$$
\left(\frac{1}{n}\sum_{i=1}^n \hat Z_{i j}^2\right)^{1/2} \geq \left(\frac{1}{n}\sum_{i=1}^n Z_{i j}^2\right)^{1/2} - \left(\frac{1}{n}\sum_{i=1}^n (\hat Z_{i j} - Z_{i j})^2\right)^{1/2} \geq \frac{1}{\sqrt 2} - \frac{\delta_n}{\log(p n)} \geq \frac{1}{2}
$$
for all $j\in[p]$ with probability at least $1 - \beta_n - (p n)^{-1}$. Thus, for all $n\geq n_0$,
$$
\max_{j\in[p]}|\hat w_j - w_j|^2(1 + \En[\hat Z_{i j}^2]) \leq 5 \max_{j\in[p]}|\hat w_j - w_j|^2 \En[\hat Z_{i j}^2] \leq 5(1+C)^2\delta_n^2/\log^2(p n)
$$
with probability at least $1 - \beta_n - \delta_n - (p n)^{-1}$, where we have $\beta_n$ instead of $2\beta_n$ since the same event $\max_{j\in[p]}(\En[(\hat Z_{i j} - Z_{i j})^2])^{1/2} \leq \delta_n/\log(p n)$ is used twice. This gives the asserted claim.
\end{proof}
Getting back to the proof of Theorem \ref{thm: MHT}, take any $w\subset\mathcal W$ and $P\in\mathcal P^w$. Also, for $j \in[p]$, denote $\hat w_j := (n^{-1}\sum_{i=1}^n \hat Z_{i j}^2)^{-1/2}$, $w_j := V_{j j}^{-1/2}$, and
$$
\bar t_j := \frac{\sqrt n(\hat\theta_j - \theta_{0 j})}{(\frac{1}{n}\sum_{i=1}^n \hat Z_{i j})^{1/2}} = \sqrt n \hat w_j (\hat\theta_j - \theta_{0 j}).
$$
Then
$$
\max_{j\in w} t_j \leq \max_{j\in w}\bar t_j,
$$
and Condition W holds by Lemma \ref{lem: verification of w}. Now, under conditions M, E, A, and W, we can proceed like in the proof of Theorems \ref{lem: quantile comparison} and \ref{thm: simultaneous conf intervals}, with probabilities like $\Pr_P(\max_{j\in w}\bar t_j \leq c_{1 - \alpha,w})$ replacing probabilities like $\Pr_P(\max_{1\leq j\leq p}|\bar t_j| \leq c_{1 - \alpha,w})$, to obtain
$$
\Pr_P\left(\max_{j\in w}\bar t_j \leq c_{1 - \alpha,w}\right) \geq 1 - \alpha + o(1),
$$
where the term $o(1)$ depends only on $(\delta_n,\beta_n)$ and on whether E.1 or E.2 is used. In particular, the term $o(1)$ does not depend on $(w,P)$. This gives \eqref{eq: conditions MHT} and completes the proof of the theorem.
\qed

\noindent
{\bf Proof of Theorem \ref{thm: fdp}.}
In this proof, all convergence results hold uniformly over $P\in\mathcal P$, and so we fix $P\in\mathcal P$, and drop the index $P$, i.e. we write, for example, $\Pr$ instead of $\Pr_P$. Also, without loss of generality, we assume that $|\mathcal H| \leq \sqrt{\log p}$ (if $|\mathcal H| > \sqrt{\log p}$, we can redefine $\mathcal H$ by keeping only the first $\sqrt{\log p}$ elements of it). We split the proof into six steps.

{\bf Step 1.} Here we show that
\begin{equation}\label{eq: thm 6 step 1}
\Pr\left(\sum_{j=1}^p 1\{t_j > \sqrt{2\log p}\} \geq |\mathcal H|\right) \to 1.
\end{equation}
To show \eqref{eq: thm 6 step 1}, note that by Chebyshev's inequality, for all $j\in \mathcal H$ and $\epsilon > 0$,
$$
\Pr\left( \left| \frac{1}{n}\sum_{i=1}^n Z_{i j}^2 - V_{j j} \right| > \epsilon \right)\leq \frac{\sum_{i=1}^n\Ep[Z_{i j}^4]}{\epsilon^2 n^2} \leq \frac{B^2}{\epsilon^2 n}.
$$
Setting $\epsilon = V_{j j} \delta_n$ and recalling that by Condition $M$, $V_{j j} \geq 1$ for all $j\in\mathcal H$ gives
$$
\Pr\left( \left|\frac{\sum_{i=1}^n Z_{i j}^2}{n V_{j j}} - 1\right| > \delta_n \right) \leq \frac{B^2}{V_{j j}^2 \delta_n^2 n} \leq \frac{B^2}{\delta_n^2 n},
$$
and so by the union bound,
$$
\Pr\left( \max_{j\in\mathcal H} \left| \frac{\sum_{i=1}^n Z_{i j}^2}{n V_{j j}} - 1 \right| > \delta_n \right) \leq \frac{B^2\sqrt{\log p}}{\delta_n^2 n} = o(1).
$$
Thus, since $|x/y - 1| \leq |(x/y)^2 - 1|$ for all $x,y > 0$, it follows that
$$
\Pr\left( \max_{j\in\mathcal H}\left| \frac{( \sum_{i=1}^n Z_{i j}^2 )^{1/2}}{( n V_{j j})^{1/2}} - 1 \right| > \delta_n \right) \leq \Pr\left( \max_{j\in\mathcal H} \left| \frac{\sum_{i=1}^n Z_{i j}^2}{n V_{j j}} - 1 \right| > \delta_n \right) \leq \frac{B^2\sqrt{\log p}}{\delta_n^2 n} = o(1).
$$
Also, by Step 2 of the proof of Theorem \ref{thm: moderate deviations for mam},
$$
\Pr\left( \max_{j\in\mathcal H}\left| \frac{( \sum_{i=1}^n \hat Z_{i j}^2 )^{1/2}}{( \sum_{i=1}^n Z_{i j}^2 )^{1/2}} - 1 \right| > \delta_n \right) = o(1).
$$
Therefore, recalling that $\hat T_j = n| \hat\theta_j |/( \sum_{i=1}^n \hat Z_{i j}^2 )^{1/2}$,
\begin{align}
&\Pr\left(\textstyle\sum_{j=1}^p 1\{ t_j > \sqrt{2\log p}\} \geq |\mathcal H|\right)\nonumber\\
&\qquad \geq \Pr\left(\textstyle\sum_{j\in\mathcal H} 1\{ t_j > \sqrt{2\log p}\} \geq |\mathcal H|\right)
 = \Pr\left(t_j > \sqrt{2\log p}\text{ for all }j\in\mathcal H\right)\nonumber\\
& \qquad \geq \Pr\left(\sqrt n(\hat \theta_j - \bar\theta_{0 j}) / V_{j j}^{1/2} > \sqrt{2\log p}(1 + \delta_n)^2\text{ for all }j\in\mathcal H\right) - o(1)\label{eq: thm 6 step 1 - 2}.
\end{align}
Next, since $\delta_n \to 0$ and $\sqrt n(\theta_{0 j} - \bar\theta_{0 j} )/V_{j j}^{1/2} > \sqrt{4\log p}$ for all $j\in\mathcal H$ by Condition L, it follows that the probability in \eqref{eq: thm 6 step 1 - 2} is bounded from below by
\begin{align*}
&\Pr\left( \sqrt n|\hat \theta_j - \theta_{0 j}| / V_{j j}^{1/2} \leq (1/2)\sqrt{\log p}\text{ for all }j\in\mathcal H \right)\\
&\qquad \geq \Pr\left( \left|(n V_{j j})^{-1/2}\textstyle\sum_{i=1}^n Z_{i j}\right| \leq (1/2)\sqrt{\log p} - \delta_n / \sqrt{\log p}\text{ for all }j\in\mathcal H \right) - o(1)\\
&\qquad \geq \Pr\left( \left|(n V_{j j})^{-1/2}\textstyle\sum_{i=1}^n Z_{i j}\right| \leq (1/3)\sqrt{\log p}\text{ for all }j\in\mathcal H \right) - o(1)\\
&\qquad \geq 1 - \sqrt{\log p}/((1/3)\sqrt{\log p})^2 - o(1) = 1 - 9/\sqrt{\log p} = 1 - o(1),
\end{align*}
where the second line follows from Conditions A and M, the third from $\delta_n \to 0$, and the fourth from Chebyshev's inequality, the union bound, and the fact that $p  \to \infty$. This gives \eqref{eq: thm 6 step 1} and completes the first step.

{\bf Step 2.} Here we show that for any $j = 1,\dots,p$, the inequality $t_{j} \geq t_{(\hat k)}$ holds if and only if $\hat T_j \geq \hat t$, where
\begin{equation}\label{eq: thm 6 step 2}
\hat t := \inf\left\{ t \in \mathbb R\colon 1 - \Phi(t) \leq \frac{\alpha\max\{\sum_{j=1}^p 1\{t_j \geq t\},1\}}{p} \right\}.
\end{equation}
To prove this equivalence, we consider two cases separately: $\hat k = 0$ and $\hat k \geq 1$. First, suppose that $\hat k = 0$. Then for all $l \in [p]$,
$$
1 - \Phi(t_{(l)}) > \frac{\alpha l}{p} = \frac{\alpha\max\{\sum_{j=1}^p 1\{t_j \geq t_{(l)}\}, 1 \}}{p}.
$$
Therefore, denoting $t_{(p + 1)} : = -\infty $, it follows that for all $t\leq t_{(1)}$, there exists $l \in [p]$ such that $t_{(l+1)} < t \leq t_{(l)}$ and
$$
1 - \Phi(t) \geq 1 - \Phi(t_{(l)}) > \frac{\alpha\max\{\sum_{j=1}^p 1\{t_j \geq t_{(l)}\},1\}}{p} = \frac{\alpha\max\{\sum_{j=1}^p 1\{t_j \geq t\},1\}}{p}.
$$
Hence, $\hat t > t_{(1)}$; so for all $j \in[p]$, we have $t_j < t_{(\hat k)}$ and $t_j < \hat t$, so that equivalence holds. Next, suppose that $\hat k \geq 1$. Then, by the same argument as above, $\hat t > t_{(\hat k + 1)}$. Also,
$$
1 - \Phi(t_{(\hat k)}) \leq \frac{\alpha\hat k}{p} = \frac{\alpha\max\{ \sum_{j=1}^p1\{ t_j \geq t_{(\hat k)} \} , 1\}}{p}
$$
by definition of $\hat k$, and so $\hat t \leq t_{(\hat k)}$. Therefore, $t_{(\hat k + 1)} < \hat t \leq t_{(\hat k)}$.  Thus, for all $j \in[p]$, the inequality $t_{j} \geq t_{(\hat k)}$ holds if and only if $t_j \geq \hat t$, so that equivalence holds again. This completes the second step.

{\bf Step 3.} Here we show that $\hat t$ satisfies
\begin{equation}\label{eq: thm 6 step 3 - 1}
\frac{p G(\hat t)}{\max\{\sum_{j=1}^p 1\{t_j \geq \hat t \} , 1\}} = \alpha
\end{equation}
and
\begin{equation}\label{eq: thm 6 step 3 - 2}
\Pr(\hat t \leq G^{-1}(\alpha|\mathcal H|/p) \to 1,
\end{equation}
where $G(t) := 1 - \Phi(t)$ for all $t\geq 0$. Indeed, \eqref{eq: thm 6 step 3 - 1} follows immediately from \eqref{eq: thm 6 step 2} since $\Phi$ is continuous. To prove \eqref{eq: thm 6 step 3 - 2}, note that for all all $x>0$, $1 - \Phi(x) \leq e^{-x^2/2}$ by Proposition 2.5 in \cite{D14}. Therefore, as long as $|\mathcal H| \geq 1/\alpha$, which holds for all $n$ large enough,
\begin{equation}\label{eq: PHI bound}
\alpha|\mathcal H|/p
\geq 1/p = \exp\left(-\frac{(\sqrt{2\log p})^2}{2}\right)\geq \left( 1 - \Phi(\sqrt{2\log p}) \right) = G(\sqrt{2\log p}).
\end{equation}
Hence, since $G$ is strictly decreasing,
\begin{equation}\label{eq: G bound thm 6}
G^{-1}(\alpha|\mathcal H|/p) \leq \sqrt{2\log p}.
\end{equation}
Therefore, setting $t = G^{-1}(\alpha|\mathcal H|/p)$, it follows from Step 1 that wp $\to 1$,
$$
1 - \Phi(t) = G(t) = \frac{\alpha|\mathcal H|}{p} \leq \frac{\alpha\max\{\sum_{j=1}^p1\{ t_j \geq t \} , 1\}}{p},
$$
and so by definition of $\hat t$, wp $\to 1$, $\hat t \leq t = G^{-1}(\alpha|\mathcal H|/p)$. This gives \eqref{eq: thm 6 step 3 - 2} and completes the third step.

{\bf Step 4.} Here we show that for any sequence of constants $(\gamma_n)_{n\geq 1}$ such that $\gamma_n \to 0$,
\begin{equation}\label{eq: thm 6 technical}
\sup_{0\leq t\leq (2\log p)^{1/2}}\left|1 - \frac{G(t + \gamma_n/\sqrt{\log p})}{G(t)}\right| = o(1).
\end{equation}
To prove \eqref{eq: thm 6 technical}, note that it is immediate that
$$
\sup_{0\leq t\leq 1}\left|1 - \frac{G(t + \gamma_n/\sqrt{\log p})}{G(t)}\right| = o(1).
$$
Also, as long as $|\gamma_n|/\sqrt{\log p} \leq 1$, which holds for all $n$ large enough, uniformly over $1< t\leq (2\log p)^{1/2}$,
\begin{align*}
|G(t) - G(t + \gamma_n/\sqrt{\log p})| &= |\Phi(t + \gamma_n/\sqrt{\log p}) - \Phi(t)| = \frac{|\gamma_n|(\phi(t)\vee\phi(t + \gamma_n/\sqrt{\log p}))}{\sqrt{\log p}}\\
& = \frac{(1+o(1))|\gamma_n|\phi(t)}{\sqrt{\log p}}\leq \frac{2(1+o(1))t|\gamma_n|G(t)}{\sqrt{\log p}} = o(1)G(t),
\end{align*}
where the last inequality follows from the fact that for all $x\geq 1$, $\phi(x) \leq 2x(1 - \Phi(x)) = 2xG(x)$; see Proposition 2.5 in \cite{D14}. This gives \eqref{eq: thm 6 technical}.

{\bf Step 5.} Here we show that
\begin{equation}\label{eq: thm 6 lln}
\sup_{0\leq t \leq G^{-1}(\alpha|\mathcal H|/p)}\frac{\sum_{j\in\mathcal H_0}1\{t_j \geq t\}}{p_0G(t)}  \leq 1 + o_p(1),
\end{equation}
where $p_0 := p - p_1$ is the number of true null hypotheses. To prove \eqref{eq: thm 6 lln}, for $j \in [p]$, let $\bar t_j:=\sqrt n(\hat\theta_j - \theta_{0 j})/(\En[\hat Z_{i j}^2])^{1/2}$, so that $\bar t_j \geq t_j $ for all $j\in\mathcal H_0$. We will show below that
\begin{equation}\label{eq: thm 6 lln - 2}
\sup_{0\leq t \leq G^{-1}(\alpha|\mathcal H|/p)}\left|\frac{\sum_{j\in\mathcal H_0}1\{\bar t_j \geq t\}}{p_0G(t)} - 1\right| = o_p(1),
\end{equation}
which implies \eqref{eq: thm 6 lln}. To prove \eqref{eq: thm 6 lln - 2}, for all $j \in [p]$, denote
$$
T_j := \frac{\sum_{i=1}^n Z_{i j}}{(\sum_{i=1}^n Z_{i j}^2)^{1/2}}\text{ and }\bar T_j := \frac{\frac{1}{\sqrt n}\sum_{i=1}^n Z_{i j}}{\left( \frac{1}{n}\sum_{i=1}^n Z_{i j}^2 - \left(\frac{1}{n}\sum_{i=1}^n Z_{i j}\right)^2 \right)^{1/2}},
$$
so that
$$
\bar T_j = \frac{T_j}{\sqrt{1 - T_j^2 / n}}.
$$
Equation (13) on page 2016 of \cite{LS14} shows that for any sequence of positive constants $(d_n)_{n\geq 1}$ such that $d_n \to \infty$ and $d_n = o(p)$ as $n\to\infty$,
\begin{equation}
\sup_{0\leq t\leq G_{\kappa}^{-1}(d_n/p)}\left|\frac{\sum_{j\in\mathcal H_0}1\{\bar T_j \geq t\}}{p_0 G_{\kappa}(t)} - 1\right| = o_p(1),
\end{equation}
where $G_{\kappa}$ is some function such that $G_{\kappa}(t)\geq G(t)$ for all $t\in\mathbb R$ and, given that $\log^3 p / n = o(1)$, $G_\kappa(t) = G(t)(1 + o(1))$ uniformly over $0\leq t\leq \sqrt{2\log p}$ (in fact, \cite{LS14} studied the case with absolute values, $|\bar T_j|$, but their argument works for $\bar T_j$ instead of $|\bar T_j|$ as well). Therefore, since we have $G^{-1}(\alpha|\mathcal H|/(2p)) \leq \sqrt{2\log p}$ for all $n$ large enough, which can be established by the same arguments as those leading to \eqref{eq: G bound thm 6},
\begin{equation}\label{eq: thm 6 Liu Shao bound}
\sup_{0\leq t\leq G^{-1}(\alpha|\mathcal H|/(2 p))}\left|\frac{\sum_{j\in\mathcal H_0}1\{\bar T_j \geq t\}}{p_0 G(t)} - 1\right| = o_p(1).
\end{equation}
Next, using the inequality $1 - \Phi(x) \leq \exp(-x^2/2)$ with $x = (4\log p)^{1/2}$ gives $\Phi^{-1}(1 - 1/p^2) \leq (4\log p)^{1/2}$. Therefore, using Theorem \ref{thm: moderate deviations for mam}, \eqref{eq: thm 4 part 2}, with $\alpha = 1/p = o(1)$ shows that
$$
\max_{j\in [p]}|\bar t_j| = O_p(\sqrt{\log p})\text{ and }\max_{1\leq j\leq p}|T_j| = O_p(\sqrt{\log p}).
$$
Thus,
\begin{align*}
\Big|\bar t_j - T_j\Big|
& = \left| \frac{n(\hat \theta_j - \theta_{0 j})}{(\sum_{i=1}^n \hat Z_{i j}^2)^{12}} - \frac{\sum_{i=1}^n Z_{i j}}{(\sum_{i=1}^n Z_{i j}^2)^{1/2}} \right|\\
&\leq |\bar t_j| \times \left| 1 - \frac{(\sum_{i=1}^n \hat Z_{i j}^2)^{1/2}}{(\sum_{i=1}^n Z_{i j}^2)^{1/2}} \right| + \frac{|\sqrt n r_{n j}|}{(\sum_{i=1}^n Z_{i j}^2)^{1/2}} = o_p(1/\sqrt{\log p})
\end{align*}
uniformly over $j \in [p]$ by Steps 1 and 2 in the proof of Theorem \ref{thm: moderate deviations for mam} and Condition A. Also,
\begin{align*}
|\bar T_j  - T_j|
& = |T_j|\times \left| \frac{1}{\sqrt{1 - T_j^2/n}} - 1 \right| \leq |T_j|\times \left| \frac{1}{1 - T_j^2/n} - 1 \right| = \frac{|T_j|^3/n}{1 - T_j^2/n} = o_p(1/\sqrt{\log p})
\end{align*}
uniformly over $j \in [p]$ since $\log^4 p / n = o(1)$. Hence, there exists a sequence of positive constants $(\gamma_n)_{n\geq 1}$ such that wp $\to 1$,
$$
\max_{j\in[p]}|\bar t_j - \bar T_j| \leq \gamma_n/\sqrt{\log p}.
$$
Further, by Step 4 and \eqref{eq: G bound thm 6}, for all $0 \leq t \leq G^{-1}(\alpha|\mathcal H|/p)$ and all $n$ large enough,
$$
G(t + \gamma_n/\sqrt{\log p}) = G(t)(1 + o(1)) \geq \alpha|\mathcal H|/(2 p),
$$
and so
$$
t + \gamma_n/\sqrt{\log p} \leq G^{-1}(\alpha|\mathcal H|/(2 p)).
$$
Hence, uniformly over $0\leq t\leq G^{-1}(\alpha |\mathcal H| / p)$, wp $\to 1$,
$$
\frac{\sum_{j\in\mathcal H_0}1\{\bar t_j \geq t\}}{p_0 G(t)} \geq \frac{\sum_{j\in\mathcal H_0}1\{ \bar T_j \geq t + \gamma_n/\sqrt{\log p}\}}{p_0 G(t)} = \frac{\sum_{j\in\mathcal H_0}1\{ \bar T_j \geq t + \gamma_n/\sqrt{\log p}\}}{p_0 G(t + \gamma_n/\sqrt{\log p})(1 + o(1))}
$$
and
$$
\frac{\sum_{j\in\mathcal H_0}1\{\bar t_j \geq t\}}{p_0 G(t)} \leq \frac{\sum_{j\in\mathcal H_0}1\{ \bar T_j \geq t - \gamma_n/\sqrt{\log p}\}}{p_0 G(t)} = \frac{\sum_{j\in\mathcal H_0}1\{ \bar T_j \geq t - \gamma_n/\sqrt{\log p}\}}{p_0 G(t - \gamma_n/\sqrt{\log p})(1 + o(1))}.
$$
Therefore, wp $\to 1$,
$$
\sup_{0\leq t \leq G^{-1}(\alpha|\mathcal H|/p)}\left|\frac{\sum_{j\in\mathcal H_0}1\{\bar t_j \geq t\}}{p_0G(t)} - 1\right| \leq \sup_{0\leq t \leq G^{-1}(\alpha|\mathcal H|/(2p))}2\times\left|\frac{\sum_{j\in\mathcal H_0}1\{\bar T_j \geq t\}}{p_0G(t)} - 1\right|.
$$
Combining this bound with \eqref{eq: thm 6 Liu Shao bound} gives \eqref{eq: thm 6 lln - 2} and so \eqref{eq: thm 6 lln}.

{\bf Step 6.} Here we complete the proof. By \eqref{eq: thm 6 step 3 - 2} in Step 3 and Step 5,
$$
\frac{\sum_{j\in\mathcal H_0}1\{t_j \geq \hat t\}}{p_0G(\hat t)} \leq 1 + o_p(1),
$$
and so
\begin{equation}\label{eq: thm 6 step 5}
\sum_{j\in\mathcal H_0} 1\{t_j \geq \hat t\} \leq p_0 G(\hat t)(1 + o_p(1)).
\end{equation}
Therefore,
$$
FDP = \frac{\sum_{j\in\mathcal H_0}1\{t_j \geq \hat t\}}{\max\{ \sum_{j=1}^p 1\{t_j \geq \hat t\} , 1\}} \leq \frac{p_0G(\hat t)(1 + o_p(1))}{p G(\hat t)/\alpha} = \frac{\alpha p_0}{p} + o_p(1),
$$
where the first equality follows from Step 2 and the second from \eqref{eq: thm 6 step 5} above and \eqref{eq: thm 6 step 3 - 1} in Step 3. Finally, since FDP is bounded between 0 and 1, it follows that
$$
FDR = \Ep[FDP] \leq \alpha p_0/p + o(1) \leq \alpha + o(1).
$$
The asserted claim follows.
\qed

\section{Proofs for Section \ref{sec: many moments}}

\noindent
{\bf Proof of Lemma \ref{lem: sparse singular values}.}
The proof follows immediately from Lemma \ref{thm: bchn theorem} in Appendix \ref{app: technical results}.
\qed

\noindent
{\bf Proof of Lemma \ref{lemma:EMC}.} By the triangle inequality,
$$
 \sup_{\theta \in \mathcal R(\theta_0)} \|\hat g(\theta)-g(\theta)\|_\infty \leq \sup_{\theta \in \mathcal R(\theta_0)} \|\hat g(\theta)-g(\theta)-\hat g(\theta_0)\|_\infty + \|\hat g(\theta_0)\|_\infty,
$$
where the second term on the right-hand side satisfies
\begin{equation}\label{B11}  \|\hat g(\theta_0)\|_\infty = \|n^{-1/2}\Gn(g(X,\theta_0))\|_\infty\leq n^{-1/2}\ell_n\end{equation}
with probability $1-\delta_n/6$ by Condition ENM. Thus, it remains to bound
$$
\sup_{\theta\in\mathcal R(\theta_0)}\|\hat g(\theta) - g(\theta) - \hat g(\theta_0)\|_{\infty}.
$$
To do so, we apply Lemma \ref{lem:MaxIneqContraction} in Appendix \ref{app: technical results}.
Denote
$$
h_j(X,t) = \tilde g_j(X,Z_{u(j)}(X)'\vartheta_{0 u(j)} + t) - \tilde g_j(X,Z_{u(j)}(X)'\vartheta_{0 u(j)}),\quad j\in[m],
$$
$$
\eta = \theta - \theta_0 = ((\vartheta_1 - \vartheta_{0 1})',\dots,(\vartheta_{\bar u} - \vartheta_{0 \bar u})')',
$$
and $\Delta = \{\theta - \theta_0\colon \theta\in\mathcal R(\theta_0)\}$. Then $\|\Delta\|_1 = \sup_{\eta\in\Delta} \| \eta \|_{1} \leq 2K$ by Condition DM and
$$
\sup_{\theta\in\mathcal R(\theta_0)}\|\hat g(\theta) - g(\theta) - \hat g(\theta_0)\|_{\infty} = n^{-1/2} \max_{j \in [m]} \sup_{\eta \in \Delta} |\Gn(h_j(X,Z_{u(j)}(X)'(\vartheta_{u(j)} - \vartheta_{0 u(j)}))) |.
$$
Also, by Condition ENM, the functions $h_j(X,t)$ have the contractive structure as required in Lemma \ref{lem:MaxIneqContraction} and
$$
\sup_{\eta\in\Delta, j\in[m]}{\rm Var}(h_j(X,Z_{u(j)}(X)'(\vartheta_{u(j)} - \vartheta_{0 u(j)}))) \leq B_{1 n}^2.
$$
Finally, with probability at least $1 - \delta_n/6$,
$$
\max_{j\in[m],k\in[p_z]} \En[L_j^2(X)Z_{\uu(j)k}^2(X)] \leq B_{2 n}^2,
$$
again by Condition ENM. Therefore, applying Lemma \ref{lem:MaxIneqContraction} shows that
$$
\Pr\left(\max_{j \in [m]} \sup_{\eta \in \Delta} |\Gn(h_j(X,Z_{u(j)}(X)'(\vartheta_{u(j)} - \vartheta_{0 u(j)}))) | > t\right) \leq 5\delta_n/6
$$
for all $t \geq \{4B_{1 n}\} \vee \{16\sqrt{2} B_{2 n} K \log^{1/2}(48p m/\delta_n)\} =\tilde\ell_n$. Combining this bound with \eqref{B11} gives the asserted claim.
\qed


\noindent
{\bf Proof of Lemma \ref{lem:HL-OmegaG}.}
We will verify conditions L and ELM to invoke Theorem \ref{thm:BoundLinearRGMM} since  Conditions LID$(\gamma_{0j},\Omega)$ and DM ($\|\gamma_{0j}\|_1\leq K$) are assumed. By definition of the estimator $\hat \gamma_j$ in (\ref{def:hatgamma}), it follows that $\hat\gamma_j$ is an RGMM estimator with a linear score function since $$g(\gamma_j) = \gamma_j \Omega + (G')_j \ \ \ \mbox{and} \ \ \ \hat g(\gamma_j) = \gamma_j \hat \Omega + (\hat G')_j.$$ In this case we have that ELM holds by considering $\hat G := \hat \Omega$, $\hat g(0) := (\hat G')_j$, and  $\ell_n=\ell_n^\Omega \vee \ell_n^G$ by assumption.

Next we verify that the choice of penalty satisfies Condition L for some $\alpha = \delta_n$. Note that with probability $1-\delta_n$
$$ \begin{array}{rl}
\|\gamma_{0j}\hat\Omega - (\hat G')_j \|_\infty & \leq \|\gamma_{0j}\hat\Omega - ( G')_j \|_\infty + n^{-1/2} \ell_n^G\\
& \leq \|\gamma_{0j}\Omega - ( G')_j \|_\infty + \|\gamma_{0j}\|_1\|\hat\Omega - \Omega\|_\infty + n^{-1/2} \ell_n^G\\
& \leq 0 + K n^{-1/2}\ell_n^\Omega + n^{-1/2}\ell_n^G
\end{array}
 $$ so that $\gamma_{0j}$ is feasible for all $j\in[p]$ if $\lambda_j^\gamma \geq K n^{-1/2} \ell_n^\Omega + n^{-1/2}\ell_n^G$. Thus the result follows from Theorem \ref{thm:BoundLinearRGMM}.

For the second result, we will verify conditions L, DM, LID$(\mu_{0j},G'\Omega^{-1}G)$ and ELM to invoke Theorem \ref{thm:BoundLinearRGMM}. Conditions DM and LID$(\mu_{0j},G'\Omega^{-1}G)$ are assumed. By the definition of the estimator $\hat \mu_j$, given in (\ref{def:hatmu}), we have that $\hat \mu_j$ is a RGMM estimator  associated with a  linear score as we can write $$g(\mu_j) = \mu_j \gamma_{0}G - e_j' \ \ \mbox{and} \ \ \hat g(\mu_j) = \mu_j \hat \gamma \hat G -e_j'.$$

To verify Condition ELM, note that $g(0)=\hat g(0)=-e_j$ and by assumption $\|\mu_{0j}\|_1 \leq K$ so we need to bound $\|\hat\gamma\hat G - \gamma_0 G\|_\infty$. Using that $\gamma_0 G = G'\gamma_0'$ since $\gamma_0 = G'\Omega^{-1}$, we have with probability $1-2\delta_n$ that
$$\begin{array}{rl}
 \|\hat \gamma \hat G - \gamma_0 G \|_\infty & \leq \|\hat \gamma (\hat G - G)\|_\infty + \|\hat \gamma G - G'\gamma_0'  \|_\infty \\
& \leq \max_{j\in[p]}\|\hat \gamma_j\|_1 \|\hat G - G\|_\infty  + \|\hat \gamma \Omega \Omega^{-1} G - G'\Omega^{-1} G \|_\infty \\
& \leq \max_{j\in[p]}\| \gamma_{0j}\|_1 \|\hat G - G\|_\infty  + \|\hat \gamma (\Omega - \hat\Omega) \gamma_0'\|_\infty +  \|\hat \gamma \hat\Omega \gamma_0' - G'\gamma_0' \|_\infty \\
& \leq K \|\hat G - G\|_\infty  + \|\hat \gamma (\Omega - \hat\Omega)\|_\infty \max_{j\in[p]}\|\gamma_{0j}\|_1  +  \|\hat \gamma \hat\Omega \gamma_0' - G'\gamma_0' \|_\infty \\
& \leq K \|\hat G - G\|_\infty  + K^2 \|\Omega - \hat\Omega\|_\infty   + \| (\hat G' - G')\gamma_0' \|_\infty + \|\hat \gamma \hat\Omega \gamma_0' - \hat G'\gamma_0' \|_\infty \\
& \leq K \|\hat G - G\|_\infty  + K^2 \|\Omega - \hat\Omega\|_\infty   + K \| \hat G' - G'\|_\infty + K\|\hat \gamma \hat\Omega  - \hat G' \|_\infty \\
& \leq K n^{-1/2}\ell_n^G  + K^2 n^{-1/2}\ell_n^\Omega + K n^{-1/2}\ell_n^G + K\max_{j\in[p]}\lambda_j^\gamma =:r^* \\
\end{array}$$
where we used that $\|\hat \gamma_j\|_1 \leq \|\gamma_{0j}\|_1$ for all $j\in[p]$ with probability $1-\delta_n$ and that \{$\|\hat G - G\|_\infty \leq n^{-1/2}\ell_n^G$ and $\|\hat \Omega - \Omega\|_\infty \leq n^{-1/2} \ell_n^\Omega$\} occurs with probability $1-\delta_n$. Thus Condition ELM holds with $\ell_n = n^{1/2}r^*$.

Next we verify Condition L. We have that with probability at least $1-2\delta_n$
$$\| \mu_0 \hat \gamma \hat G - I \|_\infty \leq  \| \mu_0  \gamma_0  G - I \|_\infty + \|\mu_0 (\hat\gamma \hat G - \gamma_0 G)\|_\infty \leq 0 + \max_{j\in[p]}\|\mu_{0j}\|_1\|\hat\gamma \hat G - \gamma_0 G\|_\infty\leq K r^*.$$
Therefore, if $\lambda_j^\mu \geq K r^*$ we have that $(\mu_{0j})_{j\in[p]}$ is feasible for the optimization problem (\ref{def:hatmu}) with probability at least $1-2\delta_n$ (i.e., Condition L holds with $\alpha=2\delta_n$).  Then the result follows by Theorem \ref{thm:BoundLinearRGMM}.\qed

\noindent
{\bf Proof of Lemma \ref{lem:PrimitiveGandOmegaLinear}.}
Let $M_n = \Ep[\max_{i\in[n]}\|G(X_i,\theta_0)\|_\infty^2] \vee \Ep[\max_{i\in[n]}\|g(X_i,\theta_0)\|_\infty^4]$.

To bound $\|\hat G - G\|_\infty$ we apply Lemma \ref{lem:m2bound}(4) with $t=\delta_n^{-1}$, $\bar q = 2$. Indeed, under the assumed condition $n^{-1/2}M_n\{\delta_n^{-1}+\log(2m)\}\leq c$, with probability $1-2\delta_n$ we have
$$ \max_{j\in[m],k\in[p]} |\Bbb{G}_n G_{jk}(X)| \leq C \sqrt{\sigma^2 \log(2m)} \ \ \ \mbox{and} \ \  \max_{k\in[m]} |\Bbb{G}_n g_{k}(X,0)| \leq C \sqrt{\sigma^2 \log(2m)}. $$

Next we establish the bound on $\|\hat\Omega - \Omega\|_\infty$. Using the triangle inequality we have
\begin{equation}\label{eq:0011} \begin{array}{rl}
n^{1/2}\|\hat \Omega - \Omega\|_\infty  &= \displaystyle \max_{j\in[m],k\in[m]} n^{1/2}| \En[g_j(X,\hat\theta)g_k(X,\hat\theta)] - \Ep[g_j(X,\theta_0)g_k(X,\theta_0)]| \\
& \leq  \displaystyle \max_{j\in[m],k\in[m]} n^{1/2}|   \En[ \{g_j(X,\hat\theta)-g_j(X,\theta_0)\}\{g_k(X,\hat\theta)-g_k(X,\theta_0)\}]| \\
& \qquad + \displaystyle 2\max_{j\in[m],k\in[m]} n^{1/2}| \En[ g_j(X,\theta_0)\{g_k(X,\hat\theta)-g_k(X,\theta_0)\}]|\\
 & \qquad + \displaystyle \max_{j\in[m],k\in[m]} | \Gn(g_j(X,\theta_0)g_k(X,\theta_0))|
 \end{array} \end{equation}

To bound the last term of the right-hand-side (RHS) in (\ref{eq:0011}) we apply Lemma \ref{lem:m2bound}(4) with $q=2$ and $t=\delta_n$, so that with probability $1-\delta_n$
$$\begin{array}{rl}
 \max_{k,j\in[m]}|\Gn(g_k(X,\theta_0)g_j(X,\theta_0))| & \leq C\max_{k\in[m]}\Ep[g_k^4(X,\theta_0)]^{1/2} \sqrt{\log(2m)} \\
 & \quad + Cn^{-1/2}\Ep[\max_{i\in[n]}\|g(X,\theta_0)\|_\infty^4]\{ \delta_n^{-1} + \log(m) \}\\
 & \leq C' \sqrt{\log(2m)} \end{array}
 $$
under the growth condition (iv).

To bound the first term of the RHS in (\ref{eq:0011}) we note that by linearity of the score, $g(X,\hat\theta)-g(X,\theta_0)= G(X)'(\hat\theta-\theta_0)$, and applying the Cauchy-Schwarz inequality  we have
$$ \begin{array}{rl} n^{1/2}\En[ \{g_j(X,\hat\theta)-g_j(X,\theta_0)\}\{g_k(X,\hat\theta)-g_k(X,\theta_0)\}]| &{\displaystyle \leq n^{1/2}\max_{k\in[m]}} \En[ \{G_k(X)(\hat\theta-\theta_0)\}^2] \\
& \leq n^{1/2}\Delta_{2n}^2\end{array}$$
where the last inequality holds with probability $1-\delta_n$ by condition (iii).

To bound the remaining term in the RHS in (\ref{eq:0011}), again using linearity, the Cauchy-Schwarz inequality, and condition (iii),  with probability $1-\delta_n$ $$ \begin{array}{rl}
 &|\En[ g_j(X,\theta_0)\{g_k(X,\hat\theta)-g_k(X,\theta_0)\}]| = |\En[g_j(X,\theta_0)G_k(X)](\hat\theta-\theta_0)| \\
 & \qquad \qquad \leq \max_{j\in[m]}\{\En[g_j^2(X,\theta_0)]\}^{1/2}\max_{k\in[m]} \En[ \{G_k(X)(\hat\theta-\theta_0)\}^2]^{1/2}\\
 & \qquad \qquad \leq \max_{j\in[m]}\{\En[g_j^2(X,\theta_0)]\}^{1/2} \Delta_{2n}\end{array}$$

To bound $\max_{j\in[m]}\{\En[g_j^2(X,\theta_0)]\}^{1/2}$, we will apply Lemma \ref{lem:m2bound}(3). We have
$$ \begin{array}{rl} \Ep[\max_{j\in[m]} |\En[g_j^2(X,\theta_0)] - \Ep[g_j^2(X,\theta_0)]|] &\leq n^{-1} \Ep[\max_{i\in[n]}\|g(X,\theta_0)\|_\infty^2] \log (2m) \\
& \quad + n^{-1/2}\Ep[\max_{i\in[n]}\|g(X,\theta_0)\|_\infty^2]^{1/2}\log^{1/2}(2m)\\
 & \leq C\delta_n \end{array} $$
under $n^{-1/2}\Ep[\max_{i\in[n]}\|g(X,\theta_0)\|_\infty^2]^{1/2}\leq \delta_n\log^{-1/2}(2m)$ assumed in condition (iv). Thus, by Markov's inequality we have that  with probability $1-\delta_n$
$$   \max_{j\in[m]} \En[g_j^2(X,\theta_0)] \leq \max_{j\in[m]}\Ep[g_j^2(X,\theta_0)] + C. $$
\qed

\noindent
{\bf Proof of Lemma \ref{lem:PrimitiveGandOmega}.}
Recall that $\partial_j g_k(X,\theta)=G(X,\theta)$.

We first bound $\|\hat G - G\|_\infty$.  We have that
\begin{equation}\label{eq:000} \begin{array}{rl}
\displaystyle  \max_{j\in[p],k\in[m]} n^{1/2}|\hat G_{kj} - G_{kj}|
& \displaystyle  =  \max_{j\in[p],k\in[m]} n^{1/2}|\En[\partial_{j}g_k(X,\hat\theta)]-\Ep[\partial_{j}g_k(X,\theta_0)]| \\
& \displaystyle \leq    \max_{j\in[p],k\in[m]} |\Gn(\partial_{j}g_k(X,\hat\theta)-\partial_{j}g_k(X,\theta_0))| \\
& \qquad + \max_{j\in[p],k\in[m]}|\Gn(\partial_{j}g_k(X,\theta_0))| \\
& \qquad + \max_{j\in[p],k\in[m]}n^{1/2}|\Ep[\partial_{j}g_k(X,\hat\theta)-\partial_{j}g_k(X,\theta_0)]|\\
\end{array}
\end{equation}

To control the second term in the RHS of (\ref{eq:000}) we apply Lemma \ref{lem:m2bound}(4) with $q=2$ and $t=\delta_n$, so that with probability $1-\delta_n$
$$\max_{j\in[p],k\in[m]}|\Gn(\partial_{j}g_k(X,\theta_0))| = \max_{j\in[p],k\in[m]}|\Gn(G_{kj}(X,\theta_0))| \leq C\sqrt{\log(2m)} $$
under condition (v).

To bound the last term in the RHS of (\ref{eq:000}),  we have
$$\begin{array}{rl}
\max_{j\in[p],k\in[m]} & n^{1/2}|\Ep[\partial_{j}g_k(X,\hat\theta)-\partial_{j}g_k(X,\theta_0)]|\\
& \leq_{(1)} \max_{j\in[p],k\in[m]}n^{1/2}\Ep[|\tilde L_{kj}(X)\tilde Z_{jk}(X)'(\hat\theta-\theta_0)|]\\
& \leq_{(2)} \max_{j\in[p],k\in[m]}n^{1/2}\Ep[|\tilde L_{kj}(X)\tilde Z_{jk}(X)'(\hat\theta-\theta_0)|^2]^{1/2}\\
& \leq_{(3)} C n^{1/2}\|\hat\theta-\theta_0\|_2 \\
& \leq_{(4)} Cn^{1/2} \Delta_{2n}
\end{array}$$
where (1) holds by the ENM condition $|G_{kj}(X,\tilde \theta)-G_{kj}(X,\theta)|\leq L_{kj}(X)|\tilde Z_{kj}'(\tilde \theta-\theta)|$, (2) holds by monotonicity of $L_q$ norms,  (3) holds by condition (i), and (4) holds with probability $1-\delta_n$  under condition (iii).

Finally, to control the first term in the RHS of (\ref{eq:000}),  we have by Lemma \ref{lem:MaxIneqContraction} in Appendix \ref{app: technical results} with $B_{1n} = B_n\Delta_{1n} $ and $B_{2n} = B_n$ so that
$$ \max_{j\in[p],k\in[m]} |\Gn(\partial_{j}g_k(X,\hat\theta)-\partial_{j}g_k(X,\theta_0))| \leq CB_n\Delta_{1n} \log^{1/2}(m/\delta_n)$$
with probability $1-7\delta_n$ where we used that $\|\hat\theta-\theta_0\|_1 \leq \Delta_{1n}$ with probability at least $1-\delta_n$ by condition (iii).

To bound $\|\hat G - \tilde G\|_\infty$ we note that
$$\|\hat G - \tilde G\|_\infty \leq \|\hat G -  G\|_\infty + \|\tilde G -  G\|_\infty $$
and that Lemma \ref{lem:MaxIneqContraction} in Appendix \ref{app: technical results} allows for $\bar m = m$ and different $\tilde \theta_j^*$, $j\in[m]$. Since for every $j\in[m]$ we have $\|\tilde \theta_j^* - \theta_0 \|_1 \leq \|\hat \theta - \theta_0 \|_1$ and $\|\tilde \theta_j^* - \theta_0 \|_2 \leq \|\hat \theta - \theta_0 \|_2$, the bound we derived on $\|\hat G - G\|_\infty$ applies to  $\|\tilde G -  G\|_\infty$ as well.

Next we establish the bound on $\|\hat\Omega - \Omega\|_\infty$. By the triangle inequality, we have
\begin{equation}\label{eq:001} \begin{array}{rl}
n^{1/2}\|\hat \Omega - \Omega\|_\infty & = \displaystyle \max_{j\in[m],k\in[m]} n^{1/2}| \En[g_j(X,\hat\theta)g_k(X,\hat\theta)] - \Ep[g_j(X,\theta_0)g_k(X,\theta_0)]| \\
& \leq \displaystyle \max_{j\in[m],k\in[m]}| \Gn(g_j(X,\hat\theta)g_k(X,\hat\theta)
 - g_j(X,\theta_0)g_k(X,\theta_0))| \\
&\qquad +\displaystyle \max_{j\in[m],k\in[m]}| \Gn(g_j(X,\theta_0)g_k(X,\theta_0))| \\
&\qquad + \displaystyle \max_{j\in[m],k\in[m]} n^{1/2}|\Ep[g_j(X,\hat\theta)g_k(X,\hat\theta)
 - g_j(X,\theta_0)g_k(X,\theta_0)]| \\
 \end{array} \end{equation}

To control the first term of the RHS in (\ref{eq:001}), further apply the triangle inequality to obtain
$$
\begin{array}{rl}
&\max_{j\in[m],k\in[m]} | \Gn(g_j(X,\hat\theta)g_k(X,\hat\theta)
 - g_j(X,\theta_0)g_k(X,\theta_0))| \\
 & = \max_{j\in[m],k\in[m]}| \Gn( \{g_j(X,\hat\theta)-g_j(X,\theta_0)\}g_k(X,\hat\theta)
 - g_j(X,\theta_0)\{g_k(X,\theta_0)-g_k(X,\hat\theta)\})|\\
& \leq \max_{j\in[m],k\in[m]}| \Gn( \{g_j(X,\hat\theta)-g_j(X,\theta_0)\}\{g_k(X,\hat\theta)-g_k(X,\theta_0)\})|\\
& \qquad + 2\max_{j\in[m],k\in[m]}| \Gn( \{g_j(X,\hat\theta)-g_j(X,\theta_0)\}g_k(X,\theta_0))|\\
& = (I) + 2(II)\end{array}
$$
Therefore,  we obtain with probability $1-C\delta_n$
$$\begin{array}{rl}
(I) & \leq_{(1)} \max_{j\in[m],k\in[m]}  n^{1/2}\En[\{g_j(X,\hat\theta)-g_j(X,\theta_0)\}^2]^{1/2}\En[\{g_k(X,\hat\theta)-g_k(X,\theta_0)\}^2]^{1/2} \\
& \qquad  + \max_{j\in[m],k\in[m]} n^{1/2}\Ep[\{g_j(X,\hat\theta)-g_j(X,\theta_0)\}^2]^{1/2}\Ep[\{g_k(X,\hat\theta)-g_k(X,\theta_0)\}^2]^{1/2} \\
&\leq_{(2)} \max_{j\in[m],k\in[m]}  n^{1/2}\En[\{L_j(X)Z_j'(\hat\theta-\theta_0)\}^2]^{1/2}\En[\{L_k(X)Z_k'(\hat\theta-\theta_0)\}^2]^{1/2} \\
& \qquad  + \max_{j\in[m],k\in[m]} n^{1/2}\Ep[\{L_j(X)Z_j'(\hat\theta-\theta_0)\}^2]^{1/2}\Ep[\{L_k(X)Z_k'(\hat\theta-\theta_0)\}^2]^{1/2} \\
&\leq_{(3)} \max_{j\in[m]}  n^{1/2}\{\|\En[L_j^2(X)Z_jZ_j']\|_\infty+\|\Ep[L_j^2(X)Z_jZ_j']\|_\infty\}\|\hat\theta-\theta_0\|_1^2 \\
& \leq_{(4)} 2n^{1/2}B_n^2\Delta_{1n}^2\\
\end{array}
$$ where (1) follows from triangle inequality and Cauchy-Schwarz inequality, (2) follows from the ENM condition $|g_j(X,\tilde\theta)-g_j(X,\theta)|\leq L_j(X)|Z_j'(\tilde\theta-\theta)|$, (3) follows from $|v'Mv|\leq \|v\|_1\|M\|_\infty$, and (4) follows from $\|\hat\theta-\theta_0\|_1 \leq \Delta_{1n}$ with probability $1-\delta_n$ by condition (iii), $\|\Ep[L_j^2(X)Z_jZ_j']\|_\infty\leq \max_{k\in [p]}\Ep[L_j^2(X)Z_{jl}^2]\leq B_n^2$ by condition (i) and $\|\En[L_j^2(X)Z_jZ_j']\|_\infty\leq \max_{k\in [p]}\En[L_j^2(X)Z_{jl}^2]\leq B_n^2$ by condition (ii) with probability $1-\delta_n$.\footnote{An alternative bound can be obtained  by $\{{\rm maxeig}(\En[ L_j^2(X)ZZ']) \vee {\rm maxeig}(\Ep[ L_j^2(X)ZZ']) \} n^{1/2}\|\hat \theta - \theta_0\|_2^2$.}

Next we will apply the contraction principle in Lemma \ref{lem:MaxIneqContraction} in Appendix \ref{app: technical results} to control (II) with $\tilde L_{jk}(X) : = |L_j(X)g_k(X,\theta_0)|$. By condition (ii) and because for any $\theta$ such that $\|\theta-\theta_0\|_\ell\leq \Delta_{\ell n}$, $\ell \in\{\ell_1,\ell_2\}$,
$$\begin{array}{rl}
{\rm Var}(\Gn( \{g_j(X, \theta)-g_j(X,\theta_0)\}g_k(X,\theta_0))) & \leq {\rm Var}(\Gn( \{L_j(X)Z_j'( \theta-\theta_0)\}g_k(X,\theta_0)))\\
& \leq C\| \theta-\theta_0\|^2 \wedge B_n^2\|\theta-\theta_0\|_1^2 \\
& \leq C\Delta_{2n}^2 \wedge B_n^2 \Delta_{1n}^2
\end{array}$$
by condition (i), we have that the condition of Lemma \ref{lem:MaxIneqContraction} hold.
Therefore, by Lemma \ref{lem:MaxIneqContraction} we have with probability $1-7\delta_n$
$$ (II) \leq C' \{ B_n\Delta_{1n} \vee  B_n\Delta_{1n} \log^{1/2}(m  p/\delta_n)\}$$
where we used that $\|\hat\theta-\theta_0\|_1\leq \Delta_{1n}$ with probability $1-\delta_n$ by condition (iii).

To control the last term of the RHS in (\ref{eq:001}), we use that
$$ \begin{array}{rl}
&|\Ep[g_j(X,\hat\theta)g_k(X,\hat\theta)- g_j(X,\theta_0)g_k(X,\theta_0)]|  \\
&  \qquad \qquad \leq |\Ep[g_j(X,\hat\theta)\{g_k(X,\hat\theta)
 - g_k(X,\theta_0)\}]| + |\Ep[\{g_j(X,\hat\theta)-
  g_j(X,\theta_0)\}g_k(X,\theta_0)]|\\
& \qquad \qquad \leq B_n^2\|\hat \theta - \theta_0\|_1^2 + 2\max_{j,k\in[m]}|\Ep[\{g_j(X,\hat\theta)-
  g_j(X,\theta_0)\}g_k(X,\theta_0)]|\\
& \qquad \qquad \leq B_n^2\Delta_{1n}^2 + 2 C\Delta_{2n}.
\end{array}$$

To bound the second term of the RHS in (\ref{eq:001}) we use Lemma \ref{lem:m2bound}(4) so that with probability $1-\delta_n$
$$\begin{array}{rl}
 \max_{k,j\in[m]}|\Gn(g_k(X,\theta_0)g_j(X,\theta_0))| & \leq C\max_{k\in[m]}\Ep[g_k^4(X,\theta_0)]^{1/2} \sqrt{\log(2m)} \\
 & \quad + Cn^{-1/2}\Ep[\max_{i\in[n]}\|g(X_i,\theta_0)\|_\infty^4]\{ \delta_n^{-1} + \log(m) \}\\
 & \leq C' \sqrt{\log(2m)} \end{array}
 $$
under the growth condition (iv).
\qed

\noindent
{\bf Proof of Theorem \ref{thm:MainInference:Linear}.} We will establish that $\bar r_1  + \bar r_2 + \bar r_3 = o_P(\log^{-1/2}p)$. Throughout the proof we assume $s\geq 1$.

Because we are considering a special case, a few simplifications occur. Because this is a linear case, we can take $\tilde G = \hat G$ so that $\bar r_2 = 0$. Moreover, due to the homoskedastic setting, we note that $\mu_0\gamma_0$ is independent of $\sigma^2$ so we do not need to estimate $\sigma^2$ in the construction of $\hat\mu$ and $\hat\gamma$.

By Lemma \ref{lem:m2bound}(4) with probability $1-\delta_n$ we have
$$\begin{array}{rl}
\|\sqrt{n}\hat g(\theta_0)\|_\infty & = \max_{k\in[m]}|\Gn(g_k(X,\theta_0))| \\
 & \leq C\max_{k\in[m]}\Ep[g_k^2(X,\theta_0)]^{1/2} \sqrt{\log(2m)} \\
 & \qquad + Cn^{-1/2}\Ep[\max_{i\in[n]}\|g(X_i,\theta_0)\|_\infty^2]\{ \delta_n^{-1} + \log(2m) \}\\
 & \leq C' \sqrt{\log(2m)} \end{array}
 $$
under the growth condition $n^{-1/2}\Ep[\max_{i\in[n]}\|g(X,\theta_0)\|_\infty^2]\{ \delta_n^{-1} + \log(2m) \} \leq c\log^{1/2}(2m)$. Thus Condition L holds for Step 1 under the proposed choice of penalty parameter. Moreover, since the score is linear, we have that Condition ELM holds since $\hat G = -\En ZW'$ and $\hat g(0) = \En YW$; and we have by Lemma \ref{lem:PrimitiveGandOmegaLinear} 
that with probability $1-C\delta_n$
$$
\|\hat G - G\|_\infty   \leq C \sqrt{\log (2m)}, \ \ \mbox{and} \ \ \
\| \hat g(0) - g(0)\|_\infty   \leq C \sqrt{\log (2m)} $$
under Conditions (3) and (4).

Since Conditions DM ($\|\theta_0\|_1\leq K$) and LID$(\theta_0,G)$ hold by assumption, and by Theorem \ref{thm:BoundLinearRGMM} with $\ell_n = C' \sqrt{\log(2m)}$, we have for $q\in\{1,2\}$ that \begin{equation}\label{eq:RateHatTheta}\|\hat\theta-\theta_0\|_q \leq \Delta_{q n} := Cn^{-1/2} s^{1/q} \ell_n \{(1+L)\mu_n^{-1}+C\} \leq C'n^{-1/2}s^{1/q}\log^{1/2}(2m)\end{equation}
 where we used that $\mu_n \geq c$ and $L\leq C$.

Moreover by Lemma \ref{lem:PrimitiveGandOmegaLinear} we have that with probability $1-C\delta_n$
$$
\|\hat G - G\|_\infty   \leq C \sqrt{\log (2m)}, \ \ \mbox{and} \ \ \
\| \En ZZ' - \Ep ZZ'\|_\infty   \leq C \sqrt{\log (2m)} $$
under where we used that we do not need to estimate $\sigma$.

In order to apply  Lemma \ref{lem:HL-OmegaG}, recall the choices of penalty \begin{equation}\label{eq:lambdaFinal}\bar \lambda = \lambda_j^\gamma = \frac{1}{2}\lambda_j^\mu= n^{-1/2 }C(1\vee K^2)(1+K)\Phi^{-1}(1-(mpn)^{-1}),\end{equation}
so that we consider $\ell_n = n^{1/2}\bar \lambda \geq C(1\vee K^2)(1+K)\sqrt{\log(2m)}$ for some constant $C>0$ chosen sufficiently large so they satisfy the requirements on the penalty choices of Lemma \ref{lem:HL-OmegaG}. Therefore, Lemma \ref{lem:HL-OmegaG} yields with probability $1-C\delta_n$ that
$$ \max_{j\in[p]}\|\hat \gamma_j - \sigma^2\gamma_{0j}\|_1 \leq C\bar\lambda s\{1+\mu_n^{-1}\} \ \ \mbox{and} \ \ \|\hat\gamma_j\|_1 \leq \|\sigma^2\gamma_{0j}\|_1 \leq CK \ \ \mbox{for all} \ \ j\in[p]$$
$$ \max_{j\in[p]}\|\hat \mu_j - \sigma^{-2}\mu_{0j}\|_1 \leq C'\bar\lambda s\{1+\mu_n^{-1}\} \ \ \mbox{and} \ \ \|\hat\mu_j\|_1 \leq \|\sigma^{-2}\mu_{0j}\|_1 \leq CK \ \ \mbox{for all} \ \ j\in[p]$$

Now we are in position to bound $\bar r_1, \bar r_2$ and $\bar r_3$. By Lemma \ref{lemma:linearize} and using the definition of the estimators (\ref{def:hatmu}) we have
$$
\begin{array}{rl}
\bar r_1 & =  \sqrt{n} \| I - \hat \mu \hat \gamma \hat G\|_\infty \| \hat \theta- \theta_0\|_1 \leq \sqrt{n} \bar \lambda \Delta_{1n} \\
&  \leq C(1\vee K^2)(1+K)\Phi^{-1}(1-(pmn)^{-1})\Delta_{1n}\\
\bar r_2 & = 0 \\
\bar r_3 & \leq K C \bar \lambda s\{1+\mu_n^{-1}\} \|\sqrt{n}\hat g(\theta_0)\|_\infty\\
&  \leq K C(1\vee K^2)(1+K) n^{-1/2}\Phi^{-1}(1-(pmn)^{-1}) s \log^{1/2}(2m)\\
\end{array}
$$
Under $K \leq C$, $\mu_n \geq c$, and $n^{-1/2}s\log (2pmn) \leq u_n$, with probability $1-C\delta_n$ we have
$$ \bar r_1 + \bar r_2 + \bar r_3 \leq Cu_n. $$
\qed

\noindent
{\bf Proof of Theorem \ref{thm:MainInference:NonLinear}.} We will establish that $\bar r_1  + \bar r_2 + \bar r_3 = o_P(\log^{-1/2}p)$. Throughout the proof we assume $s\geq 1$.
%
%

By Lemma \ref{lem:m2bound}(4) with probability $1-\delta_n$ we have
$$\begin{array}{rl}
\|\sqrt{n}\hat g(\theta_0)\|_\infty & = \max_{k\in[m]}|\Gn(g_k(X,\theta_0))| \\
 & \leq C\max_{k\in[m]}\Ep[g_k^2(X,\theta_0)]^{1/2} \sqrt{\log(2m)} \\
 & \qquad + Cn^{-1/2}\Ep[\max_{i\in[n]}\|g(X,\theta_0)\|_\infty^2]\{ \delta_n^{-1} + \log(2m) \}\\
 & \leq C' \sqrt{\log(2m)} \end{array}
 $$
under the growth condition $n^{-1/2}\Ep[\max_{i\in[n]}\|g(X,\theta_0)\|_\infty^2]\{ \delta_n^{-1} + \log(2m) \} \leq c\log^{1/2}(2m)$.

By Theorem \ref{thm:NonLinear} with $\ell_n = C' \sqrt{\log(2m)}$, we have for $q\in\{1,2\}$ that \begin{equation}\label{eq:RateHatTheta}\|\hat\theta-\theta_0\|_q \leq \Delta_{q n} := Cn^{-1/2} s^{1/q} \ell_n \{(1+L)\mu_n^{-1}+C\} \leq C'n^{-1/2}s^{1/q}\log^{1/2}(2m)\end{equation}
 where we used that $\mu_n \geq c$ and $L\leq C$.

Moreover by Lemma \ref{lem:PrimitiveGandOmega} we have that
$$
\|\hat G - G\|_\infty   \leq C\Delta_{2n}, \ \ \ \|\hat G - \tilde G\|_\infty   \leq C\Delta_{2n} \ \ \ \mbox{and} \ \ \
\|\hat\Omega - \Omega\|_\infty   \leq C\Delta_{2n}
$$
under $B_n\Delta_{1n}\log^{1/2}(m/\delta_n) \leq Cn^{1/2}\Delta_{2n}$ (implied by $B_ns^{1/2}\log(mn)\leq Cn^{1/2}\log(2m)$), and $s\geq 1$.

In order to apply  Lemma \ref{lem:HL-OmegaG}, recall the choices of penalty \begin{equation}\label{eq:lambdaFinal}\bar \lambda = \lambda_j^\gamma = \frac{1}{2}\lambda_j^\mu= n^{-1/2+\bar a}\Phi^{-1}(1-(mpn)^{-1})\end{equation}
so that we consider $\ell_n = n^{1/2}\bar \lambda \geq C(1\vee K^2)(1+K)n^{1/2}\Delta_{2n}$ for some constant $C>0$ chosen sufficiently large so they satisfy the requirements on the penalty choices of Lemma \ref{lem:HL-OmegaG}. Therefore, Lemma \ref{lem:HL-OmegaG} yields with probability $1-C\delta_n$ that
$$ \max_{j\in[p]}\|\hat \gamma_j - \gamma_{0j}\|_1 \leq C\bar\lambda s\{1+\mu_n^{-1}\} \ \ \mbox{and} \ \ \|\hat\gamma_j\|_1 \leq \|\gamma_{0j}\|_1 \leq K \ \ \mbox{for all} \ \ j\in[p]$$
$$ \max_{j\in[p]}\|\hat \mu_j - \mu_{0j}\|_1 \leq C'\bar\lambda s\{1+\mu_n^{-1}\} \ \ \mbox{and} \ \ \|\hat\mu_j\|_1 \leq \|\mu_{0j}\|_1 \leq K \ \ \mbox{for all} \ \ j\in[p]$$

Now we are in position to bound $\bar r_1, \bar r_2$ and $\bar r_3$. By Lemma \ref{lemma:linearize} and using the definition of the estimators (\ref{def:hatmu}) we have
$$
\begin{array}{rl}
\bar r_1 & =  \sqrt{n} \| I - \hat \mu \hat \gamma \hat G\|_\infty \| \hat \theta- \theta_0\|_1 \leq \sqrt{n} \bar \lambda \Delta_{1n} \\
&  \leq n^{\bar a}\Phi^{-1}(1-(pmn)^{-1})\Delta_{1n}\\
\bar r_2 & = \sqrt{n} \max_{j\in[p]}\|\hat\mu_j\|_1 \max_{j\in[p]}\|\hat\gamma_j\|_1 \|\hat G - \tilde G\|_\infty \|\hat\theta-\theta_0\|_1 \\
& \leq CK^2 n^{1/2}\Delta_{2n}\Delta_{1n}\\
\bar r_3 & \leq K C \bar \lambda s\{1+\mu_n^{-1}\} \|\sqrt{n}\hat g(\theta_0)\|_\infty\\
&  \leq K C \bar n^{-1/2+\bar a}\Phi^{-1}(1-(pmn)^{-1}) s \log^{1/2}(2m)\\
\end{array}
$$
Under $K + \mu_n^{-1}\leq C$, and $n^{-1/2+\bar a}s\log (2pmn) \leq u_n$, with probability $1-C\delta_n$ we have
$$ \bar r_1 + \bar r_2 + \bar r_3 \leq Cu_n. $$
\qed

\begin{lemma}\label{lem:m2bound}
Let $X_i, i=1,\ldots,n,$ be independent random vectors in $\mathbb{R}^p$, $p\geq 3$. Define $\bar m_k^k := \max_{j\in [p]}\frac{1}{n}\sum_{i=1}^n\mathbb{E}[|X_{ij}|^k]$ and $M_{k}^k \geq \mathbb{E}[ {\displaystyle \max_{i\leq n}}\|X_i\|_\infty^k]$. Then we have the following bounds:
\begin{eqnarray*}
&& (1) \ \mathbb{E}\left[\max_{j\in [p]}\frac{1}{n}\sum_{i=1}^n|X_{ij}|\right] \leq CM_{1} n^{-1}\log p+ C\bar m_1 \\
&& (2) \ \mathbb{E}\left[\max_{j\in [p]}\frac{1}{n}\left|\sum_{i=1}^n X_{ij}-\mathbb{E}[X_{ij}]\right|\right] \leq C  \bar m_2\sqrt{n^{-1}\log p} + CM_2 n^{-1}\log p\\
&& (3) \ \mathbb{E}\left[\max_{j\in [p]}\frac{1}{n}\left|\sum_{i=1}^n|X_{ij}|^k-\mathbb{E}[|X_{ij}|^k]\right|\right] \leq \frac{C M_{k}^k\log p}{n}+C\sqrt{\frac{M_{k}^k\bar m_k^k \log p}{n}}\\
\end{eqnarray*}
for some universal constant $C$. Moreover, for $\bar q \geq 2$, we have with probability $1-t^{-\bar q/2}$
$$ (4) \ \max_{j\in[p]} |\Gn(X_j)| \leq C \bar m_2\sqrt{\log p} +  n^{-1/2}C_{\bar q}\{ M_{\bar q} t^{1/2} + M_2 (t+\log p)\}.
$$
\end{lemma}
\begin{proof}
The proofs of the first inequalities are given in Lemmas 8 and 9 of \cite{CCK15}. The proof of the last inequality can be found in \cite{BRT:coniceiv}.

The last result follows from the second inequality and Theorem 5.1 in\cite{CCK14} with $\alpha = 1$. Indeed we have with probability $1-t^{\bar q/2}$
$$ \max_{j\in[p]} |\Gn(X_j)| \leq 2\left\{\bar m_2\sqrt{\log p} + n^{-1/2}M_2 \log p\right\}
+ K(\bar q) \Big [ ( \bar m_2 + n^{-1/2} M_{\bar q}) \sqrt{t}
+  n^{-1/2} M_{2} t \Big ]$$
The result follows by collecting the terms.
\end{proof}

\section{Technical Lemmas}\label{app: technical results}
\begin{lemma}\label{lem: summation-integration}
For any $q>1/a$,
$$
A^q\sum_{j=1}^p j^{-a q}1\{j \geq (A/\lambda)^{1/a}\} \leq \frac{2^{a q}s \lambda^q}{a q - 1},
$$
where $s := \lceil (A/\lambda)^{1/a}\rceil$.
\end{lemma}
\noindent
{\bf Proof of Lemma \ref{lem: summation-integration}.}
We consider two cases separately: $A \leq \lambda$ and $A > \lambda$. When $A \leq \lambda$, we have
\begin{align*}
\sum_{j=1}^p j^{-a q}1\{j \geq (A/\lambda)^{1/a}\}
&\leq \sum_{j=1}^p j^{-a q} = 1 + \sum_{j=2}^p j^{-a q} \leq 1 + \int_1^{\infty}x^{-a q}d x = 1 + \frac{1}{a q - 1},
\end{align*}
so that
$$
A^q \sum_{j=1}^p j^{-a q}1\{j \geq (A/\lambda)^{1/a}\} \leq \left(1 + \frac{1}{a q - 1}\right) \lambda^q = \left(1 + \frac{1}{a q - 1}\right) s \lambda^q \leq \frac{2^{a q}s \lambda^q}{a q - 1}.
$$
When $A > \lambda$, we have $s = \lceil(A/\lambda)^{1/a}\rceil \geq 2$, and
$$
\sum_{j=1}^p j^{-a q}1\{j \geq (A/\lambda)^{1/a}\} = \sum_{j = s}^p j^{-a q} \leq \int_{s - 1}^{\infty} x^{- a q}d x = \frac{(s - 1)^{1 - a q}}{a q - 1},
$$
so that
$$
A^q \sum_{j=1}^p j^{-a q}1\{j \geq (A/\lambda)^{1/a}\} \leq \frac{s}{a q - 1}\left(\frac{A}{(s - 1)^a}\right)^q \leq \frac{2^{a q}s\lambda^q}{a q - 1}
$$
since $2 (s - 1) \geq (A/\lambda)^{1/a}$. Conclude that the asserted claim holds in both cases.
\qed

\begin{lemma}[General Lower Bound for $k(\theta_0,\ell)$ in Exactly Sparse Models]\label{thm: bchn theorem}
Recall the definition of $l$-sparse smallest and $l$-sparse largest singular values $\sigma_{\min}(l)$ and $\sigma_{\max}(l)$ in \eqref{eq: sparse singular values def} and assume that Condition ES is satisfied. Then
\begin{equation}\label{eq: alexbound2}
k (\theta_0, \ell_q)  \geq  \max_{l \geq s } \frac{\sigma_{\min}(l)}{\sqrt{l}}
\left(1 - \frac{\sigma_{\max}(l)}{\sigma_{\min}(l)}\sqrt{\frac{4 s}{l}}\right)\frac{s^{1/2-1/q}}{2+4\sqrt{s/l}},\quad q\in\{1,2\}.
\end{equation}
\end{lemma}
\noindent
{\bf Proof of Lemma \ref{thm: bchn theorem}}
By Condition ES, there exists $T\subset\{1,\dots,p\}$ such that $|T| = s$ and $\theta_{0 j} = 0$ for all $j\in T^c$. Also, for any $\theta\in\mathcal R(\theta_0)$, we have $\|\theta_T\|_1 + \|\theta_{T^c}\|_1 = \|\theta\|_1 \leq \|\theta_0\|_1$, and so $\|\theta_{T^c}\|_1 \leq \|\theta_0\|_1 - \|\theta_T\|_1 \leq \|(\theta - \theta_0)_T\|_1$ by the triangle inequality. Hence, $k(\theta_0,\ell_q) \geq \kappa_q^G(s,1)$ for $\kappa_q^G(s,1)$ defined in front of Theorem 1 in \cite{BCHN17}. Thus, given that by Theorem 1 in \cite{BCHN17}, $\kappa_q^G(s,1)$ is bounded from above by the right-hand side of \eqref{eq: alexbound2}, the asserted claim follows.
\qed

\begin{lemma}[Maximal Inequality Based on Contraction Principle]\label{lem:MaxIneqContraction}
Let $(X_i)_{i=1}^n$ be independent random vectors with common support $\mathcal X$, and let $\mathbb G_n$ be the corresponding empirical process. Consider the maximum over suprema of  empirical processes with contractive structure:
$$
\max_{j \in [m]} \sup_{\eta \in \Delta} |\Gn(h_j(X,Z_{u(j)}(X)'v_{u(j)})) |,
$$
where (1) $u(j)\in[\bar u]$ for all $j\in[m]$; (2) $v_u$ is a parameter vector in $\mathbb R^{p_u}$ for all $u \in [\bar u]$; (3) $\Delta$ is a set in $\mathbb R^p$ with $p = p_1 + \dots + p_{\bar u}$; (4) $\eta = (v_1',\dots,v_{\bar u}')'$; (5) for all $j\in[m]$, the link function $h_j$
is a measurable map from $\mathcal{X} \times \Bbb{R}$ to $\Bbb{R}$, is pointwise Lipschitz: $|h_j(x,t)-h_j(x,s)|\leq L_j(x)|t-s|$ for all $x\in\mathcal X$ and $t,s \in \mathbb R$, and is passing through the origin $h_j(x,0)=0$, for all $x \in \mathcal{X}$; (6) for all $j\in[m]$, the function $L_j$ is a measurable map from $\mathcal X$ to $\mathbb R$; and (7) for all $u\in[\bar u]$, the function $Z_u$ is a measurable map from $\mathcal X$ to $\mathbb R^{p_u}$.

Suppose that $$
\sup_{\eta\in\Delta, j\in[m]}\En{\rm Var}(h_j(X,Z_{u(j)}(X)'v_{u(j)})) \leq B_{1 n}^2,   \quad \max_{j\in[m],k\in[p_{u(j)}]} \En[L_{j}^2(X)Z_{u(j)k}^2(X)]  \leq B^2_{2 n}
$$
with probability at least $1-\delta_n$. Then for all $t\geq 4 B_{1 n}$,
$$
\Pr\left(\sup_{\eta \in \Delta,j\in[m]} |\Gn(h_j(X,Z_{u(j)}(X)'v_{u(j)}))| > t \right) \leq 4\delta_n + 8p m \exp\left(-\frac{t^2}{128B_{2 n}^2 \|\Delta\|_{1,v}^2}\right),
$$
where $\|\Delta\|_{1,v} := \max_{u\in[\bar u]}\sup_{\eta\in \Delta}\|v_u\|_1$. Thus,
$$
 \Pr\left(\sup_{\eta \in \Delta,j\in[m]} |\Gn(h_j(X,Z_{u(j)}(X)'v_{u(j)}))| > t \right) \leq 5\delta_n
 $$
as long as $t\geq \{4B_{1 n}\} \vee \{8\sqrt{2} B_{2 n} \|\Delta\|_{1,v} \log^{1/2}(8p m/\delta_n)\}$.
\end{lemma}

\noindent
{\bf Proof of Lemma \ref{lem:MaxIneqContraction}.}
The proof is a variant of the argument given by \cite{BC11}. Since the second asserted claim follows immediately from the first one, it suffices to prove the first one. To do so, we split the proof into two steps.

{\bf Step 1}. We first invoke a symmetrization lemma for probabilities. Let $(\sigma_i)_{i=1}^n$ be i.i.d. copies of  the Rademacher random variable $\sigma$, which takes on values $\{-1,1\}$ with equal probabilities. Fix any $t$ such that
$$
t^2 \geq 16 B^2_{1 n} \geq 16\sup_{\eta\in\Delta, j\in[m]} \En  {\rm Var}(h_j(X,Z_{u(j)}(X)'v_{u(j)})).
$$
Then by Lemma 2.3.7 in \cite{VW96} and Chebyshev's inequality,
\begin{align*}
(\star):  & =  \Pr \left ( \sup_{\eta \in \Delta,j\in[m]} |\Gn(h_j(X,Z_{u(j)}(X)'v_{u(j)}))| > t  \right ) \\
& \leq    4 \Pr \left (\sup_{\eta \in \Delta, j\in[m]} \left|\Gn(\sigma h_j(X,Z_{u(j)}(X_i)'v_{u(j)}))\right|> t/4 \right ). \end{align*}
Next, let
$$
\mathcal{A}:= \sup_{\eta \in \Delta, j\in[m]} \left|\Gn(\sigma h_j(X,Z_{u(j)}(X)'v_{u(j)}))\right|
$$
and define the event $\Omega := \{ \max_{j\in[m],k\in[p_{u(j)}]} \En[L_j^2(X)Z^2_{u(j)k}(X)] \leq B_{2 n}^2\}$. Then
$$
\Pr( \mathcal{A} > t /4) \leq \Pr(\mathcal{A} > t/4 \mid \Omega) + \Pr(\Omega^c) \leq \Ep[ \Pr ( \mathcal{A} > t /4 \mid (X_i)_{i=1}^n, \Omega)] + \delta_n,
$$
where we used that $P(\Omega^c)\leq \delta_n$ by assumption.

Now, by Markov's inequality, for $\psi := t/ (16B_{2 n}^2\|\Delta\|_{1,v}^2)$,
\begin{align}
\Pr ( \mathcal{A} > t/4 \mid (X_i)_{i=1}^n, \Omega) & \leq   \exp(-\psi t/4)\Ep[\exp(\psi \mathcal{A})\mid (X_i)_{i=1}^n, \Omega]\nonumber\\
&\leq   2p m \exp(-\psi t/4) \exp\left( 2\psi^2 B_{2 n}^2 \|\Delta\|_{1,v}^2\right)\label{eq: key inequality contraction}\\
&= 2p m \exp(- t^2/\{128 B_{2 n}^2\|\Delta\|_{1,v}^2\}),\nonumber
\end{align}
where \eqref{eq: key inequality contraction} is established in Step 2 below. Therefore,
$$
(\star) \leq 4\delta_n + 8p m \exp\left(-\frac{t^2}{128B_{2 n}^2 \|\Delta\|_{1,v}^2}\right),
$$
which is the first asserted claim. It remains to establish \eqref{eq: key inequality contraction}.

{\bf Step 2}. Here, we bound $\Ep[\exp(\psi\mathcal{A})\mid (X_i)_{i=1}^n,\Omega]$ and establish \eqref{eq: key inequality contraction}. In this step, we will condition throughout on $\{(X_i)_{i=1}^n,\Omega\}$ but omit explicit notation for this conditioning to keep the notation lighter.  

We have
\begin{align*}
 &  \Ep\left[ \exp\left( \psi\sup_{\eta\in \Delta,j\in[m]} \left| \Gn(\sigma h_{j}(X, Z_{u(j)}(X)'v_{u(j)})) \right|  \right) \right] \\
 &\qquad = \Ep\left[ \exp\left( \frac{\psi}{\sqrt n}\sup_{\eta\in \Delta,j\in[m]} \left| \sum_{i=1}^n \sigma_i h_{j}(X_i, Z_{u(j)}(X_i)'v_{u(j)}) \right|  \right) \right] \\
 &\qquad \leq m\max_{j\in[m]}\Ep\left[ \exp\left( \frac{2 \psi}{\sqrt n}\sup_{\eta\in \Delta} \left| \sum_{i=1}^n \sigma_i L_j(X_i)Z_{u(j)}(X_i)'v_{u(j)} \right|  \right) \right] \\
&   \qquad =   m\max_{j\in[m]} \Ep\left[ \exp\left( 2\psi\sup_{\eta\in \Delta} \left| \Gn(\sigma L_j(X)Z_{u(j)}(X)'v_{u(j)}) \right|  \right) \right]
\end{align*}
by Lemma \ref{lem:contraction} in this appendix with
$$
h_{i j}(t) = h_j(X_i,t), \ \gamma_{i j} = L_j(X_i),  \ G(t) = \exp(\psi t / \sqrt n), \quad i\in[n],\ j\in[m], \ t\in\mathbb R
$$
and
$$
T_j = \{ t = (t_1,\dots,t_n)' \in \mathbb{R}^n : t_{i} = Z_{u(j)}(X_i)'v_{u(j)} \ \mbox{for all $i\in[n]$ and some $\eta\in\Delta$}\},\  j\in[m].
$$
Further, for all $j\in[m]$,
\begin{align*}
& \Ep\left[ \exp\left( 2\psi\sup_{\eta\in \Delta} \left| \Gn(\sigma L_j(X)Z_{u(j)}(X)'v_{u(j)}) \right|  \right) \right] \\
&  \qquad  \leq \Ep\left[ \exp\left( 2\psi \max_{k\in[p_{u(j)}]}|\Gn(\sigma L_j(X) Z_{u(j)k}(X))| \|\Delta\|_{1,v}\right)\right]\\
&  \qquad  \leq p \max_{k\in[p_{u(j)}]}\Ep\Big[ \exp\left( 2\psi |\Gn(\sigma L_j(X) Z_{u(j)k}(X))| \|\Delta\|_{1,v}\right)\Big]
\end{align*}
by H\"{o}lder's inequality and the observation that $\exp(t) >0$ for all $t\in\mathbb R$. Moreover, for all $j\in[m]$ and $k\in[p_{u(j)}]$,
\begin{align*}
&  \Ep\left[ \exp\left( 2\psi |\Gn(\sigma L_j(X) Z_{u(j)k}(X))| \|\Delta\|_{1,v}\right)\right]\\
& \qquad \leq   2 \Ep\left[ \exp\left( 2\psi \Gn(\sigma L_j(X)Z_{u(j)k}(X)) \|\Delta\|_{1,v} \right)\right]\\
&  \qquad \leq  2 \exp\left( 2\psi^2 \En[L_j(X)^2Z_{u(j)k}(X)^2] \|\Delta\|_1^2\right) \leq 2\exp(2\psi^2 B_{2 n}^2\|\Delta\|_{1,v}^2),
\end{align*}
where the first inequality follows from observing  that $\Gn(\sigma L_j(X) Z_{u(j)k}(X))$'s are symmetrically distributed, the second from sub-Gaussianity of $\Gn(\sigma L_j(X) Z_{u(j) k}(X))$ (see the proof of Lemma 2.2.7 in \cite{VW96}, for example), and the third from the definition of $\Omega$ (recall that we implicitly condition on $\Omega$). Combining the inequalities above gives \eqref{eq: key inequality contraction} and completes the proof of the lemma.
 \qed

\begin{lemma}[Contraction Principle for Maxima]\label{lem:contraction}
For $i\in[n]$ and $j\in[m]$, let the functions $h_{ij}:\mathbb{R} \to \mathbb{R}$ be such that $$|h_{ij}(t)-h_{ij}(s)|\leq \gamma_{ij}|t-s|, \text{ for all } (t, s) \in \mathbb R^2 \ \ \mbox{and} \ \ h_{ij}(0)=0.$$
Also, let  $G:\mathbb{R}_+\to\mathbb{R}_+$ be a non-decreasing convex function and for each $j\in[m]$, let $T_j$ be a bounded set in $\mathbb R^n$. Then
$$
\Ep\left[ {G}\left( \sup_{j\in[m], t\in T_j} \left| \sum_{i=1}^n \sigma_i h_{ij}(t_i) \right|  \right) \right] \leq  m\max_{j\in[m]}\Ep\left[ {G}\left( 2\sup_{t\in T_j} \left| \sum_{i=1}^n \sigma_i \gamma_{ij}t_i \right|  \right) \right],
$$
where $(\sigma_{i})_{i=1}^n$ are i.i.d. Rademacher random variables, i.e. random variables taking values $-1$ and $+1$ with probability $1/2$ each.
\end{lemma}
\noindent
{\bf Proof of Lemma \ref{lem:contraction}.}
Since $G(x)\geq 0$ for all $x\in\mathbb R_{+}$, we have
\begin{align*}
\Ep\left[ {G}\left( \sup_{j\in[m], t\in T_j} \left| \sum_{i=1}^n \sigma_i h_{ij}(t_i) \right|  \right) \right]
&\leq \Ep\left[\sum_{j=1}^m {G}\left( \sup_{t\in T_j} \left| \sum_{i=1}^n \sigma_i h_{ij}(t_i) \right|  \right) \right]\\
&\leq m\max_{j\in[m]}\Ep\left[{G}\left( \sup_{t\in T_j} \left| \sum_{i=1}^n \sigma_i h_{ij}(t_i) \right|  \right) \right].
\end{align*}
Also, by Theorem 4.12 in \cite{LD91}, for each $j\in[m]$,
$$
\Ep\left[{G}\left( \sup_{t\in T_j} \left| \sum_{i=1}^n \sigma_i h_{ij}(t_i) \right|  \right) \right] \leq \Ep\left[ {G}\left( 2\sup_{t\in T_j} \left| \sum_{i=1}^n \sigma_i \gamma_{ij}t_i \right|  \right) \right].
$$
Combining these inequalities gives the asserted claim.\qed

\bibliographystyle{plain}

\end{document}